\documentclass{amsart}

\usepackage[english]{babel}
\usepackage{amsmath,amssymb,mathtools,enumitem}
\usepackage[T1]{fontenc}

\usepackage{titlesec}
\usepackage{amsthm}
\usepackage{cite}
\usepackage{amsrefs}

\usepackage{varioref}
\usepackage{hyperref}
\usepackage[capitalise]{cleveref}

\newcommand{\mc}[1]{\mathcal{#1}}
\newcommand{\mf}[1]{\mathfrak{#1}}

\newcommand{\ol}[1]{\overline{#1}}

\newcommand{\ti}[1]{\textit{#1}}
\newcommand{\tx}[1]{\textrm{#1}}

\newcommand{\C}{\mathbb{C}}
\newcommand{\N}{\mathbb{N}}

\newcommand{\R}{\mathbb{R}}
\newcommand{\T}{\mathbb{T}}
\newcommand{\Z}{\mathbb{Z}}

\newcommand{\dd}{d}
\newcommand{\dds}{\frac{d}{ds}}
\newcommand{\ddt}{\frac{d}{dt}}

\newcommand{\del}{\delta}
\newcommand{\deta}{\,d\eta}

\newcommand{\ds}{\,ds}

\newcommand{\dt}{\,dt}
\newcommand{\dx}{\,dx}
\newcommand{\dxi}{\,d\xi}
\newcommand{\dy}{\,dy}

\newcommand{\eps}{\epsilon}

\newcommand{\bigo}{O}
\newcommand{\cn}{\operatorname{cn}}
\newcommand{\ft}{\mc{F}}
\newcommand{\hktilde}{\widetilde{H}_{\kappa}}
\newcommand{\hvktilde}{\widetilde{H}_{\varkappa}}
\newcommand{\hkdv}{H_{\tx{KdV}}}
\newcommand{\htkdv}{\widetilde{H}_{\tx{KdV}}}
\newcommand{\I}{\mf{I}}
\newcommand{\norm}[1]{\left\lVert#1\right\rVert}
\newcommand{\op}{\tx{op}}

\newcommand{\snorm}[1]{\lVert#1\rVert}
\newcommand{\tr}{\operatorname{tr}}
\newcommand{\wh}[1]{\widehat{#1}}
\newcommand{\wt}[1]{\widetilde{#1}}

\theoremstyle{plain}\newtheorem{thm}{Theorem}[section]
\theoremstyle{plain}\newtheorem{cor}[thm]{Corollary}
\theoremstyle{plain}\newtheorem{prop}[thm]{Proposition}
\theoremstyle{plain}\newtheorem{lem}[thm]{Lemma}
\theoremstyle{definition}
\theoremstyle{definition}\newtheorem{dfn}[thm]{Definition}
\theoremstyle{remark}
\theoremstyle{definition}\newtheorem*{ack}{Acknowledgments}

\crefname{sec}{Section}{Sections}
\crefname{dfn}{Definition}{Definitions}
\crefname{hyp}{Hypothesis}{Hypotheses}
\crefname{lem}{Lemma}{Lemmas}
\crefname{prop}{Proposition}{Propositions}
\crefname{thm}{Theorem}{Theorems}
\crefname{cor}{Corollary}{Corollaries}

\titleformat{\section}[block]{\normalfont\normalsize\scshape\filcenter}{\thesection.}{0.5em}{}
\titleformat{\subsection}[runin]{\normalfont}{\thesubsection.}{0.5em}{\bfseries}

\numberwithin{equation}{section}

\allowdisplaybreaks

\begin{document}

\title[KdV with exotic spatial asymptotics]{Global well-posedness for $H^{-1}(\mathbb{R})$ perturbations of {K}d{V} with exotic spatial asymptotics}

\author{Thierry Laurens}
\address{Thierry Laurens \\
	Department of Mathematics\\
	University of California, Los Angeles, CA 90095, USA}
\email{laurenst@math.ucla.edu}

\begin{abstract}
	Given a suitable solution $V(t,x)$ to the Korteweg--de Vries equation on the real line, we prove global well-posedness for initial data $u(0,x) \in V(0,x) + H^{-1}(\mathbb{R})$.
	
	Our conditions on $V$ do include regularity but do not impose any assumptions on spatial asymptotics.  We show that periodic profiles $V(0,x)\in H^5(\mathbb{R}/\mathbb{Z})$ satisfy our hypotheses.  In particular, we can treat localized perturbations of the much-studied periodic traveling wave solutions (cnoidal waves) of KdV.  In the companion paper~\cite{Laurens2021} we show that smooth step-like initial data also satisfy our hypotheses.  
	
	We employ the method of commuting flows introduced in~\cite{Killip2019} where $V\equiv 0$.  In that setting, it is known that $H^{-1}(\R)$ is sharp in the class of $H^s(\R)$ spaces.
\end{abstract}

\maketitle


\section{Introduction}

The Korteweg--de Vries (KdV) equation
\begin{equation}
	\ddt u = - u''' + 6uu'
	\label{eq:kdv}
\end{equation}
(where primes $u' = \partial_xu$ denote spatial differentiation) was proposed in~\cite{Korteweg1895} to describe the phenomena of solitary traveling waves (solitons) in shallow channels.  Since its introduction over a century ago, the KdV equation has been intensively studied on the line $\R$ and the circle $\R/\Z$ and been shown to exhibit numerous special features.

A fundamental line of investigation for KdV has been well-posedness in the $L^2$-based Sobolev spaces $H^s(\R)$ and $H^s(\R/\Z)$.  The derivative in the nonlinearity of KdV prevents straightforward contraction mapping arguments from closing, so preliminary results produced continuous dependence in a weaker norm than the space of initial data.  One of the first results to overcome this loss of derivatives phenomenon was obtained by Bona and Smith~\cite{Bona1975} who proved global well-posedness in $H^3(\R)$.  In the following decades, an extensive list of methods has been developed in the effort to lower the regularity $s$; see for example~\cites{Bona1976,Kato1975,Saut1976,Temam1969,Tsutsumi1971,Kenig1991,Bourgain1993,Christ2003,Kenig1996,Colliander2003,Guo2009,Kishimoto2009}.  Recently, a new low-regularity method was introduced in~\cite{Killip2019} that yields global well-posedness in $H^{-1}(\R)$ and $H^{-1}(\R/\Z)$, a result that is sharp in both topologies.  In the $\R/\Z$ case this result was already known~\cite{Kappeler2006}.

Solutions in $H^s(\R/\Z)$ spaces are spatially periodic and solutions in $H^s(\R)$ spaces decay at infinity.  However, there are other classes of initial data which are of physical interest.  Waveforms that are step-like---in the sense that $u(0,x)$ approaches different constant values as $x\to \pm\infty$---arise in the study of bore propagation (cf. \cites{Benjamin1954,Peregrine1966,Whitham1999,Gurevich1973,Caputo2003,Fornberg1978}) and rarefaction waves (cf. \cites{Fornberg1978,Zaharov1980,Novokshenov2003,Leach2008,Andreiev2016}).

Asymptotically periodic functions are another important class of initial data.  This includes localized perturbations of a single periodic profile, wave dislocation where the periods as $x\to\pm\infty$ may not align, and waves with altogether different periodic asymptotics as $x\to\pm\infty$.

Quasi-periodic spatial asymptotics are also heavily studied in the literature (see for example \cites{Binder2018,Damanik2016,Eichinger2019,Deift2008,Deift2017}).  As we will discuss more thoroughly below, these classes are excluded by traditional analysis on the circle $\R/\Z$.

Our objective in this paper and in~\cite{Laurens2021} is to extend low-regularity methods for well-posedness to the regime of exotic spatial asymptotics.  Specifically, we employ the method of commuting flows that was introduced in~\cite{Killip2019} and used there to prove well-posedness in $H^{-1}(\R)$ and $H^{-1}(\R/\Z)$.  Developments of the method of commuting flows have been used to prove both symplectic non-squeezing~\cite{Ntekoume2019} and invariance of white noise~\cite{Killip2020} for KdV on the line.  The method of commuting flows has also been adapted to other completely integrable systems, including the cubic NLS and mKdV equations~\cite{HarropGriffiths2020}, the fifth-order KdV equation~\cite{Bringmann2019}, and the derivative NLS equation~\cite{Killip2021}.  However, aside from the white noise result~\cite{Killip2020}, the method of commuting flows has not yet been applied to nontrivial spatial asymptotics.

The method of commuting flows relies on the existence of a generating function $\alpha(\kappa,u)$ for the KdV hierarchy of conserved quantities with the asymptotic expansion 
\begin{equation}
	\alpha(\kappa , u) = \frac{1}{4 \kappa^{3}} P(u) - \frac{1}{16 \kappa^{5}} \hkdv(u) + \bigo(\kappa^{-7})
	\label{eq:alpha intro}
\end{equation}
for Schwartz $u$.  Here, $P$ and $\hkdv$ denote the momentum and KdV energy functionals
\begin{equation}
	P(u) := \tfrac{1}{2} \int u(x)^2\dx , \qquad
	\hkdv(u) := \int \big( \tfrac{1}{2} u'(x)^2 + u(x)^3 \big) \dx .
	\label{eq:momentum hkdv}
\end{equation}
Rearranging the expansion~\eqref{eq:alpha intro}, we might expect that the dynamics of the Hamiltonians
\begin{equation}
	H_\kappa(u) := -16\kappa^5 \alpha(\kappa,u) + 4\kappa^2 P(u)
	\label{eq:hk intro}
\end{equation}
approximate that of KdV as $\kappa\to\infty$.  In~\cite{Killip2019} the authors demonstrated that the flow induced by the Hamiltonian~\eqref{eq:hk intro} is well-posed in $H^{-1}$ and converges to that of KdV in $H^{-1}$ as $\kappa\to\infty$.

Given a solution $V(t,x)$ to KdV we define
\begin{equation*}
V_\kappa(t,x) \tx{ to be the solution to the }H_\kappa\tx{ flow with initial data }V(0,x).
\end{equation*}
Both here and in~\cite{Laurens2021} we show that the well-posedness of the $H_\kappa$ flows follows from standard PDE techniques.  We contend that other integrable methods (such as the inverse scattering transform) could also be employed to study these integrable $H_\kappa$ flows.

Throughout this paper we assume that the background wave $V$ is sufficiently regular in the following sense.
\begin{dfn}
	\label{thm:hyp}
	We call the background wave $V(t,x) : \R\times\R \to \R$ \emph{admissible} if for every $T>0$ it satisfies the following:
	\begin{enumerate}[label=(\roman*)]
		\item $V$ solves KdV~\eqref{eq:kdv} and is bounded in $W^{2,\infty}(\R_x)$ uniformly for $|t|\leq T$,
		\item The $H_\kappa$ flows $V_\kappa$ are bounded in $W^{4,\infty}(\R_x)$ uniformly for $|t|\leq T$ and $\kappa>0$ sufficiently large,
		\item $V_\kappa - V \to 0$ in $W^{2,\infty}(\R_x)$ as $\kappa\to\infty$ uniformly for $|t|\leq T$ and initial data in the set $\{V_\varkappa(t) : |t|\leq T,\ \varkappa\geq\kappa\}$.
	\end{enumerate}
\end{dfn}
 
In this paper we will show that for admissible waves $V$ the KdV equation~\eqref{eq:kdv} is well-posed for $H^{-1}(\R)$ perturbations of $V$ (cf.~\cref{thm:wellposed 2}):
\begin{thm}[Global well-posedness]
	\label{thm:intro gwp}
	Given $V$ admissible in the sense of \cref{thm:hyp}, the KdV equation~\eqref{eq:kdv} with initial data $u(0) \in V(0) + H^{-1}(\R)$ is globally well-posed in the following sense: $u(t) = V(t) + q(t)$ and $q(t)$ is given by a jointly continuous data-to-solution map $\R_t \times H^{-1}(\R) \to H^{-1}(\R)$ for the equation
	\begin{equation}
	\ddt q = - q''' + 6qq' + 6(Vq)'
	\label{eq:tkdv}
	\end{equation}
	with initial data $q(0) = u(0) - V(0)$.
\end{thm}

Using the techniques introduced in~\cite{Killip2019} (cf. Cor.~5.3) one can also obtain well-posedness in $H^s(\R)$ for $s>-1$.  The key ingredient in the argument is $H^{-1}$-equicontinuity, which will be further discussed below.

As we cannot make sense of the nonlinearity of KdV for $H^{-1}(\R)$ solutions (even in the distributional sense), the solutions in \cref{thm:intro gwp} will be constructed as limits of the $H_\kappa$ flows as $\kappa\to \infty$.  This is the right notion of solution because it coincides with the classical notion on a dense subset of initial data and \cref{thm:intro gwp} guarantees continuous dependence of the solution upon the initial data.  To provide further evidence in this direction, in \cref{sec:classical} we show that for initial data $u(0) \in V(0) + H^3(\R)$ our solution $u(t)$ solves KdV and is unique:
\begin{thm}
	\label{thm:intro classical solns}
	Fix $V$ admissible and $T>0$.   Given initial data $u(0) \in V(0) + H^{3}(\R)$, the solution $u(t)$ constructed in \cref{thm:intro gwp} lies in $V(t) + (C_tH^2 \cap C^1_tH^{-1})$ $([-T,T]\times\R)$ for all $T>0$, solves KdV~\eqref{eq:kdv}, and is unique in this class.
\end{thm}

There are rich classes of initial data that are admissible according to our definition.  In this paper we cover the important case of smooth periodic initial data $V(0,x)$ (cf.~\cref{thm:periodic gwp}):
\begin{cor}[Periodic background]
	\label{thm:intro periodic}
	Given $V(0) \in H^5(\R/\Z)$, the KdV equation~\eqref{eq:kdv} is globally well-posed for $u(0) \in V(0) + H^{-1}(\R)$ in the sense of \cref{thm:intro gwp}.
\end{cor}

In particular, this includes the periodic traveling wave solutions (cnoidal waves) of KdV.  Indeed, among all periodic asymptotics appearing in the literature, the most common choice for the background wave $V$ are the periodic traveling wave solutions to KdV.  These are the spatially periodic analogues of solitons, and they can be expressed in terms of the Jacobian elliptic cosine function $\cn(z;k)$ as
\begin{equation}
	V(t,x) = \eta - h \cn^2\left[ \sqrt{\tfrac{h}{2k^2}} (x-ct) ; k \right] .
	\label{eq:cnoidal 1}
\end{equation}
Here, $k\in [0,1)$ is the elliptic modulus, $h>0$ is the wave height, and the trough level $\eta$ and wave speed $c$ are determined by
\begin{equation*}
	\eta = \tfrac{h}{k^2} \left[ \tfrac{E(k)}{K(k)} - 1 + k^2 \right] , \qquad
	c = \tfrac{2h}{k^2} \left[ 2 - k^2 - \tfrac{3E(k)}{K(k)} \right]
\end{equation*}
where $K(k)$ and $E(k)$ are the complete elliptic integrals of the first and second kind.  In view of the representation~\eqref{eq:cnoidal 1}, Korteweg and de Vries~\cite{Korteweg1895} dubbed these solutions \ti{cnoidal waves}.

The explicit solutions~\eqref{eq:cnoidal 1} have been used to justify many empirical observations of water waves.  In the pioneering paper~\cite{McKean1977}, McKean proved orbital (nonlinear) stability of cnoidal waves with respect to co-periodic $C^k(\R/L\Z)$ perturbations using energy arguments, where $L$ denotes the period of the underlying cnoidal wave.  Orbital stability in $H^1(\R/L\Z)$ was later proven using variational methods~\cite{Pava2006}, and (in)stability results in $H^1(\R/L\Z)$ have since been obtained for generalized KdV, fractional KdV, and various other families of nonlinear dispersive equations; see for example~\cites{Pava2008,Neves2009,Arruda2009,Chen2013,Natali2014,Hur2015,Kapitula2015,BenzoniGavage2016,Alves2019,Andrade2019}.   Orbital stability has also been demonstrated at the lower regularity $L^2(\R/L\Z)$~\cites{Johnson2009,Bronski2011,Erdogan2012} and for $L^2(\R/nL\Z)$ with $n>1$~\cite{Nivala2010}.  All of the orbital stability results we know of pertain to spatially periodic perturbations.

The stability of cnoidal waves with respect to localized perturbations has only been studied in the context of spectral (linear) stability, where the linearized equation is considered as an operator on $L^2(\R)$ or $L^2(\R/\Z)$ perturbations.  Spectral stability for $L^2(\R)$ perturbations of cnoidal waves was established in~\cite{Bottman2009}, and spectral and modulational (in)stability for KdV-like equations have been explored in~\cites{Haragus2008,Bronski2010,Johnson2010,Nivala2010,Johnson2013,AnguloPava2016,BenzoniGavage2016,Jin2019}.  Evidently, a full nonlinear (in)stability theory must be predicated on one's ability to solve the equation for such initial data.  Indeed, this is one of the inspirations of the work undertaken herein.

More generally, existence for the Cauchy problem with periodic spatial asymptotics was first addressed in the physics literature for the case of cnoidal waves~\cite{Kuznetsov1974}, and again for more general periodic backgrounds in~\cites{Ermakova1982,Ermakova1982a,Firsova1988} via the inverse scattering transform.  A complete mathematical treatment of the Cauchy problem for highly regular initial data with distinct periodic asymptotics as $x\to\pm\infty$ was later given in~\cites{Egorova2009,Egorova2011}.

In the companion paper~\cite{Laurens2021} we will also apply \cref{thm:intro gwp} to step-like initial data.  This requires showing that smooth step-like $V$ are admissible; while such $V$ are highly studied, their $H_\kappa$ flows $V_\kappa$ are not.  We believe that \cref{thm:intro gwp} can also be applied to classes of quasi-periodic initial data, or any other class amenable to complete integrability methods.  The known results on exotic backgrounds use integrable methods like the inverse scattering transform, which are well suited to treat the $H_\kappa$ flows and verify the admissibility criteria.

In the case of a highly regular step-like background, existence for the Cauchy problem has been examined in~\cites{Buslaev1962,Cohen1984,Kappeler1986,Cohen1987,Egorova2009,Egorova2011} using the inverse scattering transform.  This process has been adapted to classes of one-sided step-like initial data~\cites{Rybkin2011,Grudsky2014,Rybkin2018} and even to one-sided step-like elements of $H^{-1}_{\tx{loc}}(\R)$~\cite{Grudsky2015}.  Despite the lack of assumptions at $-\infty$ (the direction in which radiation propagates), these low-regularity arguments require rapid decay at $+\infty$ and global boundedness from below, while our analysis is symmetric in $\pm x$ and in $\pm u$.  A great deal of attention in the literature has been given to the long-time behavior of such solutions; see for example~\cites{Hruslov1976,Kotlyarov1986,Kotlyarov1986a,Bikbaev1989,Bikbaev1989a,Bikbaev1989b,Khruslov1994,Khruslov1998,Baranetskiui2001,Novokshenov2003,Egorova2013,Andreiev2016}.  The asymptotics are spatially asymmetric and differ between the cases of tidal bores and rarefaction waves.

The broader question of well-posedness for perturbations of a fixed background wave has also been studied.  This was first done in the context of step-like backgrounds by the authors of~\cite{Iorio1998}, who proved local well-posedness for perturbations in $H^s(\R)$, $s>\tfrac{3}{2}$ and global well-posedness for $s\geq 2$.  Local well-posedness was later extended to $s>1$ in~\cite{Gallo2005} for the same class of step-like backgrounds.  Independently, local well-posedness for perturbations in $H^2(\R)$ was proved for gKdV in~\cite{Zhidkov2001}, along with global-in-time existence in the case of a kink solution background wave and initial data that is small in $H^{1}(\R)$.

Subsequent to our work, a new result~\cite{Palacios2021} for gKdV achieves local well-posedness for perturbations in $H^s(\R)$, $s>\tfrac{1}{2}$ and global well-posedness for $s\geq 1$.  In addition to a larger family of equations, this work also applies to a wide variety of background waves, including both step-like and periodic asymptotics.  In particular, the background wave is not assumed to be time-independent nor an exact solution, but rather is allowed to solve the equation modulo a localized error term.

In all cases, the presence of the background wave $V$ breaks the macroscopic conservation laws of KdV.  If $q$ is a regular solution to~\eqref{eq:tkdv} then the momentum functional~\eqref{eq:momentum hkdv} evolves according to
\begin{equation*}
\ddt \int q(t,x)^2 \dx = 6 \int V'(t,x) q(t,x)^2\dx .
\end{equation*} 
Interpreting $V$ as a potential-like function that affects the change in momentum, we will refer to~\eqref{eq:tkdv} as \emph{KdV with potential}.  In particular, for increasing step-like initial data $V(0,x)$ the term $V'$ has a sign and there is no cancellation in the integral; the resulting growth in the momentum is manifested in a dispersive shock that develops in the long-time asymptotics~\cite{Egorova2013}*{Fig.~1}.  Despite this lack of conservation, we are able to adapt the method of commuting flows to KdV with potential~\eqref{eq:tkdv} because these conserved quantities do not blow up in finite time.

In order to outline our argument, we will first introduce some notation.  The KdV equation~\eqref{eq:kdv} is governed by the Hamiltonian functional~\eqref{eq:momentum hkdv} via the Poisson structure
\begin{equation*}
\{ F,G \} = \int \frac{\del F}{\del q}(x) \bigg( \frac{\del G}{\del q} \bigg)'(x) \dx .
\end{equation*}
Here we are using the notation
\begin{equation*}
dF|_q(f)
= \frac{d}{ds} \bigg|_{s=0} F(q+sf)
= \int \frac{\del F}{\del q}(x) f(x) \dx
\end{equation*}
for the derivative of the functional $F(q)$.  This Poisson structure is the bracket associated to the almost complex structure $J := \partial_x$ and the $L^2$ pairing.  The momentum functional~\eqref{eq:momentum hkdv} generates translations under this structure, and~\eqref{eq:tkdv} is the evolution induced by the time-dependent Hamiltonian
\begin{equation*}
	\htkdv(q) := \int \big( \tfrac{1}{2} q'(x)^2 + q(x)^3 + 3V(t,x)q(x)^2 \big) \dx .
\end{equation*}

Our analysis will not rely upon this Hamiltonian structure, but we will borrow the convenient notations
\begin{equation*}
q(t) = e^{t J \nabla H} q(0) \quad\tx{for the solution to}\quad \frac{d q}{d t} = \partial_{x} \frac{\delta H}{\delta q} ,
\end{equation*}
and
\begin{equation*}
\ddt F(q(t)) = \{ F(q),H(q) \}
\quad\tx{for the quantity }F(q)\tx{ with }q(t) = e^{t J \nabla H} q(0) .
\end{equation*}
For example, given a background solution $V(t,x)$ to KdV,
\begin{equation}
V_\kappa(t,x) := e^{tJ\nabla H_\kappa} V(0,x)
\label{eq:vk dfn}
\end{equation}
denotes the solution to the $H_\kappa$ flow with initial data $V(0,x)$.

Just as the Hamiltonian $H_\kappa$ approximates $\hkdv$, we define $\hktilde$ to approximate $\htkdv$; subtracting the background $V$ from $u$ we obtain the $\hktilde$ flow of $q = u - V$ from time $0$ to $t$:
\begin{equation*}
\wt{\Phi}_{\kappa}(t) q = e^{tJ\nabla H_\kappa} (q+V(0)) - V_\kappa(t) .
\end{equation*}
We will show that for admissible $V$ the $\hktilde$ flow is well-posed in $H^{-1}$ and converges to the KdV equation with potential~\eqref{eq:tkdv} in $H^{-1}$ uniformly on bounded time intervals as $\kappa\to\infty$.  As in~\cite{Killip2019}, one asset of this method is that well-posedness of the $\hktilde$ flow in $H^{-1}(\R)$ follows from an ODE argument because $\alpha$ is real-analytic on $V + H^{-1}(\R)$ and $P$ generates translations.

The major structural difference of our argument from that in~\cite{Killip2019} is that we cannot assume the existence of regular solutions to~\eqref{eq:tkdv}.  Although some results in this direction do exist (e.g.~\cites{Egorova2009,Egorova2011}), we would need to significantly increase our assumptions on the background wave $V$ in order to employ them.  Instead of showing that the $\hktilde$ flows converge to that of KdV with potential as $\kappa\to\infty$, we show that the $\hktilde$ flows are Cauchy and we define the limit to be an $H^{-1}$ solution of~\eqref{eq:tkdv}.  To verify that this is indeed the solution map, we show that it is a jointly continuous map from $\R_t \times H^{-1}(\R)$ to $H^{-1}(\R)$ and agrees with the classical notion of solution on a dense subset of initial data (cf. \cref{thm:intro classical solns}).    The proof of \cref{thm:intro classical solns} relies on an energy argument in $H^3(\R)$ (similar to that of Bona--Smith~\cite{Bona1975}) using the fact that the $H_\kappa$ flows also conserve the polynomial conservation laws of KdV (as is suggested by the asymptotic expansion~\eqref{eq:alpha intro} and Poisson commutativity).

Following the argument of~\cite{Killip2019}, the Cauchyness of the $\hktilde$ flows in $H^{-1}(\R)$ as $\kappa\to\infty$ is implied by convergence at some lower $H^s(\R)$ regularity together with equicontinuity in $H^{-1}(\R)$ for all $\kappa$ large.  In the $V\equiv 0$ case~\cite{Killip2019}, equicontinuity quickly followed from the relation~\eqref{eq:alpha 3} and the conservation of the functional $\alpha(\varkappa,q)$ under both the $H_\kappa$ and KdV flows.  However, the appearance of the background wave $V$ in the $\hktilde$ flow breaks the conservation of $\alpha(\varkappa,q)$.  Instead of obtaining full equicontinuity, we introduce a dependence between the energy parameters $\kappa$ and $\varkappa$ in the proof of Cauchyness (cf. \cref{thm:diff flow conv 2}) that we can match when we estimate the growth of $\alpha(\varkappa,q)$ (cf. \cref{thm:hktilde alphadot 1}).

This paper is organized as follows.  In \cref{sec:prelim} we introduce the diagonal Green's function $g$ for perturbations $q\in H^{-1}(\R)$ of the background wave $V$, the logarithm $\alpha(\kappa,q)$ of the renormalized perturbation determinant, and the $\hktilde$ flow~\eqref{eq:hktilde flow}.  In \cref{sec:hktilde} we obtain an \ti{a priori} estimate for the growth of $\alpha(\varkappa,q)$ under the dynamics of $\hktilde$ flow (\cref{thm:hktilde alphadot 1}) with enough independence of the energy parameters $\kappa$ and $\varkappa$ to aid in the proof of equicontinuity.  In \cref{thm:hktilde alphadot 2} we then prove that the $\hktilde$ flow is well-posed in $H^{-1}(\R)$.

The entirety of \cref{sec:conv 1} is dedicated to demonstrating that the $\hktilde$ flows converge in $H^s(\R)$ for some $s<-1$ (\cref{thm:diff flow conv 1}).  In \cref{sec:conv 2} we upgrade this convergence to $H^{-1}(\R)$ (\cref{thm:diff flow conv 2}) and then conclude our main result \cref{thm:intro gwp}.  In \cref{sec:classical} we prove \cref{thm:intro classical solns}.

Lastly, we proceed in \cref{sec:cnoidal} with an application to cnoidal waves~\eqref{eq:cnoidal 1} which we present separately because of the significantly shorter length.  This is subsumed by the analysis in \cref{sec:periodic}, where we consider more general smooth periodic backgrounds $V(0) \in H^5(\R/\Z)$.

\begin{ack}
	I was supported in part by NSF grants DMS-1856755 and DMS-1763074.  I would like to thank my advisors, Rowan Killip and Monica Vi\c{s}an, for their guidance.  I would also like to thank Friedrich Klaus and associate editor Kenji Nakanishi, whose comments led to the revision of the proof of \cref{thm:existence}.
\end{ack}

\section{Diagonal Green's function}
\label{sec:prelim}

We begin by reviewing our notation and the necessary tools from~\cite{Killip2019}, which can be consulted for further details.

For a Sobolev space $W^{k,p}(\R)$ we use the spacetime norm
\begin{equation*}
\norm{ q }_{C_tW^{k,p}(I\times\R)} := \sup_{t\in I} \norm{ q(t) }_{W^{k,p}(\R)}
\end{equation*}
for $I\subset\R$ an interval.  In addition to the usual Sobolev spaces $W^{k,p}$ and $H^s$, we define the norm
\begin{equation}
\norm{f}^2_{H^s_\kappa(\R)} := \int_\R (\xi^2+4\kappa^2)^s |\hat{f}(\xi)|^2\dxi .
\label{eq:Hskappa norm dfn}
\end{equation}
The presence of the factor of four is to make the calculation~\eqref{eq:hskappa identity 1} an exact identity.  Our convention for the Fourier transform is
\begin{equation*}
\hat{f}(\xi) = \frac{1}{\sqrt{2\pi}} \int_{\R} e^{-i\xi x}f(x)\dx ,\qquad
\snorm{\hat{f}}_{L^2} = \norm{f}_{L^2} .
\end{equation*}
In analogy with the usual $H^s$ spaces, we have the elementary facts
\begin{equation}
\norm{ w f }_{H_{\kappa}^{\pm 1}} \lesssim \norm{ w }_{W^{1,\infty}} \norm{f}_{H_{\kappa}^{\pm 1}} ,  \qquad
\norm{ w f }_{H_{\kappa}^{\pm 1}} \lesssim \norm{ w }_{H^{1}} \norm{ f }_{H_{\kappa}^{\pm 1}}
\label{eq:hskappa ests}
\end{equation}
uniformly for $\kappa \geq 1$.  We will exclusively use the $L^2$ pairing $\langle\cdot,\cdot\rangle$; the space $H_{\kappa}^{-1}$ is dual to $H_{\kappa}^{1}$ with respect to this pairing, and so the inequalities~\eqref{eq:hskappa ests} for $H_{\kappa}^{-1}$ are implied by those for $H_{\kappa}^{1}$.  

We write $\I_p$ for the Schatten classes (also called trace ideals) of compact operators on the Hilbert space $L^2(\R)$ whose singular values are $\ell^p$-summable.  Of particular importance will be the Hilbert--Schmidt class $\I_2$: recall that an operator $A$ on $L^2(\R)$ is Hilbert--Schmidt if and only if it admits an integral kernel $a(x,y) \in L^2(\R\times\R)$, and we have
\begin{equation*}
\norm{ A }_\op \leq \norm{ A}_{\I_2} = \left( \iint |a(x,y)|^2 \dx\dy \right)^{1/2} .
\end{equation*}
The product of two Hilbert--Schmidt operators $A$ and $B$ is of trace class $\I_1$, the trace is cyclic:
\begin{equation*}
\tr(AB) := \iint a(x,y)b(y,x) \dy\dx = \tr(BA) ,
\end{equation*}
and we have the estimate
\begin{equation*}
|\tr(AB)| \leq \norm{A}_{\I_2} \norm{B}_{\I_2} .
\end{equation*}
Additionally, Hilbert--Schmidt operators form a two-sided ideal in the algebra of bounded operators, due to the inequality
\begin{equation*}
\norm{ BAC }_{\I_2} \leq \norm{B}_\op \norm{A}_{\I_2} \norm{C}_\op .
\end{equation*}

We notate the resolvent of the Schr\"odinger operator with zero potential by
\begin{equation*}
R_0(\kappa) := \left( - \partial_x^2 + \kappa^2 \right)^{-1}
\quad\tx{with integral kernel}\quad
\langle \del_x, R_0(\kappa)\del_y \rangle = \tfrac{1}{2\kappa} e^{-\kappa|x-y|} .
\end{equation*}
The energy parameter $\kappa$ will always be real and positive.  Consequently, $R_0(\kappa)$ will always be positive definite and so we may consider its positive definite square-root $\sqrt{ R_0(\kappa) }$.

Given a Banach space $X(\R)$ of functions on $\R$, we call $Q\subset X$ \emph{equicontinuous} if
\begin{equation*}
\norm{ q( \cdot+h ) - q(\cdot) }_X \to 0 \quad \tx{as }h\to 0,\tx{ uniformly for }q\in Q.
\end{equation*}
Note that for the supremum norm this definition coincides with the equicontinuity criterion of the Arzel\`{a}--Ascoli theorem.  It is natural to call this property equicontinuity as previous authors have~\cites{Riesz1933,Frechet1937,Pego1985,Edwards1995}, because it appears in the Kolmogorov--Riesz compactness theorem for the case $X = L^p$~\cite{Brezis2011}*{Th.~4.26}.  In particular, it follows that a precompact subset of $H^{-1}(\R)$ is equicontinuous in $H^{-1}(\R)$.

In Fourier variables, equicontinuity requires that the tails of $\hat{q}(\xi)$ are small in $H^{-1}(\R)$ uniformly for $q\in Q$.  Specifically, we will use the following characterization:
\begin{lem}[Equicontinuity {\cite{Killip2019}*{\S4}}]
	\label{thm:hskappa equicty}
	A bounded subset $Q\subset H^{-1}(\R)$ is equicontinuous if and only if
	\begin{equation*}
	\lim_{\kappa\to\infty}\, \sup_{q\in Q}\, \int_{\R} \frac{|\hat{q}(\xi)|^2}{\xi^2+4\kappa^2} \dxi = 0 .
	\end{equation*}
\end{lem}

The following calculation is the basis for all of the analysis that follows.
\begin{lem}[Key estimate {\cite{Killip2019}*{Prop.~2.1}}]
	For $q\in H^{-1}(\R)$ we have
	\begin{equation}
	\norm{ \sqrt{ R_0(\kappa) }\, q \sqrt{ R_0(\kappa) } }_{\I_2}^2
	= \frac{1}{\kappa} \int \frac{|\hat{q}(\xi)|^2}{\xi^2+4\kappa^2}\dxi
	= \frac{1}{\kappa} \norm{ q }^2_{H^{-1}_\kappa} .
	\label{eq:hskappa identity 1}
	\end{equation}
\end{lem}

The identity~\eqref{eq:hskappa identity 1} guarantees that the Neumann series for the resolvent of $-\partial^2 + q$ converges for all $\kappa$ sufficiently large when $q$ belongs to a bounded subset of $H^{-1}_\kappa(\R)$.  Consequently, we will always be working within the closed balls
\begin{equation*}
B_A(\kappa) := \{ q\in H^{-1}(\R) : \norm{ q }_{H^{-1}_\kappa} \leq A \}, \qquad
B_A := \{ q\in H^{-1}(\R) : \norm{ q }_{H^{-1}} \leq A \}
\end{equation*}
of radius $A>0$.  Note that $B_A(\kappa) \supset B_A$ for $\kappa\geq 1$, and so any result obtained for $B_A(\kappa)$ with $\kappa\geq 1$ also holds for the fixed set $B_A$.  The resolvent construction also works for $q\in B_A(\kappa)$ perturbations of $V$:
\begin{lem}[Resolvents]
	\label{thm:resolvent}
	Fix $V\in L^\infty(\R)$.  Given $q\in H^{-1}(\R)$, there exists a unique self-adjoint operator corresponding to $-\partial^2_x + V + q$ with domain $H^1(\R)$.  Moreover, given $A>0$ there exists $\kappa_0>0$ so that the series
	\begin{equation}
	R(\kappa,V) = (-\partial^2+V+\kappa^2)^{-1} = \sum_{\ell=0}^\infty (-1)^\ell R_0(\kappa) [ V R_0(\kappa) ]^\ell
	\label{eq:resolvent series 1}
	\end{equation}
	converges absolutely to a positive definite operator for $\kappa\geq \kappa_0$, and the series
	\begin{equation}
	R(\kappa,V+q) = \sum_{\ell=0}^\infty (-1)^\ell \sqrt{ R(\kappa,V) } \big[ \sqrt{R(\kappa,V)} \, q \sqrt{R(\kappa,V)} \big]^\ell \sqrt{ R(\kappa,V) }
	\label{eq:resolvent series 2}
	\end{equation}
	converges absolutely for $q\in B_A(\kappa)$ and $\kappa \geq \kappa_0$.
\end{lem}
\begin{proof}
	Initially we require that $\kappa\geq 1$.  As $V\in L^\infty$, we may define the operator $-\partial^2 + V$ via the quadratic form
	\begin{equation*}
	\phi \mapsto \int \big( |\phi'(x)|^2 + V(x) |\phi(x)|^2 \big)\dx 
	\end{equation*}
	equipped with the domain $H^1(\R)$.  Using the elementary estimates $\norm{R_0}_\op \leq \kappa^{-2}$ and $\norm{V}_\op \leq \norm{V}_{L^\infty}$, it is clear that the Neumann series~\eqref{eq:resolvent series 1} for $R(\kappa,V)$ is absolutely convergent for all $\kappa^2 \geq 2\norm{V}_{L^\infty}$.  Once we know the series absolutely converges, it is straightforward to verify that multiplying by $-\partial^2 + V +\kappa^2$ produces the identity operator.
	
	Expanding the series~\eqref{eq:resolvent series 1} and using the identity~\eqref{eq:hskappa identity 1} we estimate
	\begin{align*}
	&\norm{ \sqrt{R(\kappa,V)} \, q \sqrt{R(\kappa,V)} }_{\I_2}^2
	= \tr\{ R(\kappa,V) q R(\kappa,V) \ol{q} \} \\
	&\leq \sum_{\ell,m = 0}^\infty \norm{ \sqrt{R_0}V\sqrt{R_0} }_{\op}^{\ell+m} \norm{ \sqrt{R_0}q\sqrt{R_0} }^2_{\I_2}
	\leq 4 \kappa^{-1} \norm{q}^2_{H^{-1}_\kappa}
	\end{align*}
	for all $\kappa^2 \geq 2\norm{V}_{L^\infty}$, and hence
	\begin{equation}
	\norm{ \sqrt{R(\kappa,V)} \, q \sqrt{R(\kappa,V)} }_{\I_2} \leq 2 \kappa^{-1/2} \norm{q}_{H^{-1}_\kappa} .
	\label{eq:hskappa identity 2}
	\end{equation}
	Consequently, given $q\in B_A(\kappa)$ we have
	\begin{align*}
	\int q(x) |\phi(x)|^2 \dx
	&\leq \norm{ \sqrt{R(\kappa,V)} \, q \sqrt{R(\kappa,V)} }_\op \int \big( |\phi'(x)|^2 + |V(x)| |\phi(x)|^2 \big) \dx \\
	&\leq \tfrac{1}{2} \int \big( |\phi'(x)|^2 + |V(x)| |\phi(x)|^2 \big) \dx
	\end{align*}
	for all $\phi\in H^1(\R)$ provided that $\kappa \geq 16A^2$.  We conclude that $-\partial^2 + V + q$ is a form-bounded perturbation of $-\partial^2 + V$ with relative norm strictly less than 1; this guarantees that $-\partial^2 + V + q$ exists, is unique, and has the same form domain $H^1(\R)$ (cf.~\cite{Reed1975}*{Th.~X.17}).  The estimate~\eqref{eq:hskappa identity 2} then demonstrates that the series~\eqref{eq:resolvent series 2} for $R(\kappa,V+q)$ is absolutely convergent for all $\kappa\geq 16A^2$.
\end{proof}

In~\cite{Killip2019} the diagonal Green's function---the restriction of the kernel $G(x,y;\kappa,q)$ of the operator $R(\kappa,q)$ to the diagonal---was instrumental in controlling $q$ in $H^{-1}$.  This construction also works for $q\in B_A(\kappa)$ perturbations of $V$:
\begin{prop}[Diagonal Green's function]
	\label{thm:diffeo prop}
	Fix $V\in L^\infty(\R)$.  Given $A>0$ there exists $\kappa_0>0$ such that for all $\kappa\geq\kappa_0$ the diagonal Green's function $g(x;\kappa,V+q) := G(x,x;\kappa,V+q)$ exists for $q\in B_A(\kappa)$, the two functionals
	\begin{equation}
	q \mapsto g(x;\kappa,V+q) - g(x;\kappa,V) \quad \text {and} \quad q \mapsto \frac{1}{g(x;\kappa,V+q)} - \frac{1}{g(x;\kappa,q)}
	\label{eq:diffeo prop 2}
	\end{equation}
	are real analytic from $B_{A}(\kappa)$ into $H_{\kappa}^{1}(\R)$, and we have the estimate
	\begin{equation}
	\norm{ g(x;\kappa,V+q) - g(x;\kappa,V) }_{H^1_\kappa}
	\lesssim \kappa^{-1} \norm{q}_{H^{-1}_\kappa}
	\label{eq:diffeo prop}
	\end{equation}
	uniformly for $q\in B_A(\kappa)$ and $\kappa\geq \kappa_0$.
\end{prop}
\begin{proof}
	In Fourier variables, we have
	\begin{equation*}
	\snorm{ \sqrt{R_0}\del_x }_{L^2}^2 \lesssim \kappa^{-1}, \qquad 
	\snorm{ \sqrt{R_0} \del_{x+h} - \sqrt{R_0} \del_x }_{L^2}^2
	\leq \int \frac{|e^{i\xi h}-1|^2}{\xi^2 + 1} \dxi 
	\lesssim |h|
	\end{equation*}
	for $\kappa\geq 1$.  This demonstrates that $x\mapsto \sqrt{R_0} \del_x$ is $\tfrac{1}{2}$-H\"older continuous as a map from $\R$ to $L^2$.  We initialize $\kappa_0$ to be the constant from~\cref{thm:resolvent}.  Then from the series \eqref{eq:resolvent series 1} we see that
	\begin{align*}
	&\left| \left\langle\del_x , \left[R(\kappa, V)-R_{0}(\kappa)\right] \del_y \right\rangle - \left\langle\del_{x'} , \left[R(\kappa, V)-R_{0}(\kappa)\right] \del_{y'} \right\rangle \right| \\
	&\lesssim \kappa^{-1/2} \big( |x-x'|^{1/2} + |y-y'|^{1/2} \big) \sum_{\ell =1}^{\infty} \left( \kappa^{-2} \norm{V}_{L^\infty} \right)^\ell .
	\end{align*}
	The series converges provided that $\kappa \gg \norm{V}_{L^\infty}^{1/2}$.  Consequently, the Green's function $G(x,y)=\langle \del_x, R(\kappa,V) \del_y \rangle$ is continuous in both $x$ and $y$, and so we may unambiguously define
	\begin{equation}
	g(x;\kappa,V)
	= \tfrac{1}{2\kappa} + \sum_{\ell=1}^\infty (-1)^\ell \big\langle \sqrt{R_0} \del_x, (\sqrt{R_0}V\sqrt{R_0})^\ell \sqrt{R_0} \del_x \big\rangle .
	\label{eq:g series 1}
	\end{equation}
	The zeroth-order term $\tfrac{1}{2\kappa}$ can be seen directly from the integral kernel for the free resolvent $R_0(\kappa)$.
	
	Similarly, from the series~\eqref{eq:resolvent series 2} and the estimate~\eqref{eq:hskappa identity 2} we have
	\begin{align*}
	&\left| \left\langle\del_x , \left[R(\kappa, V+q)-R(\kappa,V)\right] \del_y \right\rangle - \left\langle\del_{x'} , \left[R(\kappa, V+q)-R(\kappa,V)\right] \del_{y'} \right\rangle \right| \\
	&\lesssim \kappa^{-1/2} \big( |x-x'|^{1/2} + |y-y'|^{1/2} \big) \sum_{\ell =1}^{\infty} \big( 2 \kappa^{-1/2} A \big)^\ell
	\end{align*}
	for all $q\in B_A(\kappa)$.  The series converges provided that we also have $\kappa\gg A^2$.  Therefore $G(x,y;\kappa,V+q)$ is also a continuous function of $x$ and $y$ and so we may define
	\begin{equation*}
	g(x;\kappa,V+q) = g(x;\kappa,V) + \sum_{\ell=1}^\infty (-1)^\ell \big\langle \sqrt{R} \del_x, ( \sqrt{R}\, q \sqrt{R})^\ell \sqrt{R} \del_x \big\rangle
	\end{equation*}
	where $R = R(\kappa,V)$.  This shows that the first functional of~\eqref{eq:diffeo prop 2} is real analytic for $q\in B_A(\kappa)$.
	
	Next we check that $g(x;\kappa,V+q) - g(x;\kappa,V)$ is in $H^1_\kappa$ by duality and the operator estimate~\eqref{eq:hskappa identity 2}:
	\begin{align*}
	&\left| \int f(x) [ g(\kappa,V+q) - g(\kappa,V) ](x) \dx \right| \\
	&\leq \sum_{\ell=1}^\infty \norm{ \sqrt{R(\kappa,V)} f \sqrt{R(\kappa,V)} }_{\I_2} \norm{  \sqrt{R(\kappa,V)}\, q  \sqrt{R(\kappa,V)} }_{\I_2}^\ell  \\
	&\lesssim \kappa^{-1} \norm{f}_{H^{-1}_{\kappa}} \norm{q}_{H^{-1}_\kappa} .
	\end{align*}
	Taking a supremum over all $\norm{f}_{H^{-1}_{\kappa}} \leq 1$ we obtain the estimate~\eqref{eq:diffeo prop}.
	
	It remains to show that $g(x;\kappa,V+q)$ is nonvanishing so that the second functional of~\eqref{eq:diffeo prop 2} is also real analytic.  Using the series \eqref{eq:resolvent series 1} we estimate
	\begin{equation*}
	|g(x;\kappa,V) - g(x;\kappa,0)|
	\leq \snorm{ \sqrt{R_0} \del_x }_{L^2}^2 \sum_{\ell =1}^{\infty} \left( \kappa^{-2} \norm{V}_{L^\infty} \right)^\ell
	\lesssim \kappa^{-3}
	\end{equation*}
	for $\kappa \gg \norm{V}_{L^\infty}^{1/2}$.  As $g(x;\kappa,0) \equiv \tfrac{1}{2\kappa}$, we can take $\kappa_0$ larger if necessary to ensure
	\begin{equation*}
	\tfrac{1}{4\kappa} \leq g(x;\kappa,V) \leq \tfrac{3}{4\kappa}
	\end{equation*}
	for all $\kappa\geq \kappa_0$.  The estimate~\eqref{eq:diffeo prop} combined with the observation
	\begin{equation}
	\norm{f}_{L^\infty} 
	\leq \norm{f}^{1/2}_{L^2} \norm{f'}^{1/2}_{L^2}
	\lesssim \kappa^{-1/2} \norm{f}_{H^1_\kappa}
	\label{eq:h1kappa embedding}
	\end{equation}
	then guarantees that there exists $\kappa_0 \gg A^2$ so that
	\begin{equation*}
	\norm{ g(x;\kappa,V+q) - g(x;\kappa,V) }_{L^\infty} \leq \tfrac{1}{8\kappa} 
	\end{equation*}
	for all $q\in B_A(\kappa)$ and $\kappa\geq \kappa_0$.  Consequently, the second functional of~\eqref{eq:diffeo prop 2} is also real-analytic.
\end{proof}

Even for $V\not\equiv 0$, we will need to know that the functionals~\eqref{eq:diffeo prop 2} are diffeomorphisms in the case $V\equiv 0$.  This is because to demonstrate the convergence of flows which approximate KdV it is convenient to make the change of variables $1/g(k,q)$ in place of $q$, as was introduced in~\cite{Killip2019}.
\begin{prop}[Diffeomorphism property]
	\label{thm:diffeo prop 2}
	Given $A>0$ there exists $\kappa_0 >0$ so that the functionals
	\begin{equation*}
	q \mapsto g(x;\kappa,q) - \frac{1}{2\kappa} \quad \text {and} \quad q \mapsto \frac{1}{g(x;\kappa,q)} - 2\kappa
	\end{equation*}
	are real-analytic diffeomorphisms from $B_A$ into $H_{\kappa}^{1}(\R)$ for all $\kappa \geq \kappa_0$.
\end{prop}
\begin{proof}
	The proof follows that of \cite{Killip2019}*{Prop.~2.2}, but now we allow for arbitrary radii $A>0$ and compensate by taking $\kappa_0$ sufficiently large.  Using the resolvent identity we calculate the first functional derivative
	\begin{equation}
	dg|_q(f)
	= \dds \bigg|_{s=0} g(x;\kappa,q+sf)
	= - \langle \del_x , R(\kappa,q) f R(\kappa,q) \del_x \rangle .
	\label{eq:cleanup lem 10}
	\end{equation}
	For $q\equiv 0$, we use the integral kernel formula for $R_0$ to write
	\begin{equation*}
	dg|_0(f)
	= - \langle \del_x , R_0fR_0 \del_x \rangle
	= - \kappa^{-1} [R_0(2\kappa) f](x) .
	\end{equation*}
	Estimating by duality, expanding $R(\kappa,q)$ as the series~\eqref{eq:resolvent series 2}, and using the operator estimate~\eqref{eq:hskappa identity 1} we have
	\begin{equation*}
	\norm{ dg|_q(f) - dg|_0(f) }_{H^1_\kappa} \lesssim \kappa^{-3/2} A \norm{f}_{H^{-1}_\kappa}
	\end{equation*}
	uniformly for $q\in B_A$.  Taking a supremum over $\norm{f}_{H^{-1}_\kappa} \leq 1$, we conclude that there exists $\kappa_0 \gg A^2$ such that 
	\begin{equation*}
	\norm{ dg|_q - dg|_0 }_{H^{-1}_\kappa\to H^1_\kappa} \leq \tfrac{1}{2} \kappa^{-1} 
	\end{equation*}
	uniformly for $q\in B_A$ and $\kappa\geq\kappa_0$.  Using this and
	\begin{equation*}
	\big\lVert \big( dg|_0 \big)^{-1} \big\rVert_{H^{1}_\kappa\to H^{-1}_\kappa} \leq \kappa
	\end{equation*}
	as input, the standard contraction-mapping proof of the inverse function theorem guarantees that $q \mapsto g - \tfrac{1}{2\kappa}$ is a diffeomorphism from $B_A$ onto its image for all $\kappa\geq \kappa_0$.  
	
	For the second functional $q \mapsto \tfrac{1}{g} -2\kappa$ we write
	\begin{equation*}
	\tfrac{1}{g} - 2\kappa = -2\kappa \frac{ 2\kappa (g-\tfrac{1}{2\kappa}) }{ 1 + 2\kappa (g-\tfrac{1}{2\kappa}) } .
	\end{equation*}
	By the embedding $H^1 \hookrightarrow L^\infty$, we note that $f\mapsto \tfrac{f}{1+f}$ is a diffeomorphism from the ball of radius $\tfrac{1}{2}$ in $H^1$ into $H^1$.  The estimates~\eqref{eq:diffeo prop} and~\eqref{eq:h1kappa embedding} guarantee that we can pick $\kappa_0 \gg A^2$ so that
	\begin{equation*}
	2\kappa \norm{ g(x;\kappa,q) - \tfrac{1}{2\kappa} }_{L^\infty} \leq \tfrac{1}{2}
	\end{equation*}
	for all $q\in B_A$ and $\kappa\geq \kappa_0$.  Altogether, we conclude that $q \mapsto \tfrac{1}{g} -2\kappa$ is also a diffeomorphism from $B_A$ onto its image for all $\kappa\geq \kappa_0$.
\end{proof}

Ultimately, the convergence of the approximate flows to KdV with potential will be dominated by the linear and quadratic terms of the series~\eqref{eq:g series 1} for the diagonal Green's function.  Consequently, we will now record some useful operator identities for these two terms:
\begin{lem}
	For $\kappa\geq 1$ we have the operator identities
	\begin{align}
	&16\kappa^5 \langle \del_x, R_0 f R_0 \del_x \rangle = 16\kappa^4 R_0(2\kappa) f = \left[ 4\kappa^2 + \partial^2 + R_0(2\kappa) \partial^4 \right] f ,
	\label{eq:linear op id} \\
	&\begin{aligned} 16\kappa^5 \langle \del_x, R_0 f R_0 h R_0 \del_x \rangle &= 3fh - 3 [R_0(2\kappa)f''][R_0(2\kappa)h''] \\
	&\phantom{={}}+ 4\kappa^2 [R_0(2\kappa)f'][R_0(2\kappa)h'] ( -5 + R_0(2\kappa)\partial^2 ) \\
	&\phantom{={}}+ 4\kappa^2 [R_0(2\kappa)f][R_0(2\kappa)h] (5 \partial^2 + 2 R_0(2\kappa)\partial^4 ) , \end{aligned}
	\label{eq:quad op id}
	\end{align}
	where $R_0 = R_0(\kappa)$.
\end{lem}
\begin{proof}
	From the integral kernel formula for $R_0(\kappa)$ we see that $\langle\del_x,R_0f R_0\del_x \rangle = \kappa^{-1}R_0(2\kappa) f$, which demonstrates the first equality of~\eqref{eq:linear op id}.  The second equality follows from the symbol identity
	\begin{equation*}
	\frac{16 \kappa^{4}}{\xi^{2}+4\kappa^{2}} = 4 \kappa^{2} - \xi^{2} + \frac{\xi^{4}}{\xi^{2}+4 \kappa^{2}} 
	\end{equation*}
	in Fourier variables.
	
	Now we turn to the second identity~\eqref{eq:quad op id}.  In~\cite{Killip2019}*{Appendix} the Fourier transform of LHS\eqref{eq:quad op id} is found to be
	\begin{equation*}
	\ft\big( \tx{LHS}\eqref{eq:quad op id} \big)(\xi) = \frac{8\kappa^4}{\sqrt{2 \pi}} \int_{\R} \frac{ [\xi^{2}+(\xi-\eta)^{2}+\eta^{2}+24 \kappa^{2}] \hat{f}(\xi-\eta) \hat{h}(\eta) }{ (\xi^{2}+4 \kappa^{2}) ((\xi-\eta)^{2}+4 \kappa^{2}) (\eta^{2}+4 \kappa^{2}) } \deta .
	\end{equation*}
	(See the manipulation of the term~\eqref{eq:periodic L2 conv 3} for a similar calculation.)  The operator identity~\eqref{eq:quad op id} then follows from the equality
	\begin{align*}
	&\frac{8 \kappa^{4}\left[\xi^{2}+(\xi-\eta)^{2}+\eta^{2}+24 \kappa^{2}\right]}{(\xi^{2}+4 \kappa^{2})((\xi-\eta)^{2}+4 \kappa^{2})(\eta^{2}+4 \kappa^{2})} = 3 - \frac{3 \eta^{2}(\xi-\eta)^{2}}{((\xi-\eta)^{2}+4 \kappa^{2})(\eta^{2}+4 \kappa^{2})} \\
	&- \frac{20 \kappa^{2}\left[-\eta(\xi-\eta)+\xi^{2}\right]}{((\xi-\eta)^{2}+4 \kappa^{2})(\eta^{2}+4 \kappa^{2})} + \frac{4 \kappa^{2} \xi^{2}\left[\eta(\xi-\eta)+2 \xi^{2}\right]}{(\xi^{2}+4 \kappa^{2})((\xi-\eta)^{2}+4 \kappa^{2})(\eta^{2}+4 \kappa^{2})} . \qedhere \\
	\end{align*}
\end{proof}

As an offspring of the resolvent $R(\kappa,q)$, the diagonal Green's function comes with some algebraic identities.  In particular, in~\cite{Killip2019}*{Lem.~2.5--2.6} it is shown that for Schwartz $q$ we have the identities
\begin{equation}
\int \frac{G(x,y;\kappa,q) G(y,x;\kappa,q)}{2g(y;\kappa,q)^2}\dy = g(x;\kappa,q)
\label{eq:g alg prop 1}
\end{equation}
and
\begin{equation}
\begin{aligned}
&\int G(x,y;\kappa,q) \big[ {-}f''' + 2qf' + 2(qf)' + 4\kappa^2f' \big](y) G(y,x;\kappa,q)\dy \\
&= 2f'(x) g(x;\kappa,q) - 2f(x)g'(x;\kappa,q) 
\end{aligned}
\label{eq:g alg prop 2}
\end{equation}
for all Schwartz $f$.  To show that these hold for general $q\in B_A(\kappa)$, we argue as follows.  Given $A>0$, we pick $\kappa_0$ from \cref{thm:diffeo prop}.  Then both sides are analytic in $q$, and so equality follows from the proofs~\cite{Killip2019}*{Lem.~2.5--2.6} for the Schwartz case.

As is suggested by taking $f = g(\kappa,q)$ in~\eqref{eq:g alg prop 2}, multiplying by $1/2g(x;\kappa,q)^2$, and integrating in $x$, the diagonal Green's function satisfies the ODE
\begin{equation}
g'''(\kappa,q) = 2qg'(\kappa,q) + 2\left[ q g(\kappa,q) \right]' + 4\kappa^2g'(\kappa,q) ;
\label{eq:g alg prop 3}
\end{equation}
the proof from~\cite{Killip2019}*{Prop.~2.3} is purely algebraic and still applies.  In particular, we see that $qg'(\kappa,q)$ is a total derivative.
\begin{prop}
	Given $A>0$ there exists $F(x;\kappa,q)\in L^1(\R)$ and $\kappa_0 > 0$ so that
	\begin{equation}
	q(x)g'(x;\kappa,q) = F'(x;\kappa,q) , \qquad
	\norm{ F }_{L^1} \lesssim \kappa^{-1} \norm{q}^2_{H^{-1}_\kappa}
	\label{eq:cleanup lem 1}
	\end{equation}
	for all $q\in B_A(\kappa)$ and $\kappa\geq \kappa_0$.
\end{prop}
\begin{proof}
	From the ODE~\eqref{eq:g alg prop 3} we have
	\begin{equation*}
	qg'(\kappa,q) = \left[ \tfrac{1}{2} g''(\kappa,q) - qg(\kappa,q) - 2\kappa^2g(\kappa,q) \right]' .
	\end{equation*}
	In~\cite{Killip2020}*{Lem.~2.14} it is shown that the potential $q$ may be recovered from the diagonal Green's function via the relation
	\begin{equation*}
	q = \left[ \frac{g'(\kappa,q)}{2g(\kappa,q)} \right]' + \left[ \frac{g'(\kappa,q)}{2g(\kappa,q)} \right]^2 + \left[ \frac{1}{4g(\kappa,q)^2} - \kappa^2 \right] .
	\end{equation*}
	Rearranging this identity, we see that the claimed relation~\eqref{eq:cleanup lem 1} holds for the functional
	\begin{equation}
	F(\kappa,q) := \tfrac{1}{g(\kappa,q)} \big\{ \tfrac{1}{4}g'(\kappa,q)^2 - \kappa^2 \left[ g(\kappa,q) - \tfrac{1}{2\kappa} \right]^2 \big\} .
	\label{eq:cleanup lem 2}
	\end{equation}
	
	To see the quadratic dependence on $q$ claimed in the estimate~\eqref{eq:cleanup lem 1} we will Taylor expand $F$ about $q\equiv 0$.  From the series~\eqref{eq:g series 1} we note that $g(\kappa,0) \equiv \tfrac{1}{2\kappa}$, and so we have
	\begin{equation*}
	F(x;\kappa,0) \equiv 0 .
	\end{equation*}
	
	The Green's function for a translated potential is the translation of the original Green's function, and so
	\begin{equation}
	g(x;\kappa,q(\cdot+h)) = g(x+h;\kappa,q) \quad\tx{for all }h\in\R .
	\label{eq:g trans prop}
	\end{equation}
	This together with the resolvent identity yields
	\begin{equation}
	g'(x;\kappa,q) 
	= - \langle \del_x , R(\kappa,q) q' R(\kappa,q) \del_x \rangle .
	\label{eq:g derivative}
	\end{equation}
	Differentiating~\eqref{eq:cleanup lem 2} with respect to $q$ we obtain
	\begin{equation}
	\begin{aligned}
	dF|_q(f)
	= &- \tfrac{1}{g(\kappa,q)} F(\kappa,q)\, dg|_q(f) \\
	&+ \tfrac{1}{g(\kappa,q)} \big\{ \tfrac{1}{2} g'(\kappa,q)\, d(g')|_q(f) - 2\kappa^2 \big[ g(\kappa,q) - \tfrac{1}{2\kappa} \big]\,dg|_q(f) \big\} ,
	\end{aligned}
	\label{eq:cleanup lem 3}
	\end{equation}
	and since $F$, $g-\tfrac{1}{2\kappa}$, and $g'$ all vanish for $q\equiv 0$ we conclude
	\begin{equation*}
	dF|_0(f) \equiv 0 .
	\end{equation*}
	
	Next we turn to the Hessian of $F(q)$.  A straightforward computation shows that at $q\equiv 0$ we have
	\begin{equation}
	d^2F|_0(f,f)
	= - 4\kappa^3 \langle \del_x , R_0fR_0 \del_x \rangle^2 + \kappa \langle \del_x , R_0f'R_0 \del_x \rangle^2 ,
	\label{eq:cleanup lem 4}
	\end{equation}
	and we will estimate both terms on the RHS individually.  From the integral kernel formula for $R_0$ we write
	\begin{equation*}
	\langle \del_x , R_0fR_0 \del_x \rangle
	= \kappa^{-1} [R_0(2\kappa) f](x) .
	\end{equation*}
	Using Plancherel's theorem to estimate the first term of RHS\eqref{eq:cleanup lem 4} we have
	\begin{equation*}
	4\kappa^3 \int \left| \langle \del_x , R_0fR_0 \del_x \rangle \right|^2\dx
	= 4\kappa \int \frac{|\hat{f}(\xi)|^2}{(\xi^2+4\kappa^2)^2} \dxi
	\leq \frac{1}{\kappa} \int \frac{|\hat{f}(\xi)|^2}{\xi^2+4\kappa^2} \dxi .
	\end{equation*}
	Similarly, for the second term of RHS\eqref{eq:cleanup lem 4} we have
	\begin{equation*}
	\kappa \int \left| \langle \del_x , R_0f'R_0 \del_x \rangle \right|^2\dx
	= \frac{1}{\kappa} \int \frac{\xi^2 |\hat{f}(\xi)|^2}{(\xi^2+4\kappa^2)^2} \dxi
	\leq \frac{1}{\kappa} \int \frac{|\hat{f}(\xi)|^2}{\xi^2+4\kappa^2} \dxi .
	\end{equation*}
	Together we conclude
	\begin{equation}
	\norm{ d^2F|_0(f,f) }_{L^1} \lesssim \kappa^{-1} \norm{f}^2_{H^{-1}_\kappa} .
	\label{eq:cleanup lem 14}
	\end{equation}
	
	To finish the proof, it suffices to show that the Hessian's modulus of continuity satisfies
	\begin{equation}
	\norm{ d^2F|_q(f,f) - d^2F|_0(f,f) }_{L^1} \lesssim \kappa^{-3/2} A \norm{f}^2_{H^{-1}_\kappa}
	\label{eq:cleanup lem 5}
	\end{equation}
	uniformly for $q\in B_A(\kappa)$ and $\kappa$ large.  Indeed, the estimate~\eqref{eq:cleanup lem 1} then follows by choosing $\kappa_0 \gg A^2$ so that the RHS is smaller than RHS\eqref{eq:cleanup lem 14} for $\kappa\geq \kappa_0$.  Differentiating the first derivative~\eqref{eq:cleanup lem 3} we write
	\begin{align}
	&d^2F|_q(f,f) - d^2F|_0(f,f) \nonumber \\
	&= \tfrac{2}{g(\kappa,q)^2} F(\kappa,q) \big[ dg|_q(f) \big]^2 
	-\tfrac{1}{g(\kappa,q)} F(\kappa,q)\, d^2g|_q(f,f) 
	\label{eq:cleanup lem 6}\\
	&\phantom{={}} - \tfrac{2}{g(\kappa,q)^2} \big\{ \tfrac{1}{2} g'(\kappa,q)\, d(g')|_q(f) - 2\kappa^2 \big[ g(\kappa,q) - \tfrac{1}{2\kappa} \big]\,dg|_q(f) \big\}\, dg|_q(f) 
	\label{eq:cleanup lem 7}\\
	&\phantom{={}} + \tfrac{1}{g(\kappa,q)} \big\{ \tfrac{1}{2} g'(\kappa,q)\, d^2(g')|_q(f,f) - 2\kappa^2 \big[ g(\kappa,q) - \tfrac{1}{2\kappa} \big]\,d^2g|_q(f,f) \big\}
	\label{eq:cleanup lem 8}\\
	&\phantom{={}} + \tfrac{1}{g(\kappa,q)} \big\{  \tfrac{1}{2} [ d(g')|_q(f) ]^2 - 2\kappa^2 [ dg|_q(f) ]^2 \big\} 
	- d^2F|_0(f,f) .
	\label{eq:cleanup lem 9}
	\end{align}
	We will prove the estimate~\eqref{eq:cleanup lem 5} by estimating each of the terms~\eqref{eq:cleanup lem 6}--\eqref{eq:cleanup lem 9} in $L^1$, but first we record some useful estimates for the functional derivatives of $g$.
	
	Estimating the first functional derivative~\eqref{eq:cleanup lem 10} by duality, expanding $R(\kappa,q)$ as the series~\eqref{eq:resolvent series 2}, and using the operator estimate~\eqref{eq:hskappa identity 1} we have
	\begin{equation}
	\norm{ dg|_q(f) }_{H^1_\kappa} \lesssim \kappa^{-1} \norm{f}_{H^{-1}_\kappa}
	\label{eq:cleanup lem 11}
	\end{equation}
	uniformly for $q\in B_A(\kappa)$ and $\kappa$ large.  Similarly, if we remove the leading term of $dg|_q(f)$ we obtain
	\begin{equation}
	\norm{ dg|_q(f) - dg|_0(f) }_{H^1_\kappa} \lesssim \kappa^{-3/2} A \norm{f}_{H^{-1}_\kappa}
	\label{eq:cleanup lem 12}
	\end{equation}
	uniformly for $q\in B_A(\kappa)$ and $\kappa$ large.  Another application of the resolvent identity shows that
	\begin{equation*}
	d^2g|_q(f,f) = 2\langle \del_x, R(\kappa,q) f R(\kappa,q) f R(\kappa,q) \del_x \rangle ,
	\end{equation*}
	and estimating this by duality we have
	\begin{equation}
	\norm{ d^2g|_q(f,f) }_{H^1_\kappa} \lesssim \kappa^{-3/2} \norm{f}_{H^{-1}_\kappa}^2
	\label{eq:cleanup lem 13}
	\end{equation}
	uniformly for $q\in B_A(\kappa)$ and $\kappa$ large.
	
	For the first term~\eqref{eq:cleanup lem 6} we use the estimates~\eqref{eq:diffeo prop}, \eqref{eq:h1kappa embedding}, \eqref{eq:cleanup lem 11}, and \eqref{eq:cleanup lem 13} along with the observation that $\norm{ h }_{L^2} \lesssim \kappa^{-1} \norm{ h }_{H^1_\kappa}$ to bound
	\begin{align*}
	&\norm{ \eqref{eq:cleanup lem 6} }_{L^1} \\
	&\lesssim \big( \norm{ g'(\kappa,q) }_{L^2}^2 + \kappa^2 \norm{ g(\kappa,q) - \tfrac{1}{2\kappa} }_{L^2}^2 \big) \bigg( \norm{ \tfrac{[dg|_q(f)]^2}{g(\kappa,q)^3} }_{L^\infty} + \norm{ \tfrac{d^2g|_q(f,f) }{g(\kappa,q)^2} }_{L^\infty} \bigg) \\
	&\lesssim \kappa^{-2} A^2 \big( \kappa^3 \cdot \kappa^{-3} \norm{f}_{H^{-1}_\kappa}^2 + \kappa^2 \cdot \kappa^{-2} \norm{f}_{H^{-1}_\kappa}^2 \big)
	\lesssim \kappa^{-2} A \norm{f}^2_{H^{-1}_\kappa}
	\end{align*}
	uniformly for $q\in B_A(\kappa)$ and $\kappa$ large.
	
	For the second term~\eqref{eq:cleanup lem 7} we note that $d(g')|_q(f) = [dg|_q(f)]'$ can be bounded in $L^2$ by~\eqref{eq:cleanup lem 11}, yielding
	\begin{align*}
	&\norm{ \eqref{eq:cleanup lem 7} }_{L^1} \\
	&\lesssim \norm{ \tfrac{dg|_q(f)}{g(\kappa,q)^2} }_{L^\infty} \big( \norm{ g'(\kappa,q) }_{L^2} \norm{ d(g')|_q(f) }_{L^2} + \kappa^2 \norm{ g(\kappa,q) - \tfrac{1}{2\kappa} }_{L^2} \norm{ dg|_q(f) }_{L^2} \big) \\
	&\lesssim \kappa^{1/2} \norm{f}_{H^{-1}_\kappa} \big( \kappa^{-2} A \norm{f}_{H^{-1}_\kappa} \big)
	\lesssim \kappa^{-3/2} A \norm{f}_{H^{-1}_\kappa}^2
	\end{align*}
	uniformly for $q\in B_A(\kappa)$ and $\kappa$ large.
	
	Similarly for the third term~\eqref{eq:cleanup lem 8} we have $d^2(g')|_q(f,f) = [d^2g|_q(f,f)]'$, and hence
	\begin{align*}
	\norm{ \eqref{eq:cleanup lem 8} }_{L^1}
	&\leq \tfrac{1}{2} \norm{ \tfrac{1}{g(\kappa,q)} }_{L^\infty} \norm{ g'(\kappa,q) }_{L^2} \norm{ d^2(g')|_q(f,f) }_{L^2} \\
	&\phantom{\leq{}}+ 2\kappa^2  \norm{ \tfrac{1}{g(\kappa,q)} }_{L^\infty} \norm{ g(\kappa,q) - \tfrac{1}{2\kappa} }_{L^2} \norm{ d^2g|_q(f,f) }_{L^2} \\
	&\lesssim \kappa^{-3/2} A \norm{f}_{H^{-1}_\kappa}^2
	+ \kappa^2\cdot \kappa^{-7/2} A \norm{f}_{H^{-1}_\kappa}^2
	\lesssim \kappa^{-3/2} A \norm{f}_{H^{-1}_\kappa}^2
	\end{align*}
	uniformly for $q\in B_A(\kappa)$ and $\kappa$ large.
	
	Lastly, to witness the convergence within the fourth term~\eqref{eq:cleanup lem 9} we estimate $[ dg|_q(f) ]^2 - [ dg|_0(f) ]^2$ in $L^1$ as the difference of squares in $L^2$ using~\eqref{eq:cleanup lem 12}:	
	\begin{align*}
	&\norm{ \eqref{eq:cleanup lem 9} }_{L^1}
	= \norm{ \left[ \tfrac{ [ d(g')|_q(f) ]^2 }{2g(\kappa,q)} - \kappa [d(g')|_0(f)]^2 \right] - 4\kappa^3 \left[ \tfrac{ [ dg|_q(f) ]^2 }{ 2\kappa g(\kappa,q) } - [dg|_0(f)]^2 \right] }_{L^1} \\
	&\leq \kappa \norm{ [ d(g')|_q(f) ]^2 - [ d(g')|_0(f) ]^2 }_{L^1} + \norm{ \tfrac{ 1 }{2g(\kappa,q)} - \kappa }_{L^\infty} \norm{ d(g')|_q(f) }_{L^2}^2 \\
	&\phantom{\leq{}}+ 4\kappa^3 \norm{ [ dg|_q(f) ]^2 - [ dg|_0(f) ]^2 }_{L^1} + 4\kappa^2 \norm{ \tfrac{ 1 }{2g(\kappa,q)} - \kappa }_{L^\infty} \norm{ dg|_q(f) }_{L^2}^2 \\
	&\lesssim \kappa^{-3/2} A \norm{f}_{H^{-1}_\kappa}^2 + \kappa^{-2} A \norm{f}_{H^{-1}_\kappa}^2
	\lesssim \kappa^{-3/2} A \norm{f}_{H^{-1}_\kappa}^2
	\end{align*}
	uniformly for $q\in B_A(\kappa)$ and $\kappa$ large.  Altogether we have demonstrated the desired inequality~\eqref{eq:cleanup lem 5}, which concludes the proof.
\end{proof}

We now recall the key conserved quantity $\alpha(\kappa,q)$ constructed in ~\cite{Killip2019}*{Prop.~2.4} to control $q$ in $H^{-1}$.  The same proof shows that given $A>0$, if we take the corresponding constant $\kappa_0$ from \cref{thm:diffeo prop}, then for all $\kappa\geq \kappa_0$ the quantity
\begin{equation*}
\alpha(\kappa,q) := \int_{\R} \left\{ - \frac{1}{2g(x;\kappa,q)} + \kappa + 2\kappa [ R_0(2\kappa)q ] (x) \right\} \dx
\end{equation*}
exists for all $q\in B_A(\kappa)$, is a real analytic functional of $q \in B_A(\kappa)$, and is conserved under the KdV flow (cf.~\cite{Killip2019}*{Prop.~3.1}):
\begin{equation}
\{ \alpha, \hkdv \} = 0 .
\label{eq:alpha 1}
\end{equation}
The quantity $\alpha(\kappa, q)$ is a renormalized logarithm of the transmission coefficient for the Schr\"odinger operator with potential $q$ (called the perturbation determinant) at energy $-\kappa^2$.

The formula for $\alpha$ is the trace of the integral kernel $-1/2G(x,y;\kappa,q)$ with the first two terms of its functional Taylor series about $q\equiv 0$ canceled, and consequently $\alpha(\kappa,q)$ is a nonnegative, strictly convex, real-analytic functional of $q\in B_A$.  Specifically, in ~\cite{Killip2019}*{Prop.~2.4} it is shown that the first derivative of $\alpha$ is given by
\begin{equation}
\frac{\del\alpha}{\del q} = \frac{1}{2\kappa} - g(x;\kappa,q) .
\label{eq:alpha 2}
\end{equation}
This vanishes for $q\equiv 0$, but the nondegenerate second derivative yields
\begin{equation}
\tfrac{1}{4} \kappa^{-1} \norm{ q }^2_{H^{-1}_\kappa} \leq \alpha(\kappa,q) \leq \kappa^{-1} \norm{ q }^2_{H^{-1}_\kappa} 
\label{eq:alpha 3}
\end{equation}
uniformly for $q\in B_A(\kappa)$ and $\kappa\geq \kappa_0$.  This last statement follows from the original proof of~\cite{Killip2019}*{Prop.~2.4} together with the estimate
\begin{equation*}
\left| d^2\alpha|_q(f,f) - d^2\alpha|_0(f,f) \right|
\lesssim \kappa^{-3/2} A \norm{ f}_{H^{-1}_\kappa}^2
\end{equation*}
uniformly for $q\in B_A(\kappa)$ and $\kappa$ large (which is true by~\eqref{eq:cleanup lem 12} and~\eqref{eq:alpha 2}).

The quantity $\alpha$ is also used to construct the $H_\kappa$ flows (cf.~\cite{Killip2019}*{Prop.~3.2}) which approximate the KdV flow as $\kappa\to\infty$.  For $\kappa \geq 1$ the Hamiltonian evolution induced by
\begin{equation}
H_\kappa := - 16\kappa^5 \alpha(\kappa,q) + 2\kappa^2 \int q(x)^2 \dx
\label{eq:hk}
\end{equation}
is
\begin{equation}
\ddt q(x) = 16\kappa^5 g'(x;\kappa,q) + 4\kappa^2 q'(x) .
\label{eq:hk flow}
\end{equation}
The flow conserves $\alpha(\varkappa,q(t))$ and commutes with those of KdV and $H_\varkappa$ for all $\varkappa \geq 1$:
\begin{equation}
\{ \alpha , H_\kappa \} = 0 , \qquad \{ H_\kappa , \hkdv \} = 0 , \qquad \{ H_\kappa , H_\varkappa \} = 0 .
\label{eq:hk flow 2}
\end{equation}

Given a solution $V(t,x)$ to KdV we define $V_\kappa(t,x)$ to be the $H_\kappa$ evolution of $V(0,x)$ as in~\eqref{eq:vk dfn}.  We will always assume that $V$ and $V_\kappa$ are admissible in the sense of \cref{thm:hyp}.  Just as how we obtained KdV with potential~\eqref{eq:tkdv} from KdV, we define the $\hktilde$ flow of $q$ from time $0$ to $t$ via
\begin{equation*}
\wt{\Phi}_{\kappa}(t) q = e^{tJ\nabla H_\kappa} (q+V(0)) - V_\kappa(t) .
\end{equation*}
In other words, $q(t,x)$ solves
\begin{equation}
\ddt q(t,x) = 16\kappa^5\left[ g'(x;\kappa,q(t)+V_\kappa(t)) - g'(x;\kappa,V_\kappa(t)) \right] + 4\kappa^2q'(t,x) .
\label{eq:hktilde flow}
\end{equation}
Formally, this flow is induced by the (time-dependent) Hamiltonian
\begin{equation*}
\hktilde := - 16\kappa^5 \left\{ \alpha(\kappa,q+V_\kappa) + \int [g(x;\kappa,V_\kappa) - \tfrac{1}{2\kappa}] q(x) \dx \right\}+ 2\kappa^2 \int q(x)^2 \dx .
\end{equation*}
We will not need this explicit formula for the Hamiltonian $\hktilde$, but we include it so that we are justified in using the Poisson bracket notation for its flow.

Throughout our analysis we will need to know that the first two terms of the series~\eqref{eq:g series 1} for $g(\kappa,V_\kappa)$ converge and dominate in the limit $\kappa\to\infty$.
\begin{lem}
	\label{thm:g conv 1}
	Fix $V$ admissible (in the sense of \cref{thm:hyp}) and $T>0$.  Then
	\begin{equation*}
	- 4\kappa^3 \left[ g(\kappa,V_\kappa) - \tfrac{1}{2\kappa} \right] \to V\quad
	\tx{in }W^{2,\infty}\tx{ as }\kappa\to\infty	
	\end{equation*}
	uniformly for $|t|\leq T$.
\end{lem}
\begin{proof}
	First we will examine the leading term of the series~\eqref{eq:g series 1} for $g(\kappa,V_\kappa) - \tfrac{1}{2\kappa}$.  From the integral kernel formula for $R_0$ we note that
	\begin{equation*}
	4\kappa^3 \langle \del_x , R_0 V_\kappa R_0 \del_x \rangle
	= 4\kappa^2 R_0(2\kappa) V_\kappa ,
	\end{equation*}
	and we claim this converges to $V$ in $W^{2,\infty}$ uniformly for $|t|\leq T$.  The operator $4\kappa^2 R_0(2\kappa)$ is convolution by the function $\kappa e^{-2\kappa|x|}$, whose integral is 1 for all $\kappa>0$.  Using this and the fundamental theorem of calculus we have
	\begin{align*}
	\left| 4\kappa^2 R_0(2\kappa) V_\kappa (x) - V_\kappa(x) \right|
	&= \left| \int \kappa e^{-2\kappa|y|} [V_\kappa(x-y) - V_\kappa(x)] \dy \right| \\
	&\leq \kappa \norm{V'_\kappa}_{L^\infty} \int e^{-2\kappa|y|} |y| \dy
	= 2 \kappa^{-1} \norm{V'_\kappa}_{L^\infty} 
	\end{align*}
	for all $x$.  As $V'_\kappa \in L^\infty$ and $V_\kappa \to V$ in $L^\infty$ uniformly for $|t|\leq T$ (by \cref{thm:hyp}), we conclude that $4\kappa^2 R_0(2\kappa) V_\kappa \to V$ in $L^\infty$ uniformly for $|t|\leq T$.  Differentiation commutes with $R_0(2\kappa)$, and so replacing $V_\kappa$ with $V'_\kappa,V''_\kappa$ and recalling that $V_\kappa \in W^{3,\infty}$ uniformly for $|t|\leq T$, we conclude that $4\kappa^2 R_0(2\kappa) V_\kappa \to V$ in $W^{2,\infty}$ uniformly for $|t|\leq T$.
	
	It remains to show that
	\begin{equation*}
	- 4\kappa^3 \left[ g(\kappa,V_\kappa) - \tfrac{1}{2\kappa} \right] - 4\kappa^2 R_0(2\kappa)V_\kappa \to 0\quad
	\tx{in }W^{2,\infty}\tx{ as }\kappa\to\infty
	\end{equation*}
	uniformly for $|t|\leq T$.  Using the series~\eqref{eq:g series 1} we estimate
	\begin{align*}
	&\left| 4\kappa^3 \left[ g(x;\kappa,V_\kappa) - \tfrac{1}{2\kappa} \right] + 4\kappa^2 R_0(2\kappa)V_\kappa(x) \right| \\
	&\leq 4\kappa^3 \sum_{\ell=2}^\infty \norm{ \sqrt{R_0}\del_x }_{L^2}^2 \norm{ \sqrt{R_0}V_\kappa \sqrt{R_0} }^\ell_\op 
	\lesssim 4\kappa^2 \sum_{\ell=2}^\infty \big( \kappa^{-2} \norm{ V_\kappa }_{L^\infty} \big)^\ell
	\lesssim_V \kappa^{-2} ,
	\end{align*}
	where we noted that $\snorm{ \sqrt{R_0}\del_x }_{L^2}^2 \lesssim \kappa^{-1}$ in Fourier variables.  This demonstrates the desired convergence in $L^\infty$.  Differentiating the translation identity~\eqref{eq:g trans prop} with respect to $h$ at $h=0$ yields
	\begin{equation*}
	g'(\kappa,V_\kappa)
	= \sum_{\ell=1}^\infty (-1)^\ell \sum_{j=0}^{\ell-1} \left\langle \del_x , R_0(V_\kappa R_0)^j V'_\kappa R_0 (V_\kappa R_0)^{\ell-1-j} \del_x \right\rangle .
	\end{equation*}
	Computing the second derivative similarly and using the same estimates, we also conclude 
	\begin{equation*}
	\snorm{ 4\kappa^3 g'(\kappa,V_\kappa) + 4\kappa^2 R_0(2\kappa)V'_\kappa }_{L^\infty} 
	+ \snorm{ 4\kappa^3 g''(\kappa,V_\kappa) + 4\kappa^2 R_0(2\kappa)V''_\kappa }_{L^\infty} 
	\lesssim_V \kappa^{-2}
	\end{equation*}
	because $V_\kappa \in W^{2,\infty}$ uniformly for $|t|\leq T$.  This demonstrates that the first and second derivatives converge in $L^\infty$ uniformly for $|t|\leq T$ as well.
\end{proof}

\section{The \texorpdfstring{$\hktilde$}{modified H-kappa} flow}
\label{sec:hktilde}

To eventually show that the $\hktilde$ flow~\eqref{eq:hktilde flow} converges to KdV with potential~\eqref{eq:tkdv} we will need to control the $H^{-1}$ norm of $q(t)$ under the $\hktilde$ flow.  As the $\hktilde$ flow already has the associated energy parameter $\kappa$, our tool for controlling $q$ in $H^{-1}$ is $\alpha(\varkappa,q(t))$ at an independent energy parameter $\varkappa$.  Both the $\hktilde$ flow and $\alpha$ involve the diagonal Green's function, and so we will be led to an integral involving $g(x;\kappa,q)$ and $g(x;\varkappa,q)$.  Expanding both into series, the resulting summands are no longer simply traces and so we will need to develop a new technique in order efficiently estimate such an integral.

To introduce the technique that we later use, we will first prove the commutativity relation
\begin{equation}
\int g(x;\varkappa,q) g'(x;\kappa,q) \dx = 0
\label{eq:biham relation}
\end{equation}
for Schwartz functions $q$, which expresses that $\alpha(\kappa,q)$ and $\alpha(\varkappa,q)$ Poisson commute (cf.~\cite{Killip2019}*{Prop.~3.2}).  When $\varkappa = \kappa$ the integrand is a total derivative and the vanishing of the integral is immediate, so assume $\varkappa\neq\kappa$.  First, we use the ODE~\eqref{eq:g alg prop 3} for $g(\kappa,q) = g(x;\kappa,q)$ to write
\begin{equation}
4(\kappa^2-\varkappa^2) g'(\kappa,q)
= g'''(\kappa,q) - 2q g'(\kappa,q) - 2 \left[ qg(\kappa,q) \right]' - 4\varkappa^2 g'(\kappa,q) .
\label{eq:biham relation 2}
\end{equation}
Substituting this for $g'(\kappa,q)$ in~\eqref{eq:biham relation} and integrating by parts we obtain
\begin{align*}
&\int g(\varkappa,q) g'(\kappa,q) \dx \\
&= \tfrac{1}{4(\kappa^2-\varkappa^2)} \int g(\varkappa,q) \left\{ g'''(\kappa,q) - 2q g'(\kappa,q) - 2 \left[ qg(\kappa,q) \right]' - 4\varkappa^2 g'(\kappa,q) \right\} \dx \\
&= - \tfrac{1}{4(\kappa^2-\varkappa^2)} \int \left\{ g'''(\varkappa,q) - 2q g'(\varkappa,q) - 2 \left[ qg(\varkappa,q) \right]' - 4\varkappa^2 g'(\varkappa,q) \right\} g(\kappa,q) \dx .
\end{align*}
Now we see that this last integral vanishes due to the ODE~\eqref{eq:g alg prop 3} for $g(\varkappa,q)$, thus proving~\eqref{eq:biham relation}.  We will refer to this procedure of using the ODE for one term, integrating by parts, and using the ODE for the other term as the \ti{commutativity relation trick}.

Now we are prepared to prove our main estimate for controlling the $H^{-1}$ norm of $q(t)$ under the $\hktilde$ flow:
\begin{prop}
	\label{thm:hktilde alphadot 1}
	Fix $T>0$ and $V$ admissible.  There exists a constant $C>0$ so that the following holds: given $A>0$ there exists $\varkappa_0>0$ so that solutions $q(t)\in B_A(\varkappa)$ to the $\hktilde$ flow~\eqref{eq:hktilde flow} obey
	\begin{equation*}
	\left| \ddt \alpha(\varkappa,q(t)) \right| \leq C\alpha(\varkappa,q(t))
	\end{equation*}
	uniformly for $|t|\leq T$, $\kappa \geq 2\varkappa$, and $\varkappa\geq \varkappa_0$.
\end{prop}
\begin{proof}
	We initialize $\varkappa_0$ so that the results from \cref{sec:prelim} hold for the balls $B_A(\varkappa)$ for all $\varkappa\geq\varkappa_0$.  First we compute the time derivative of $\alpha(\varkappa,q(t))$.  We will show that the $\hktilde$ flow is locally well-posed in $H^{-1}$ in \cref{thm:hktilde alphadot 2}, and so we may assume that $q$ is Schwartz by approximation.  Using the functional derivative~\eqref{eq:alpha 2} of $\alpha$ and the $\hktilde$ flow~\eqref{eq:hktilde flow}, we compute
	\begin{align*}
	&\ddt \alpha(\varkappa,q(t))
	= \{ \alpha , \hktilde \} \\
	&= - \int \left( g(x;\varkappa,q) - \tfrac{1}{2\varkappa} \right) \left\{ 16\kappa^5\left[ g'(x;\kappa,q+V_\kappa) - g'(x;\kappa,V_\kappa) \right] + 4\kappa^2q'(x) \right\} \dx ,
	\end{align*}
	where we have suppressed the time dependence of $q$ and $V_\kappa$.  The contribution from $q'$ vanishes because the $H_\kappa$ flow conserves momentum; this can be seen by integrating by parts and noting from~\eqref{eq:cleanup lem 1} that $qg'(\kappa,q)$ is a total derivative.  We are left with the expression
	\begin{equation*}
	\ddt \alpha(\varkappa,q(t))
	= - 16\kappa^5 \int \left( g(x;\varkappa,q) - \tfrac{1}{2\varkappa} \right) \left[ g'(x;\kappa,q+V_\kappa) - g'(x;\kappa,V_\kappa) \right] \dx .
	\end{equation*}
	We expect this to remain bounded in the limit $\kappa\to\infty$ from the convergence of $H_\kappa$ to $\hkdv$, but the factor of $\kappa^5$ obscures this bound.  To circumvent this, we use  the commutativity relation trick~\eqref{eq:biham relation 2} introduced at the beginning of this section.  Using the ODEs~\eqref{eq:g alg prop 3} for $g(\kappa,q+V_\kappa)$ and $g(\kappa,V_\kappa)$, integrating by parts, and then using the ODE for $g(\varkappa,q)$, we obtain
	\begin{align}
	&\ddt\alpha(\varkappa,q(t))
	\nonumber \\
	&= - \tfrac{8\kappa^5}{\kappa^2-\varkappa^2} \int \big\{ \left[ q ( g(\varkappa,q) - \tfrac{1}{2\varkappa} ) \right]' + qg'(\varkappa,q) \big\} g(\kappa,V_\kappa) \dx 
	\label{eq:hktilde alphadot 1}\\
	&+ \tfrac{8\kappa^5}{\kappa^2-\varkappa^2} \int \tfrac{1}{2\varkappa} q' \left[ g(\kappa,q+V_\kappa) - g(\kappa,V_\kappa) \right] \dx
	\label{eq:hktilde alphadot 2}\\
	&- \tfrac{8\kappa^5}{\kappa^2-\varkappa^2} \int \big\{ \left[ V_\kappa ( g(\varkappa,q) - \tfrac{1}{2\varkappa} ) \right]' + V_\kappa g'(\varkappa,q) \big\} \left[ g(\kappa,q+V_\kappa) - g(\kappa,V_\kappa) \right] \dx .
	\label{eq:hktilde alphadot 3}
	\end{align}
	We have suppressed the spatial integration variable $x$ for all integrands.  We will show that~\eqref{eq:hktilde alphadot 1} and~\eqref{eq:hktilde alphadot 2} are acceptable contributions, and then we will manipulate~\eqref{eq:hktilde alphadot 3} further.
	
	For the term~\eqref{eq:hktilde alphadot 1} we insert $\tfrac{1}{2\kappa} - \tfrac{1}{4\kappa^3} V$ in place of $g(\kappa,V_\kappa)$:
	\begin{align*}
	\eqref{eq:hktilde alphadot 1}
	&= - \tfrac{8\kappa^5}{\kappa^2-\varkappa^2} \int \left[ q ( g(\varkappa,q) - \tfrac{1}{2\varkappa} ) \right]' \big( g(\kappa,V_\kappa) -\tfrac{1}{2\kappa} + \tfrac{1}{4\kappa^3} V \big) \dx \\
	&\phantom{={}}- \tfrac{8\kappa^5}{\kappa^2-\varkappa^2} \int q g'(\varkappa,q) \big( g(\kappa,V_\kappa) -\tfrac{1}{2\kappa} + \tfrac{1}{4\kappa^3} V \big) \dx \\
	&\phantom{={}}- \tfrac{8\kappa^5}{\kappa^2-\varkappa^2} \int \big\{ \big[ q ( g(\varkappa,q) - \tfrac{1}{2\varkappa} ) \big]' + qg'(\varkappa,q) \big\} \left( \tfrac{1}{2\kappa} - \tfrac{1}{4\kappa^3} V \right) \dx .
	\end{align*}
	To estimate the first integral on the RHS we integrate by parts, use $H^1_\varkappa$-$H^{-1}_\varkappa$ duality, and use the estimates~\eqref{eq:hskappa ests}, \eqref{eq:diffeo prop}, and \eqref{eq:alpha 3}:
	\begin{align*}
	&\tfrac{8\kappa^5}{\kappa^2-\varkappa^2} \left| \int ( g(\varkappa,q) - \tfrac{1}{2\varkappa} ) q \big( g(\kappa,V_\kappa) -\tfrac{1}{2\kappa} + \tfrac{1}{4\kappa^3} V \big)' \dx \right| \\
	&\lesssim \tfrac{\kappa^5}{\kappa^2-\varkappa^2} \norm{ g(\kappa,V_\kappa) -\tfrac{1}{2\kappa} + \tfrac{1}{4\kappa^3} V }_{W^{2,\infty}} \norm{ g(\varkappa,q) - \tfrac{1}{2\varkappa} }_{H^1_\varkappa} \norm{q}_{H^{-1}_\varkappa} \\
	&\lesssim \tfrac{\kappa^5}{\kappa^2-\varkappa^2} \norm{ g(\kappa,V_\kappa) -\tfrac{1}{2\kappa} + \tfrac{1}{4\kappa^3} V }_{W^{2,\infty}} \alpha(\varkappa,q) .
	\end{align*}
	The prefactor of $\alpha(\varkappa,q)$ here is bounded uniformly for $\kappa\geq 2\varkappa$ and $\varkappa$ large by the convergence of \cref{thm:g conv 1}.  For the second integral we use the identity~\eqref{eq:cleanup lem 1} for $qg'(\varkappa,q)$ and the $\alpha$ estimate~\eqref{eq:alpha 3}:
	\begin{align*}
	&\tfrac{8\kappa^5}{\kappa^2-\varkappa^2} \left| \int ( g(\varkappa,q) - \tfrac{1}{2\varkappa} ) \big[ q \big( g(\kappa,V_\kappa) -\tfrac{1}{2\kappa} + \tfrac{1}{4\kappa^3} V \big) \big]' \dx \right| \\
	&\lesssim \tfrac{\kappa^5}{\kappa^2-\varkappa^2} \norm{ g(\kappa,V_\kappa) -\tfrac{1}{2\kappa} + \tfrac{1}{4\kappa^3} V }_{W^{1,\infty}} \alpha(\varkappa,q) ,
	\end{align*}
	and the prefactor is again bounded uniformly for $\kappa\geq 2\varkappa$ and $\varkappa$ large.  For the third integral we integrate by parts to obtain
	\begin{equation*}
	\tfrac{8\kappa^5}{\kappa^2-\varkappa^2} \int ( g(\varkappa,q) - \tfrac{1}{2\varkappa} ) \big\{ \big[ ( \tfrac{1}{2\kappa} - \tfrac{1}{4\kappa^3} V ) q \big]' - \tfrac{1}{4\kappa^3} V'q \big\} \dx .
	\end{equation*}
	Note that the contribution from the term $\tfrac{1}{2\kappa}$ vanishes since $qg'(\varkappa,q)$ is a total derivative (cf.~\eqref{eq:cleanup lem 1}).  This leaves
	\begin{equation*}
	- \tfrac{2\kappa^2}{\kappa^2-\varkappa^2} \int ( g(\varkappa,q) - \tfrac{1}{2\varkappa} ) \left[ V'q + (Vq)' \right] \dx .
	\end{equation*}
	The prefactor $\tfrac{2\kappa^2}{\kappa^2-\varkappa^2}$ is now bounded uniformly for $\kappa \geq 2\varkappa$.  As before the contribution of $V'q$ is estimated by $H^1_\varkappa$-$H^{-1}_\varkappa$ duality and the contribution of $(Vq)'$ by~\eqref{eq:cleanup lem 1}, yielding
	\begin{equation*}
	\tfrac{2\kappa^2}{\kappa^2-\varkappa^2} \left| \int ( g(\varkappa,q) - \tfrac{1}{2\varkappa} ) \left[ V'q + (Vq)' \right] \dx \right| \lesssim \norm{ V }_{W^{2,\infty}} \alpha(\varkappa,q) .
	\end{equation*}
	
	To estimate the second term~\eqref{eq:hktilde alphadot 2} we expand
	\begin{equation*}
	\eqref{eq:hktilde alphadot 2}
	= \tfrac{4\kappa^5}{(\kappa^2-\varkappa^2)\varkappa} \sum_{\ell=1}^\infty (-1)^\ell \tr\{ (R(\kappa,V_\kappa)q)^\ell R(\kappa,V_\kappa) q' \} .
	\end{equation*}
	Next, we write
	\begin{equation*}
	R(\kappa,V_\kappa) q' = [ \partial , R(\kappa,V_\kappa) q ] - [\partial, R(\kappa,V_\kappa) ] q .
	\end{equation*}
	Note that contribution from $[ \partial , R(\kappa,V_\kappa) q ]$ vanishes by cycling the trace, and the contribution from $[\partial, R(\kappa,V_\kappa)] = - R(\kappa,V_\kappa) V'_\kappa R(\kappa,V_\kappa)$ is acceptable using the estimate~\eqref{eq:hskappa identity 2}:
	\begin{align*}
	|\eqref{eq:hktilde alphadot 2}|
	&\leq \tfrac{4\kappa^5}{(\kappa^2-\varkappa^2)\varkappa} \sum_{\ell=1}^\infty \left| \tr\{ (R(\kappa,V_\kappa)q)^{\ell} R(\kappa,V_\kappa) V'_\kappa R(\kappa,V_\kappa) q \} \right| \\
	&\leq \tfrac{4\kappa^5}{(\kappa^2-\varkappa^2)\varkappa} \sum_{\ell=1}^\infty \norm{ \sqrt{R(\kappa,V_\kappa)} \, q \sqrt{R(\kappa,V_\kappa)} }_{\I_2}^{\ell+1} \norm{ \sqrt{ R(\kappa,V_\kappa) } V'_\kappa \sqrt{ R(\kappa,V_\kappa) } }_\op \\
	&\lesssim \tfrac{4\kappa^5}{(\kappa^2-\varkappa^2)\varkappa} \sum_{\ell=1}^\infty \big(  \kappa^{-1/2} \norm{q}_{H^{-1}_\kappa} \big)^{\ell+1} \kappa^{-2} \norm{V'_\kappa}_{L^\infty} 
	\lesssim \norm{V'_\kappa}_{L^\infty} \alpha(\varkappa,q) .
	\end{align*}
	In the last step we noted that $\norm{q}_{H^{-1}_\kappa} \leq \norm{q}_{H^{-1}_\varkappa}$ for $\kappa\geq 2\varkappa$.
	
	It remains to estimate the third term~\eqref{eq:hktilde alphadot 3}, which will require more manipulation because the leading term in the expansion of $g(\kappa,q+V_\kappa) - g(\kappa,V_\kappa)$ is only $\bigo(\kappa)$.  First, we integrate by parts to move the derivative back onto $g(\kappa,q+V_\kappa) - g(\kappa,V_\kappa)$:
	\begin{align}
	\eqref{eq:hktilde alphadot 3}
	&= \tfrac{8\kappa^5}{\kappa^2-\varkappa^2} \int V'_\kappa ( g(\varkappa,q) - \tfrac{1}{2\varkappa} ) \left[ g(\kappa,q+V_\kappa) - g(\kappa,V_\kappa) \right] \dx 
	\label{eq:hktilde alphadot 4} \\
	&\phantom{={}}+ \tfrac{16\kappa^5}{\kappa^2-\varkappa^2} \int V_\kappa ( g(\varkappa,q) - \tfrac{1}{2\varkappa} ) \left[ g'(\kappa,q+V_\kappa) - g'(\kappa,V_\kappa) \right] \dx .
	\label{eq:hktilde alphadot 5}
	\end{align}
	Using $H^1_\kappa$-$H^{-1}_\kappa$ duality, the diagonal Green's function estimate~\eqref{eq:diffeo prop}, and the observation that $\norm{f}_{H^{-1}_\kappa} \leq \kappa^{-2} \norm{f}_{H^1_\kappa}$, we have
	\begin{align*}
	|\eqref{eq:hktilde alphadot 4}|
	&\leq \tfrac{8\kappa^5}{\kappa^2-\varkappa^2} \norm{ V'_\kappa }_{W^{1,\infty}} \norm{ g(\kappa,q+V_\kappa) - g(\kappa,V_\kappa) }_{H^{-1}_\kappa} \norm{ g(\varkappa,q) - \tfrac{1}{2\varkappa} }_{H^1_\kappa} \\
	&\lesssim \tfrac{8\kappa^5}{\kappa^2-\varkappa^2} \norm{ V'_\kappa }_{W^{1,\infty}} \kappa^{-3}\norm{q}_{H^{-1}_\kappa} \varkappa^{-1} \norm{q}_{H^{-1}_\varkappa} 
	\lesssim \norm{ V'_\kappa }_{W^{1,\infty}} \alpha(\varkappa,q)
	\end{align*}
	since $\norm{q}_{H^{-1}_\kappa} \leq \norm{q}_{H^{-1}_\varkappa}$ for $\kappa\geq 2\varkappa$.
	
	The remaining term~\eqref{eq:hktilde alphadot 5} resembles our original expression for $\ddt\alpha(\varkappa,q(t))$, except we have gained $\kappa^{-2}$ in decay and have introduced an extra factor of $V_\kappa$.  Consequently, we repeat the commutativity relation trick~\eqref{eq:biham relation 2}; pushing derivatives past the factor of $V_\kappa$ introduces extra terms, but they are relatively harmless.  After this manipulation, we regroup terms to arrive at
	\begin{align}
	&\eqref{eq:hktilde alphadot 5}
	\nonumber \\
	&= \int \tfrac{8\kappa^5 V_\kappa}{(\kappa^2-\varkappa^2)^2} \big\{ \big[ q ( g(\varkappa,q) - \tfrac{1}{2\varkappa} ) \big]' + qg'(\varkappa,q) \big\} g(\kappa,V_\kappa)  
	\label{eq:hktilde alphadot 6}\\ 
	&- \int \tfrac{8\kappa^5 V_\kappa q'}{ 2\varkappa (\kappa^2-\varkappa^2)^2} [ g(\kappa,q+V_\kappa) - g(\kappa,V_\kappa) ] 
	\label{eq:hktilde alphadot 7}\\ 
	&+ \int \tfrac{8\kappa^5 V_\kappa}{(\kappa^2-\varkappa^2)^2} \big\{ \big[ V_\kappa ( g(\varkappa,q) - \tfrac{1}{2\varkappa} ) \big]' + V_\kappa g'(\varkappa,q) \big\} [ g(\kappa,q+V_\kappa) - g(\kappa,V_\kappa) ]  
	\label{eq:hktilde alphadot 8}\\ 
	&- \int \tfrac{4\kappa^5}{(\kappa^2-\varkappa^2)^2} (V'''_\kappa - 4\varkappa^2 V'_\kappa - 2V'_\kappa) ( g(\varkappa,q) - \tfrac{1}{2\varkappa} ) [ g(\kappa,q+V_\kappa) - g(\kappa,V_\kappa) ] 
	\label{eq:hktilde alphadot 9}\\ 
	&- \int \tfrac{12\kappa^5}{(\kappa^2-\varkappa^2)^2} \big[ V'_\kappa g'(\varkappa,q) \big]' [ g(\kappa,q+V_\kappa) - g(\kappa,V_\kappa) ] 
	\label{eq:hktilde alphadot 10}\\ 
	&- \int \tfrac{8\kappa^5 V'_\kappa}{(\kappa^2-\varkappa^2)^2} ( g(\varkappa,q) - \tfrac{1}{2\varkappa} ) \big\{ q g(\kappa,q+V_\kappa) + V_\kappa [ g(\kappa,q+V_\kappa) - g(\kappa,V_\kappa) ] \big\} .
	\label{eq:hktilde alphadot 11} 
	\end{align}
	The first three terms~\eqref{eq:hktilde alphadot 6}--\eqref{eq:hktilde alphadot 8} are analogous to~\eqref{eq:hktilde alphadot 1}--\eqref{eq:hktilde alphadot 3} respectively, and the new terms~\eqref{eq:hktilde alphadot 9}--\eqref{eq:hktilde alphadot 11} are the result of derivatives falling on the new factor of $V_\kappa$.
	
	Estimating~\eqref{eq:hktilde alphadot 6} as we did~\eqref{eq:hktilde alphadot 1}, we obtain
	\begin{equation*}
	|\eqref{eq:hktilde alphadot 6}| \lesssim \big( \norm{V_\kappa}_{W^{2,\infty}} + \norm{V_\kappa}^2_{W^{2,\infty}} \big) \alpha(\varkappa,q) .
	\end{equation*}
	
	Conversely, the extra factor of $V_\kappa$ prohibits us from treating the term~\eqref{eq:hktilde alphadot 7} as we did~\eqref{eq:hktilde alphadot 2}.  Instead, we must maneuver the derivative onto $V_\kappa$.  Expanding
	\begin{equation*}
	\eqref{eq:hktilde alphadot 7}
	=- \frac{4\kappa^5}{(\kappa^2-\varkappa^2)^2} \sum_{\ell=1}^\infty \frac{(-1)^\ell}{\varkappa} \tr\{ (R(\kappa,V_\kappa)q)^\ell R(\kappa,V_\kappa) V_\kappa [\partial,q] \} ,
	\end{equation*}
	we write
	\begin{align*}
	&\tr\{ (R(\kappa,V_\kappa)q)^\ell R(\kappa,V_\kappa) V_\kappa [\partial,q] \} \\
	&= \tr\{ (R(\kappa,V_\kappa)q)^\ell [ R(\kappa,V_\kappa) , V_\kappa ] \partial q \}
	- \tr\{ (R(\kappa,V_\kappa)q)^\ell V_\kappa [ \partial, R(\kappa,V_\kappa)]  q \}
	\end{align*}
	by linearity and cycling the trace.  For the contribution of $[ \partial, R(\kappa,V_\kappa)] = - R(\kappa,V_\kappa)$ $V'_\kappa R(\kappa,V_\kappa)$, we use the operator estimate~\eqref{eq:hskappa identity 2} to bound
	\begin{align*}
	&|\tr\{ (R(\kappa,V_\kappa)q)^\ell V_\kappa [ \partial, R(\kappa,V_\kappa)]  q \}| \\
	&\leq \norm{ \sqrt{R(\kappa,V_\kappa)} V'_\kappa \sqrt{R(\kappa,V_\kappa)} }_\op \norm{ \sqrt{R(\kappa,V_\kappa)} \, q \sqrt{R(\kappa,V_\kappa)} }^{\ell+1}_{\I_2}  \\
	&\lesssim \kappa^{-2} \norm{V'_\kappa}_{L^\infty} \big( \kappa^{-1/2} \norm{q}_{H^{-1}_\kappa} \big)^{\ell+1}
	\end{align*}
	for $\ell \geq 1$.  For the contribution of first commutator
	\begin{equation*}
	[ R(\kappa,V_\kappa) , V_\kappa ] \partial
	= R(\kappa,V_\kappa) V''_\kappa R(\kappa,V_\kappa) \partial + R(\kappa,V_\kappa) 2V'_\kappa \partial R(\kappa,V_\kappa) \partial ,
	\end{equation*}
	we pair each $q$ with two copies of $\sqrt{R(\kappa,V_\kappa)}$ in $\I_2$ and each $\partial$ with one copy of $\sqrt{R(\kappa,V_\kappa)}$ in operator norm; the first term contributes
	\begin{align*}
	&| \tr\{ (R(\kappa,V_\kappa)q)^\ell R(\kappa,V_\kappa) V''_\kappa R(\kappa,V_\kappa) \partial q \} | \\
	&\leq \norm{ \sqrt{R(\kappa,V_\kappa)} V''_\kappa \partial \sqrt{R(\kappa,V_\kappa)} }_\op \norm{ \sqrt{R(\kappa,V_\kappa)}\, q \sqrt{R(\kappa,V_\kappa)} }^{\ell+1}_{\I_2} \\
	&\phantom{\leq{}}+ | \tr\{ (R(\kappa,V_\kappa)q)^\ell R(\kappa,V_\kappa) V''_\kappa [ \partial, R(\kappa,V_\kappa) ] q \} |  \\
	&\lesssim \kappa^{-1} \big( \norm{V_\kappa}_{W^{2,\infty}} + \norm{V_\kappa}_{W^{2,\infty}}^2 \big) \big( \kappa^{-1/2} \norm{q}_{H^{-1}_\kappa} \big)^{\ell+1} ,
	\end{align*}
	and the second term contributes
	\begin{align*}
	&| \tr\{ (R(\kappa,V_\kappa)q)^\ell R(\kappa,V_\kappa) V'_\kappa \partial R(\kappa,V_\kappa) \partial q \} | \\
	&\leq \norm{ \sqrt{R(\kappa,V_\kappa)} \big( \partial V'_\kappa - V''_\kappa \big) \partial \sqrt{R(\kappa,V_\kappa)} }_\op \norm{ \sqrt{R(\kappa,V_\kappa)} \, q \sqrt{R(\kappa,V_\kappa)} }^{\ell+1}_{\I_2} \\
	&\phantom{\leq{}}+ | \tr\{ (R(\kappa,V_\kappa)q)^\ell R(\kappa,V_\kappa) V'_\kappa \partial [\partial,R(\kappa,V_\kappa)] q \} |  \\
	&\lesssim \norm{V'_\kappa}_{L^\infty} \big( \kappa^{-1/2} \norm{q}_{H^{-1}_\kappa} \big)^{\ell+1} .
	\end{align*}
	Summing over $\ell \geq 1$ we gain a factor of $\kappa^{-1}$ to counteract the prefactor $\kappa^5 (\kappa^2-\varkappa^2)^{-2}$, and the remaining factor of $\varkappa^{-1}$ is paired with $\norm{q}_{H^{-1}_\kappa}^2\leq \norm{q}_{H^{-1}_\varkappa}^2$ to conclude
	\begin{equation}
	|\eqref{eq:hktilde alphadot 7}| \lesssim \big( \norm{V_\kappa}_{W^{2,\infty}} + \norm{V_\kappa}_{W^{2,\infty}}^2 \big) \alpha(\varkappa,q)
	\label{eq:hktilde alphadot 12}
	\end{equation}
	uniformly for $\kappa \geq 2\varkappa$ and $\varkappa$ large.
	
	We now have enough decay in $\kappa$ to treat the term~\eqref{eq:hktilde alphadot 8} directly.  Using Cauchy--Schwarz, the diagonal Green's function estimate~\eqref{eq:diffeo prop}, and the observation that $\norm{f}_{L^2} \lesssim \kappa^{-1} \norm{f}_{H^1_\kappa}$ we have
	\begin{align*}
	|\eqref{eq:hktilde alphadot 8}|
	&\lesssim \tfrac{\kappa^5}{(\kappa^2-\varkappa^2)^2} \norm{V_\kappa}_{W^{1,\infty}}^2 \norm{ g(\varkappa,q) - \tfrac{1}{2\varkappa} }_{H^1} \norm{ g(\kappa,q+V_\kappa) - g(\kappa,V_\kappa) }_{L^2} \\
	&\lesssim \tfrac{\kappa^5}{(\kappa^2-\varkappa^2)^2} \norm{V_\kappa}_{W^{1,\infty}}^2 \varkappa^{-1} \norm{q}_{H^{-1}_\varkappa}  \kappa^{-2} \norm{q}_{H^{-1}_\kappa}
	\lesssim \norm{V_\kappa}_{W^{1,\infty}}^2 \alpha(\varkappa,q)
	\end{align*}
	uniformly for $\kappa\geq 2\varkappa$ and $\varkappa$ large.  This same technique also works for~\eqref{eq:hktilde alphadot 9} and~\eqref{eq:hktilde alphadot 10} after integrating by parts once:
	\begin{align*}
	|\eqref{eq:hktilde alphadot 9}| 
	&\lesssim \tfrac{\kappa^5}{(\kappa^2-\varkappa^2)^2} \big( \norm{ V''_\kappa }_{L^\infty} \varkappa^{-2} \kappa^{-1} + \norm{ V'_\kappa}_{L^\infty} \kappa^{-2} \big) \norm{q}_{H^{-1}_\varkappa} \norm{q}_{H^{-1}_\kappa} \\
	&\lesssim \norm{ V'_\kappa}_{W^{1,\infty}} \alpha(\varkappa,q),  \\
	|\eqref{eq:hktilde alphadot 10}| 
	&\lesssim \tfrac{\kappa^5}{(\kappa^2-\varkappa^2)^2} \norm{V'_\kappa}_{L^\infty} \varkappa^{-1} \norm{q}_{H^{-1}_\varkappa} \kappa^{-1} \norm{q}_{H^{-1}_\kappa}
	\lesssim \norm{V'_\kappa}_{L^\infty} \alpha(\varkappa,q)
	\end{align*}
	uniformly for $\kappa\geq 2\varkappa$ and $\varkappa$ large.  For the last term~\eqref{eq:hktilde alphadot 11} we note that the extra factor of $g(\kappa,q+V_\kappa)$ can be put in $H^1$ by the second inequality of~\eqref{eq:hskappa ests}, and so by $H^1_\varkappa$-$H^{-1}_\varkappa$ duality we have
	\begin{align*}
	|\eqref{eq:hktilde alphadot 11}|
	&\lesssim \tfrac{\kappa^5}{(\kappa^2-\varkappa^2)^2} \norm{V'_\kappa}_{L^\infty} \big\{ \norm{ g(\kappa,q+V_\kappa) }_{H^1} \norm{ q }_{H^{-1}_\varkappa} \norm{ g(\varkappa,q) - \tfrac{1}{2\varkappa} }_{H^1_\varkappa} \\
	&\qquad\qquad\qquad + \norm{V_\kappa}_{L^\infty} \norm{ g(\varkappa,q) - \tfrac{1}{2\varkappa} }_{L^2} \norm{ g(\kappa,q+V_\kappa) - g(\kappa,V_\kappa) }_{L^2} \big\} \\
	&\lesssim \norm{V'_\kappa}_{L^\infty} \left( 1+\norm{V_\kappa}_{L^\infty} \right) \alpha(\varkappa,q)
	\end{align*}
	uniformly for $\kappa\geq 2\varkappa$ and $\varkappa$ large.  This concludes the estimate of~\eqref{eq:hktilde alphadot 5} and hence the proof of \cref{thm:hktilde alphadot 1}.
\end{proof}

From \cref{thm:hktilde alphadot 1} we are able to conclude that the $H^{-1}$ norm of $q(t)$ is controlled by that of $q(0)$.  We use this to show that the approximate flows $\hktilde$ are globally well-posed in $H^{-1}$:
\begin{prop}
	\label{thm:hktilde alphadot 2}
	Fix $V$ admissible.  Given $A,T>0$ there exists $\kappa_0>0$ so that for $\kappa \geq \kappa_0$ the $\hktilde$ flows~\eqref{eq:hktilde flow} with initial data $q(0) \in B_A$ have solutions $q_\kappa(t)$ which are unique in $C_tH^{-1}([-T,T]\times\R)$, depend continuously on the initial data, and are bounded in $C_tH^{-1}([-T,T]\times\R)$ uniformly for $\kappa\geq \kappa_0$.
	
	Moreover, for all $\varkappa$ sufficiently large the diagonal Green's function $g(\varkappa,q) = g(x;\varkappa,q(t))$ evolves according to
	\begin{align*}
	&\begin{aligned} 
	\tfrac{d}{dt} \tfrac{1}{2g(\varkappa,q)} =
	\left\{ \tfrac{4\kappa^5}{\kappa^2-\varkappa^2} \left[ \tfrac{\varkappa}{\kappa} - \tfrac{g(\kappa,q+V_\kappa)-g(\kappa,V_\kappa)}{g(\varkappa,q)} \right] + \left( 2\kappa^2 + \tfrac{\kappa^4}{\kappa^2-\varkappa^2} \right) \left[ \tfrac{1}{g(\varkappa,q)} - 2\varkappa \right] \right\}'
	\end{aligned}\\
	&\begin{aligned} - &\tfrac{4\kappa^5}{\kappa^2-\varkappa^2} \tfrac{1}{g(\varkappa,q)^2} \int G(x,y) \big\{ \big[ q(g(\kappa,V_\kappa)-\tfrac{1}{2\kappa}) \big]' + qg'(\kappa,V_\kappa) \\
	&+ \big[ V_\kappa ( g(\kappa,q+V_\kappa) - g(\kappa,V_\kappa) ) \big]' + V_\kappa [ g'(\kappa,q+V_\kappa) - g'(\kappa,V_\kappa) ] \big\}(y) G(y,x) \dy ,
	\end{aligned}
	\end{align*}
	where $G(x,y) = G(x,y;\varkappa,q)$ and the dependence on $(t,x)$ is suppressed.
\end{prop}
\begin{proof}
	The solution $q(t)$ of the $\hktilde$ flow satisfies the integral equation
	\begin{equation*}
	q(t) = e^{t4\kappa^2\partial_x}q(0) + 16\kappa^5 \int_0^t e^{(t-s)4\kappa^2\partial_x} \left[ g'(\kappa,q(s)+V_\kappa(s)) - g'(\kappa,V_\kappa(s)) \right]\ds .
	\end{equation*}
	Local well-posedness is proved by contraction mapping, provided we have the Lipschitz estimate
	\begin{align*}
	&\norm{ g'(\kappa,q+V_\kappa) - g'(\kappa,\tilde{q}+V_\kappa) }_{H^{-1}} \\
	&\lesssim \norm{ \left[ g(\kappa,q+V_\kappa) - g(\kappa,V_\kappa) \right] - \left[ g(\kappa,\tilde{q}+V_\kappa) - g(\kappa,V_\kappa) \right] }_{H^{1}}
	\lesssim \norm{ q - \tilde{q} }_{H^{-1}} .
	\end{align*}
	To prove this Lipschitz estimate, it suffices to show that $f \mapsto \dd [ g(\kappa, \cdot + V_\kappa)] |_q (f)$ is bounded as an operator $H^{-1} \to H^1$ uniformly for $q\in B_A$.  Using the resolvent identity we calculate
	\begin{equation*}
	d[ g(\kappa, \cdot + V_\kappa) ]|_q (f)
	= - \langle \del_x , R(\kappa,q+V_\kappa) f R(\kappa,q+V_\kappa) \del_x \rangle .
	\end{equation*}
	Estimating by duality, expanding the series~\eqref{eq:resolvent series 2}, and using the estimate~\eqref{eq:hskappa identity 2} we obtain
	\begin{equation*}
	\norm{ d[g(\kappa, \cdot + V_\kappa)]|_q (f) }_{H^1} \lesssim \norm{f}_{H^{-1}}
	\end{equation*}
	uniformly for $q\in B_A$ and $\kappa\geq \kappa_0$.  Here, $\kappa_0$ is chosen so that the results of \cref{sec:prelim} and the above estimate apply to the ball of radius $A$.
	
	For global well-posedness, we will need to choose $\kappa_0$ even larger.  Let $C$ denote the constant from \cref{thm:hktilde alphadot 1}, which depends only on the background wave $V$ and $T>0$.  Then Gr\"onwall's inequality and the $\alpha$ estimate~\eqref{eq:alpha 3} tell us that the $\hktilde$ flows $q_\kappa(t)$ obey
	\begin{equation}
	\norm{q_\kappa(t)}^2_{H^{-1}_\varkappa}
	\leq 4 \varkappa\alpha(\varkappa,q_\kappa(t))
	\leq 4 e^{CT} \varkappa\alpha(\varkappa,q(0))
	\leq 4 e^{CT} A^2
	\label{eq:gwp}
	\end{equation}
	for $|t|\leq T$, $\kappa\geq 2\varkappa$, and $\varkappa$ sufficiently large.  
	
	Fix $\varkappa=\varkappa_0$ sufficiently large so that \cref{thm:hktilde alphadot 1}  and the $\alpha$ estimate~\eqref{eq:alpha 3} apply throughout the ball $B_R(\varkappa)$ in $H^{-1}_\varkappa$ with radius $R = 2e^{CT/2}A$.  The estimate~\eqref{eq:gwp} then applies, and so the $\hktilde$ flows $q_\kappa(t)$ remain in the ball $B_R(\varkappa)$ as long as $|t|\leq T$ and $\kappa\geq 2\varkappa$.  For each $\kappa \geq 2\varkappa$, the $\hktilde$ flow is locally well-posed on the ball $B_R$ in $H^{-1}_\varkappa$ by the elementary estimate $\norm{f}_{H^{-1}_\varkappa} \approx_\varkappa \norm{f}_{H^{-1}}$.  Therefore we may iterate the local well-posedness result to the whole time interval $[-T,T]$.  Moreover, using the estimate $\norm{f}_{H^{-1}} \lesssim \varkappa \norm{f}_{H^{-1}_\varkappa}$ again, we conclude that the $\hktilde$ flows $q_\kappa(t)$ are bounded in $H^{-1}$ uniformly for $|t|\leq T$ and $\kappa\geq 2\varkappa$.
	
	Next we turn to the second statement.  From the expression~\eqref{eq:cleanup lem 10} for the functional derivative of $g$ we have
	\begin{align*}
	&\ddt g(x;\varkappa,q(t))
	= \big\{ g(\varkappa,q) , \hktilde \big\}
	= - \int G(x,y) \frac{dq}{dt}(y) G(y,x) \dy \\
	&= - \int G(x,y) \big\{ 16\kappa^5 [g'(\kappa,q+V_\kappa)-g'(\kappa,V_\kappa)] + 4\kappa^2 q' \big\} (y) G(y,x) \dy \\
	&= - 4\kappa^2g'(x;\varkappa,q) + 16\kappa^5 \int G(x,y) [g'(\kappa,q+V_\kappa)-g'(\kappa,V_\kappa)](y) G(y,x) \dy ,
	\end{align*}
	where $G(x,y) = G(x,y;\varkappa,q)$.  Using the ODEs~\eqref{eq:g alg prop 3} for $g(\kappa,q+V_\kappa)$ and $g(\kappa,V_\kappa)$ and then the identity~\eqref{eq:g alg prop 2}, we obtain
	\begin{align*}
	&16\kappa^5 \int G(x,y) [g'(\kappa,q+V_\kappa)-g'(\kappa,V_\kappa)](y) G(y,x) \dy = \\
	&= \tfrac{8\kappa^5}{\kappa^2-\varkappa^2} \big\{ g'(\varkappa,q) [g(\kappa,q+V_\kappa)-g(\kappa,V_\kappa)] - g(\varkappa,q) \big[ g(\kappa,q+V_\kappa)-g(\kappa,V_\kappa) \big]' \big\} \\
	&\phantom{={}} - \tfrac{4\kappa^5}{\kappa^2-\varkappa^2} \int G(x,y) \big\{ \big( V_\kappa [ g(\kappa,q+V_\kappa) - g(\kappa,V_\kappa) ] \big)' \\
	&\phantom{={}}\qquad + V_\kappa [ g'(\kappa,q+V_\kappa) - g'(\kappa,V_\kappa) ] + \big[ qg(\kappa,V_\kappa) \big]' + qg'(\kappa,V_\kappa) \big\}(y) G(y,x) \dy.
	\end{align*}
	Lastly, replacing $g$ by $g-\tfrac{1}{2\kappa}$ in the term $[ qg(\kappa,V_\kappa) ]'$ and using the formula~\eqref{eq:g derivative} for $g'$ we write
	\begin{align*}
	&- \tfrac{4\kappa^5}{\kappa^2-\varkappa^2} \int G(x,y) \big[ qg(\kappa,V_\kappa) \big]'(y) G(y,x) \dy \\
	&= - \tfrac{4\kappa^5}{\kappa^2-\varkappa^2} \int G(x,y) \big[ q( g(\kappa,V_\kappa) - \tfrac{1}{2\kappa} ) \big]'(y) G(y,x) \dy
	+ \tfrac{2\kappa^4}{\kappa^2-\varkappa^2} g'(x;\varkappa,q) .
	\end{align*}
	Differentiating $1/2g(\varkappa,q)$ using the chain rule and regrouping terms yields the desired expression.
\end{proof}

\section{Convergence at low regularity}
\label{sec:conv 1}

Ultimately we will show that the $\hktilde$ flows $q_\kappa(t)$ are convergent in $H^{-1}$ as $\kappa\to\infty$.  To this end, we will show the difference $q_\kappa - q_\varkappa$ for $\varkappa\geq \kappa$ converges to zero as $\kappa\to\infty$.  This is a difficult task, as it involves estimating two different functions that solve separate nonlinear equations.  To circumvent this, we will use that the $H_\kappa$ and $H_\varkappa$ flows commute (cf.~\eqref{eq:hk flow 2}).  This allows us to write the $H_\varkappa$ flow of $u$ by time $t$ as
\begin{equation*}
e^{tJ\nabla H_\varkappa} u = e^{tJ\nabla (H_\varkappa - H_\kappa)} e^{tJ\nabla H_\kappa} u .
\end{equation*}
We apply this identity to $u = q+V$ and $u=V$.  Then $q_\varkappa(t)$, the $\hvktilde$ flow of $q(0)$ by time $t$, is the solution to
\begin{equation}
\begin{aligned}
\ddt q ={} &16\varkappa^5 [ g'(\varkappa,q+W(t)) - g'(\varkappa,W(t)) ] + 4\varkappa^2 q' \\
&-16\kappa^5 [ g'(\kappa,q+W(t)) - g'(\kappa,W(t)) ] - 4\kappa^2 q'  
\end{aligned}
\label{eq:diff flow q}
\end{equation}
at time $t$ with initial data $q_\kappa(t)$.  Here, the background wave $W(t) \equiv W_{\varkappa,\kappa}(t)$ is the solution to
\begin{equation}
\ddt W = 16\varkappa^5 g'(\varkappa,W) + 4\varkappa^2 W' -16\kappa^5 g'(\kappa,W) - 4\kappa^2 W'
\label{eq:diff flow V}
\end{equation}
at time $t$ with initial data $V_\kappa(t)$.  The upshot of this manipulation is that we may now write the difference $q_\varkappa(t) - q_\kappa(t)$ as the solution to the \emph{single} equation~\eqref{eq:diff flow q} minus its initial data.

The purpose of this section is to first demonstrate convergence at some lower $H^s$ regularity.  As was introduced in~\cite{Killip2019}, the change of variables $1/g(k,q)$ in place of $q$ is convenient in witnessing this convergence.
\begin{prop}
	\label{thm:diff flow conv 1}
	Fix $V$ admissible, $T>0$, and $k>0$ sufficiently large.  Given a bounded and equicontinuous set $Q\subset H^{-1}$ of initial data, define the set of $H_\kappa$ and $\hktilde$ flows
	\begin{equation*}
	V^*_T(\kappa) := \{ e^{tJ\nabla H_\varkappa}V(0) : |t|\leq T,\ \varkappa\geq \kappa\}, \qquad Q^*_{T}(\kappa) := \{ \wt{\Phi}_\kappa(t) q : q\in Q,\ |t| \leq T \}
	\end{equation*}
	for $\kappa > 0$ sufficiently large.  Then the solutions $q(t)$ to the difference flows~\eqref{eq:diff flow q} with background waves $W(t)$ and initial data $q(0)$ obey
	\begin{equation*}
	\lim_{\kappa\to\infty}\, \sup_{\substack{ q(0)\in Q^*_T(\kappa) \\ W(0) \in V^*_T(\kappa)}}\ \sup_{\varkappa\geq \kappa}\ \norm{ \frac{1}{g(k,q(t))} - \frac{1}{g(k,q(0))} }_{C_tH^{-2}([-T,T]\times\R)} = 0 .
	\end{equation*}
\end{prop}

Throughout the proof of~\cref{thm:diff flow conv 1} all spacetime norms will be over the slab $[-T,T]\times\R$.  As $V$ is admissible, there exists a constant $\kappa_0$ so that the $H_\kappa$ flows $V_\kappa(t)$ exist and are bounded in $C_tW^{4,\infty}$ uniformly for $\kappa\geq\kappa_0$.  By \cref{thm:hktilde alphadot 2} the difference flows $q(t)$ for $q(0)\in Q_T^*(\kappa)$ are bounded in $C_tH^{-1}$ uniformly for $\kappa$ large, and hence are contained in a ball $B_A$ for some $A>0$.  In particular, the functional $g(k,q) - \tfrac{1}{2k}$ for $q(t)$ exists for all $k$ sufficiently large.

By the fundamental theorem of calculus we have
\begin{equation*}
\norm{ \frac{1}{2g(k,q(t))} - \frac{1}{2g(k,q(0))} }_{C_tH^{-2}}
\leq T \norm{ \ddt \left( \frac{1}{2g(k,q(t))} - k \right) }_{C_tH^{-2}} ,
\end{equation*}
and so it suffices to show that
\begin{equation*}
\lim_{\kappa\to\infty}\, \sup_{\substack{ q(0)\in Q^*_T(\kappa) \\ W(0) \in V^*_T(\kappa)}}\ \sup_{\varkappa\geq \kappa}\ \norm{ \ddt \left( \frac{1}{2g(k,q(t))} - k \right) }_{C_tH^{-2}} = 0.
\end{equation*}

The equation~\eqref{eq:diff flow q} for the evolution of $q$ is the difference of the equations for the $\hvktilde$ and $\hktilde$ flows with the same background wave $W(t)$.  In fact, for a general function $F(q)$ evaluated at $q(t)$ we have
\begin{equation*}
\ddt F(q(t))
= \big\{ F , \hvktilde \big\} - \big\{ F, \hktilde \big\} .
\end{equation*}
We will apply this to the quantity $1/2g(k,q(t))$.  The expression for the evolution of $1/2g(k,q(t))$ under the $\hktilde$ flow was obtained in \cref{thm:hktilde alphadot 2}.  After regrouping terms, we arrive at
\begin{align}
&\ddt \frac{1}{2g(k,q(t))}
\nonumber \\
&= \left\{ \tfrac{1}{g(k,q)} \left( q + \tfrac{4\kappa^5}{\kappa^2-k^2} \left[ g(\kappa,q+W) - g(\kappa,W) \right] - \tfrac{4k^5}{\kappa^2-k^2} \left[ g(k,q) - \tfrac{1}{2k} \right] \right) \right\}' 
\label{eq:diff flow 4}\\
&+\tfrac{1}{g(k,q)^2} \int G(x,y) \big\{ V'q + \tfrac{4\kappa^5}{\kappa^2-k^2} g'(\kappa,W)q \big\}(y) G(y,x) \dy
\label{eq:diff flow 5}\\
&+\tfrac{1}{g(k,q)^2} \int G(x,y) \big\{ V'q + \tfrac{4\kappa^5}{\kappa^2-k^2} W' [g(\kappa,q+W) - g(\kappa,W)] \big\}(y) G(y,x) \dy
\label{eq:diff flow 6}\\
&+\tfrac{1}{g(k,q)^2} \int G(x,y) \big\{ (Vq)' + \tfrac{4\kappa^5}{\kappa^2-k^2} ( [ g(\kappa,W)-\tfrac{1}{2\kappa} ] q )' \big\}(y) G(y,x) \dy
\label{eq:diff flow 7}\\
&+\tfrac{1}{g(k,q)^2} \int G(x,y) \big\{ 2Vq' + \tfrac{8\kappa^5}{\kappa^2-k^2} W [g(\kappa,q+W) - g(\kappa,W)]' \big\}(y) G(y,x) \dy
\label{eq:diff flow 8}\\[0.4em]
&- \left\{ \tx{(4.3)--(4.7) with }\kappa\tx{ replaced by }\varkappa \right\} ,
\nonumber
\end{align}
where $G(x,y) = G(x,y;k,q)$.  Note that for each term we have subtracted the limiting expression as $\kappa\to\infty$ (e.g. inserting $(q/g(k,q))'$ in~\eqref{eq:diff flow 4} and $V'q$ in the integrand of~\eqref{eq:diff flow 5}) which is canceled by its counterpart in the corresponding $\varkappa$ terms.

To prove \cref{thm:diff flow conv 1} we must show that all of the terms above converge to zero in $C_tH^{-2}$ as $\kappa\to\infty$ uniformly for $\varkappa \geq \kappa$, $q(0)\in Q^*_T(\kappa)$, and $W(0) \in V^*_T(\kappa)$.  To simplify the notation, we will only show that the terms~\eqref{eq:diff flow 4}--\eqref{eq:diff flow 8} converge to zero as $\kappa\to\infty$; the upper bound we will obtain for each $\kappa$ term will also hold for the corresponding $\varkappa$ term uniformly for $\varkappa \geq \kappa$.

First, we claim that the admissibility of $V$ implies that the background waves $W(t) = e^{tJ\nabla (H_\varkappa - H_\kappa)}W$ obey
\begin{equation}
\lim_{\kappa\to\infty}\, \sup_{W\in V^*_T(\kappa)}\, \sup_{\varkappa\geq\kappa}\, \snorm{ e^{tJ\nabla (H_\varkappa - H_\kappa)}W - W }_{C_tW^{2,\infty}} = 0 .
\label{eq:diff flow 19}
\end{equation}
For $\varkappa\geq\kappa$ and $W\in V^*_T(\kappa)$, we use the commutativity of the KdV and $H_\kappa$ flows (cf.~\eqref{eq:hk flow 2}) to write
\begin{align*}
&\snorm{ e^{tJ\nabla (H_\varkappa - H_\kappa)}W - W }_{C_tW^{2,\infty}} \\
&\leq \snorm{ e^{tJ\nabla H_\varkappa} e^{- tJ\nabla H_\kappa} W - e^{tJ\nabla \hkdv} e^{- tJ\nabla H_\kappa} W }_{C_tW^{2,\infty}} \\
&\phantom{\leq{}}+ \snorm{ e^{- tJ\nabla H_\kappa} e^{tJ\nabla \hkdv} W - W }_{C_tW^{2,\infty}} \\
&\leq 2\, \sup_{W \in V^*_{2T}(\kappa)}\, \sup_{\varkappa\geq \kappa}\, \snorm{ e^{tJ\nabla H_\varkappa} W - e^{tJ\nabla \hkdv} W }_{C_tW^{2,\infty}} .
\end{align*}
The RHS converges to zero as $\kappa\to\infty$ by condition (iii) of \cref{thm:hyp}.

Now we turn to the first term~\eqref{eq:diff flow 4}, which arises in the case $V \equiv 0$ and is handled as in~\cite{Killip2019}.  Using the second estimate of~\eqref{eq:hskappa ests} we can put the factor of $1/g(k,q)$ in $H^1$ and bound
\begin{align*}
&\norm{\eqref{eq:diff flow 4}}_{H^{-2}} \\
&\lesssim \norm{ \tfrac{1}{g(k,q)} }_{H^1} \norm{ q + \tfrac{4\kappa^5}{\kappa^2-k^2} \left[ g(\kappa,q+W) - g(\kappa,W) \right] - \tfrac{4k^5}{\kappa^2-k^2} \left[ g(k,q) - \tfrac{1}{2k} \right] }_{H^{-1}} \\
&\lesssim \norm{ q + 4\kappa^3 \left[ g(\kappa,q+W) - g(\kappa,W) \right] }_{H^{-1}} \\
&\phantom{\lesssim{}}+ \kappa \norm{ g(\kappa,q+W) - g(\kappa,W) }_{H^{-1}} + \kappa^{-2}\norm{ g(k,q) - \tfrac{1}{2k} }_{H^{-1}} .
\end{align*}
We allow implicit constants to depend on the fixed constant $k>0$.  The second and third terms converge to zero as $\kappa\to 0$ by the estimate~\eqref{eq:diffeo prop}.  In the following lemma we check that the first term also converges to zero:

\begin{lem}
	\label{thm:g conv 2}
	We have
	\begin{equation*}
	4\kappa^3 [ g(\kappa,q+W) - g(\kappa,W) ] +q \to 0 \quad\tx{in }H^{-1}\tx{ as }\kappa\to\infty
	\end{equation*}
	uniformly for $q\in Q^*_T(\kappa)$.
\end{lem}
\begin{proof}
	We claim that the first term $4\kappa^2R_0(2\kappa)q$ of the series for $4\kappa^3 [ g(\kappa,q+W) - g(\kappa,W) ]$ converges to $q$ in $H^{-1}$.  We compute
	\begin{equation*}
	\norm{ 4\kappa^2R_0(2\kappa)q-q }_{H^{-1}}^2
	= \int \frac{\xi^4|\hat{q}(\xi)|^2}{(\xi^2+4\kappa^2)^2(\xi^2+4)} \dxi
	\lesssim \int  \frac{|\hat{q}(\xi)|^2}{\xi^2+4\kappa^2} \dxi
	= \norm{q}^2_{H^{-1}_\kappa} .
	\end{equation*}
	Note that for $q\in Q^*_T(\kappa)$, \cref{thm:hktilde alphadot 1} only gives us control over $\alpha(\varkappa,q)$ for $\kappa \geq 2\varkappa$ and not $\varkappa = \kappa$.  To circumvent this, we simply take $\varkappa = \kappa/2$ and note that trivially $\norm{q}^2_{H^{-1}_\kappa} \leq \norm{q}^2_{H^{-1}_{\kappa/2}}$.  If we let $C$ denote the constant from \cref{thm:hktilde alphadot 1}, then Gr\"onwall's inequality and the $\alpha$ estimate~\eqref{eq:alpha 3} yield
	\begin{equation*}
	\norm{q}^2_{H^{-1}_\kappa}
	\leq \norm{q}^2_{H^{-1}_{\kappa/2}}
	\lesssim \tfrac{\kappa}{2} \alpha(\tfrac{\kappa}{2},q)
	\leq e^{CT} \tfrac{\kappa}{2} \alpha(\tfrac{\kappa}{2},q(0)) ,
	\end{equation*}
	where $q(0)\in Q$ is the initial data for $q\in Q^*_T(\kappa)$.  As $Q$ is equicontinuous, we know $\tfrac{\kappa}{2} \alpha(\tfrac{\kappa}{2},q(0))$ converges to zero as $\kappa \to \infty$ uniformly for $q(0)\in Q$ by \cref{thm:hskappa equicty} and the $\alpha$ estimate~\eqref{eq:alpha 3}.  This completes the claim.
	
	It remains to show that
	\begin{equation*}
	4\kappa^3 [ g(\kappa,q+W) - g(\kappa,W) ] + 4\kappa^2 R_0(2\kappa)q \to 0 \quad\tx{in }H^{-1}\tx{ as }\kappa\to\infty
	\end{equation*}
	uniformly for $q\in B_A$.  We expand $g(\kappa,q+W) - g(\kappa,W)$ as a series in powers of $q$, and then expand each resolvent $R(\kappa,W)$ in powers of $W$.  We estimate by duality; for $f\in H^1$ and $R_0 = R_0(\kappa)$ we have
	\begin{align*}
	&\left| \int f(x) \big\{ 4\kappa^3 [  g(\kappa,q+W) - g(\kappa,W) ] + \tfrac{1}{\kappa} R_0(2\kappa)q \big\}(x) \dx \right| \\
	&\leq 4\kappa^3 \sum_{\substack{ m_0,m_1\geq 0 \\ m_0+m_1 \geq 1 }} \big| \tr\{ f R_0 (W R_0)^{m_0} q R_0 (W R_0)^{m_1} \} \big| \\
	&\phantom{\leq{}} + 4\kappa^3 \sum_{\ell\geq 2,\ m_0,\dots,m_\ell\geq 0} \big| \tr\{ f R_0 (W R_0)^{m_0} q R_0 (W R_0)^{m_1} q R_0 \cdots qR_0 (W R_0)^{m_\ell} \} \big| .
	\intertext{In the first sum, we put $\sqrt{R_0}q\sqrt{R_0}$ and $\sqrt{R_0}f\sqrt{R_0}$ in $\I_2$ and  measure the rest in operator norm.  For the second sum there are always at least two factors of $\sqrt{R_0}q\sqrt{R_0}$, and so we put $\sqrt{R_0}q\sqrt{R_0}$ in $\I_2$ and the rest in operator norm:}
	&\lesssim \kappa^3 \sum_{\substack{ m_0,m_1\geq 0 \\ m_0+m_1 \geq 1 }} \frac{\norm{f}_{L^2}}{\kappa^{3/2}} \frac{\norm{q}_{H^{-1}_\kappa}}{\kappa^{1/2}} \left( \frac{\norm{W}_{L^\infty}}{\kappa^2} \right)^{m_0+m_1}  \\
	&\phantom{\lesssim{}} + \kappa^3 \sum_{\ell\geq 2,\ m_0,\dots,m_\ell\geq 0} \frac{\norm{f}_{L^\infty}}{\kappa^{2}} \left( \frac{\norm{q}_{H^{-1}_\kappa}}{\kappa^{1/2}} \right)^\ell \left( \frac{\norm{W}_{L^\infty}}{\kappa^2} \right)^{m_0+\dots+m_\ell} .
	\end{align*}
	Re-indexing $m = m_0+\dots+m_\ell$, we compute
	\begin{equation}
	\begin{aligned}
	\sum_{m_0,\dots,m_\ell \geq 0} \left( \frac{\norm{W}_{L^\infty}}{\kappa^2} \right)^{m_0+\dots+m_\ell}
	&= \sum_{m=0}^\infty \frac{(\ell+m)!}{\ell!\,m!} \left( \frac{\norm{W}_{L^\infty}}{\kappa^2} \right)^{m} \\
	&= \left( 1 - \frac{\norm{W}_{L^\infty}}{\kappa^2} \right)^{\ell+1}
	\leq 1
	\end{aligned}
	\label{eq:double sum inner}
	\end{equation}
	uniformly for $\ell\geq 1$ and $\kappa$ large.  Altogether we obtain
	\begin{equation*}
	\left| \int f \big\{ 4\kappa^3 [  g(\kappa,q+W) - g(\kappa,W) ] + \tfrac{1}{\kappa} R_0(2\kappa)q \big\} \dx \right|
	\lesssim \norm{f}_{H^1} \norm{q}_{H^{-1}_\kappa} .
	\end{equation*}
	Taking a supremum over $\norm{f}_{H^1} \leq 1$, we conclude
	\begin{equation*}
	\norm{ 4\kappa^3 [  g(\kappa,q+W) - g(\kappa,W) ] + \tfrac{1}{\kappa} R_0(2\kappa)q }_{H^{-1}} \lesssim \norm{q}_{H^{-1}_\kappa} .
	\end{equation*}
	We have already shown that the RHS converges to zero as $\kappa\to\infty$ uniformly for $q(0)\in Q^*_T(\kappa)$, and so the claim follows.
\end{proof}

For the terms~\eqref{eq:diff flow 5} and~\eqref{eq:diff flow 6} we note that the expressions inside the curly brackets converge to zero in $H^{-1}$ by~\eqref{eq:diff flow 19} and \cref{thm:g conv 1,,thm:g conv 2}.  In fact, this is enough to show that the contributions of~\eqref{eq:diff flow 5} and~\eqref{eq:diff flow 6} converge to zero in $H^{-2}$ because the integral operator is bounded on $H^{-1}$:
\begin{lem}
	\label{thm:g bddness}
	There exists $k>0$ sufficiently large so that the operator
	\begin{equation*}
	h(x) \mapsto \frac{1}{g(x;k,q)^2} \int G(x,y;k,q) h(y) G(y,x;k,q) \dy
	\end{equation*}
	is bounded $H^{-1} \to H^1$ uniformly for $q\in B_A$.
\end{lem}
\begin{proof}
	First, we use the second estimate of~\eqref{eq:hskappa ests} to put the factors of $1/g(k,q)$ in $H^1$.  The remaining operator is easily estimated by duality, expanding $G(k,q)$ in a series, and using the estimate~\eqref{eq:hskappa identity 1}.
\end{proof}

The remaining two terms~\eqref{eq:diff flow 7} and~\eqref{eq:diff flow 8} are more delicate, because the derivative falling on $q$ obstructs convergence of curly-bracketed terms in $H^{-1}$.  First we consider~\eqref{eq:diff flow 7}.  Write
\begin{equation*}
\norm{ \eqref{eq:diff flow 7} }_{H^{-2}}
\lesssim \norm{ \int G(x,y;k,q) (F_\kappa q)'(y) G(y,x;k,q) \dy }_{H^{-1}}
\end{equation*}
for a function $F_\kappa$ which we know converges to $0$ in $W^{2,\infty}$ by~\eqref{eq:diff flow 19} and \cref{thm:g conv 1}.  To show that the RHS converges to zero as $\kappa\to\infty$, we exploit that the integrand and the Green's functions all contain $q$:
\begin{lem}
	\label{thm:g conv 3}
	If $F_\kappa \to 0$ in $W^{2,\infty}$ as $\kappa\to\infty$, then there exists $k>0$ sufficiently large so that
	\begin{equation*}
	\int G(x,y;k,q) (F_\kappa q)'(y) G(y,x;k,q) \dy \to 0
	\quad\tx{in }H^{-1}\tx{ as }\kappa\to\infty
	\end{equation*}
	uniformly for $q\in B_A$.
\end{lem}
\begin{proof}
	We estimate the integral by duality and maneuver the derivative onto $F_\kappa$ and the test function.  For $f\in H^1$ we have
	\begin{align*}
	&\int f(x) \int G(x,y;k,q) (F_\kappa q)'(y) G(y,x;k,q) \dy \\
	&= \sum_{\ell,m=0}^\infty (-1)^{\ell+m} \tr\{ f(R_0q)^\ell R_0 [\partial,F_\kappa q] R_0 (qR_0)^m \} \\
	&= \sum_{\ell,m=0}^\infty(-1)^{\ell+m} \big( \tr\{ f(R_0q)^\ell R_0 \partial F_\kappa (qR_0)^{m+1} \} - \tr\{ f(R_0q)^{\ell+1} F_\kappa R_0 \partial (qR_0)^{m} \} \big) ,
	\end{align*}
	where $R_0 = R_0(k)$.  The first factor of $fR_0$ prevents us from combining the two terms in the summand to create a commutator.  Instead, we pair the first term of the $(\ell,m)$ summand with the second term of the $(\ell-1,m+1)$ summand to create a commutator, and leave the first term of the $(0,m)$ summand and the second term of the $(\ell,0)$ summand:
	\begin{align}
	&\left| \int f(x) \int G(x,y;k,q) (F_\kappa q)'(y) G(y,x;k,q) \dy \right| 
	\nonumber \\
	&\leq \left| \sum_{n=1}^\infty (-1)^{n} \sum_{j=1}^n \tr\{ f(R_0q)^j [ R_0\partial, F_\kappa ] (qR_0)^{n+1-j} \} \right|
	\label{eq:diff flow 9} \\
	&\phantom{\leq{}}+ \left| \sum_{m=1}^\infty \tr\{ f R_0 \partial F_\kappa (qR_0)^{m+1} \} \right| + \left| \sum_{\ell=1}^\infty \tr\{ f(R_0q)^{\ell+1} F_\kappa R_0 \partial \} \right|
	\label{eq:diff flow 10} \\
	&\phantom{\leq{}}+ \left| \tr\{ f R_0 \partial F_\kappa qR_0 \} - \tr\{ fR_0qF_\kappa R_0 \partial \} \right| .
	\label{eq:diff flow 11}
	\end{align}
	
	For the first term~\eqref{eq:diff flow 9}, we write
	\begin{equation*}
	[R_{0} \partial, F_\kappa]
	= R_{0}[\partial, F_\kappa] + [R_{0}, F_\kappa] \partial
	= R_{0} F'_\kappa + R_{0} F''_\kappa \partial R_{0} + R_{0} 2 F'_\kappa \partial^{2} R_{0} .
	\end{equation*}
	Putting each factor of $\sqrt{R_0}q\sqrt{R_0}$ in $\I_2$ and pairing each $\partial$ with one copy of $\sqrt{R_0}$ in operator norm (cf.~\eqref{eq:hktilde alphadot 12}) we estimate
	\begin{equation*}
	\eqref{eq:diff flow 9}
	\lesssim \sum_{n=1}^\infty n k^{-2} \norm{f}_{L^\infty} \norm{ F_\kappa }_{W^{2,\infty}} \big( k^{-1/2} \norm{q}_{H^{-1}} \big)^{n+1}
	\lesssim \norm{f}_{H^1} A^2 \norm{ F_\kappa }_{W^{2,\infty}}
	\end{equation*}
	uniformly for $q\in B_A$ and $k$ sufficiently large.  The RHS converges to zero because $F_\kappa \to 0$ in $W^{2,\infty}$.
	
	Similarly, for the second term~\eqref{eq:diff flow 10} we have
	\begin{align*}
	\eqref{eq:diff flow 10}
	&\leq 2 \sum_{m=1}^\infty \norm{ f \sqrt{R_0} }_\op \norm{ \partial \sqrt{R_0} }_\op \norm{\sqrt{R_0}F_\kappa q \sqrt{R_0} }_{\I_2} \norm{ \sqrt{R_0} \, q \sqrt{R_0} }_{\I_2}^m \\
	&\lesssim \norm{f}_{H^1} A^2 \norm{F_\kappa}_{W^{2,\infty}}
	\end{align*}
	uniformly for $q\in B_A$ and $k$ sufficiently large, and again the RHS converges to zero.
	
	For the last term~\eqref{eq:diff flow 11}, we recombine the traces as a commutator to obtain
	\begin{equation*}
	\eqref{eq:diff flow 10}
	\lesssim \norm{ \sqrt{R_0} f' \sqrt{R_0} }_{\I_2} \norm{ \sqrt{R_0} F_\kappa q \sqrt{R_0} }_{\I_2}
	\lesssim \norm{f}_{H^1} A \norm{F_\kappa}_{W^{2,\infty}}
	\end{equation*}
	uniformly for $q\in B_A$ and $k$ sufficiently large.  Again, the RHS converges to zero by premise.
\end{proof}

To finish the proof of \cref{thm:diff flow conv 1}, we must show that the last term~\eqref{eq:diff flow 8} converges to zero in $H^{-2}$ as $\kappa\to\infty$.  As previously mentioned the convergence within the curly brackets occurs in $H^{-2}$, and now $q$ is concealed within another Green's function.  To overcome this, we use the commutativity relation trick~\eqref{eq:biham relation 2} used in \cref{thm:hktilde alphadot 1}.  Using the ODEs~\eqref{eq:g alg prop 3} for $g(\kappa,q+W)$ and $g(\kappa,W)$, applying the identity~\eqref{eq:g alg prop 2}, and regrouping terms, we have
\begin{align}
&\eqref{eq:diff flow 8}
\nonumber \\
&= - \tfrac{4\kappa^5}{(\kappa^2-k^2)^2} \left\{ \tfrac{W [ g(\kappa,q+W) - g(\kappa,W) ] }{g(k,q)} \right\}' 
\label{eq:diff flow 12}\\
&+ \tfrac{2\kappa^5}{(\kappa^2-k^2)^2} \tfrac{1}{g(k,q)^2} \int G(x,y) \big\{ (-W''' + 4k^2W') [g(\kappa,q+W) \!-\! g(\kappa,W)]
\label{eq:diff flow 13}\\
&\quad -3W'' \big[ g(\kappa,q+W) \!-\! g(\kappa,W) \big]' - 3W' \big[ g(\kappa,q+W) \!-\! g(\kappa,W) \big]''
\label{eq:diff flow 14}\\
&\quad - (W^2)' [g(\kappa,q+W) \!-\! g(\kappa,W)] - 4W^2\big[ g(\kappa,q+W) \!-\! g(\kappa,W) \big]^{\!\prime} \big\} G(y,x) 
\label{eq:diff flow 15}\\
&+ \tfrac{4\kappa^5}{(\kappa^2-k^2)^2} \tfrac{1}{g(k,q)^2} \int G(x,y) \big\{ W' q [g(\kappa,q+W) \!-\! g(\kappa,W)] \big\}(y) G(y,x) \dy
\label{eq:diff flow 16}\\
&+ \tfrac{1}{g(k,q)^2} \int G(x,y) \big\{ 2Vq' - \tfrac{4\kappa^5}{(\kappa^2-k^2)^2} W \big[ qg(\kappa,W) \big]' \big\}(y) G(y,x) \dy ,
\label{eq:diff flow 17}
\end{align}
where $G(x,y) = G(x,y;k,q)$.  Note that in~\eqref{eq:diff flow 17} we have isolated the term which cancels $2Vq'$.  We will show that each of the terms~\eqref{eq:diff flow 12}--\eqref{eq:diff flow 17} converge to zero.

The first term~\eqref{eq:diff flow 12} is easily estimated using the estimates~\eqref{eq:hskappa ests}:
\begin{equation*}
\norm{ \eqref{eq:diff flow 12} }_{H^{-2}}
\lesssim \norm{W}_{W^{1,\infty}} \norm{ \tfrac{1}{g(k,q)} }_{H^1} \tfrac{4\kappa^5}{(\kappa^2-k^2)^2} \norm{ g(\kappa,q+W) - g(\kappa,W) }_{H^{-1}} .
\end{equation*}
The RHS converges to zero as $\kappa\to\infty$ by \cref{thm:g conv 2}.  

For the contribution from~\eqref{eq:diff flow 13}, we first use \cref{thm:g bddness} to put the curly bracketed terms in $H^{-1}$:
\begin{equation*}
\norm{ \eqref{eq:diff flow 13} }_{H^{-2}}
\lesssim \norm{ W'}_{W^{3,\infty}} \tfrac{2\kappa^5}{(\kappa^2-k^2)^2} \norm{ g(\kappa,q+W) - g(\kappa,W) }_{H^{-1}} .
\end{equation*}
Again, the RHS converges to zero by \cref{thm:g conv 2}.  It was exactly to estimate this term that we required that $V_\kappa$ be in $W^{4,\infty}$ in \cref{thm:hyp}.

For the contributions~\eqref{eq:diff flow 14} and~\eqref{eq:diff flow 15} we again use \cref{thm:g bddness} to obtain
\begin{align*}
\norm{ \eqref{eq:diff flow 14} }_{H^{-2}}
&\lesssim \norm{ W'}_{W^{2,\infty}} \tfrac{2\kappa^5}{(\kappa^2-k^2)^2} \norm{ g(\kappa,q+W) - g(\kappa,W) }_{H^{1}} , \\
\norm{ \eqref{eq:diff flow 15} }_{H^{-2}}
&\lesssim \norm{ W}_{W^{2,\infty}}^2 \tfrac{2\kappa^5}{(\kappa^2-k^2)^2} \norm{ g(\kappa,q+W) - g(\kappa,W) }_{H^{1}} .
\end{align*}
These still converge to zero as $\kappa\to\infty$ by the equicontinuity of $Q$:
\begin{lem}
	\label{thm:g conv 4}
	We have
	\begin{equation*}
	\kappa [ g(\kappa,q+W) - g(\kappa,W) ] \to 0
	\quad\tx{in }H^{1}\tx{ as }\kappa\to\infty
	\end{equation*}
	uniformly for $q\in Q^*_T(\kappa)$.
\end{lem}
\begin{proof}
	By the diagonal Green's function estimate~\eqref{eq:diffeo prop} we have
	\begin{equation*}
	\kappa \norm{ g(\kappa,q+W) - g(\kappa,W) }_{H^1}
	\lesssim \norm{q}_{H^{-1}_{\kappa}} .
	\end{equation*}
	In the proof of \cref{thm:g conv 2} we used Gr\"onwall's inequality and equicontinuity to show that the RHS converges to zero uniformly for $q\in Q^*_T(\kappa)$.
\end{proof}

For the term~\eqref{eq:diff flow 16} we use \cref{thm:g bddness} and the estimates~\eqref{eq:hskappa ests} to put $q$ in $H^{-1}$:
\begin{equation*}
\norm{ \eqref{eq:diff flow 16} }_{H^{-2}}
\lesssim \norm{W'}_{W^{1,\infty}} A \tfrac{2\kappa^5}{(\kappa^2-k^2)^2} \norm{ g(\kappa,q+W) - g(\kappa,W) }_{H^1} , 
\end{equation*}
and the RHS converges to zero by \cref{thm:g conv 4}.

Finally, for the last term~\eqref{eq:diff flow 17} we write
\begin{align*}
\norm{ \eqref{eq:diff flow 17} }_{H^{-2}}
&\lesssim \norm{ \int G(x,y) \big\{ 2Vq' - \tfrac{4\kappa^5}{(\kappa^2-k^2)^2} W g(\kappa,W) q' \big\}(y) G(y,x) \dy }_{H^{-1}} \\
&\phantom{\lesssim{}}+ \norm{ \int G(x,y) \big\{ \tfrac{4\kappa^5}{(\kappa^2-k^2)^2} W g'(\kappa,W) q \big\}(y) G(y,x) \dy }_{H^{-1}} .
\end{align*}
The first term converges to zero by \cref{thm:g conv 3,,thm:g conv 1} due to the leading term $\tfrac{1}{2\kappa}$ of $g(\kappa,W)$.  The second term converges to zero by \cref{thm:g bddness,,thm:g conv 1}.  This completes the estimate of~\eqref{eq:diff flow 8}, and hence concludes the proof of \cref{thm:diff flow conv 1}.

\section{Well-posedness}
\label{sec:conv 2}

We are now equipped to prove that KdV with potential~\eqref{eq:tkdv} is globally well-posed in $H^{-1}$.  We begin by constructing solutions as the limit of the $\hktilde$ flow solutions as $\kappa \to \infty$.
\begin{thm}
	\label{thm:existence}
	Fix $V$ admissible and $T>0$.  Given initial data $q(0) \in H^{-1}(\R)$, the corresponding solutions $q_\kappa(t)$ to the $\hktilde$ flows~\eqref{eq:hktilde flow} are Cauchy in $C_tH^{-1} ([-T,T]\times\R)$ as $\kappa\to\infty$.  
	
	We define the limit $q(t) := \lim_{\kappa\to\infty} q_{\kappa}(t)$ in $C_tH^{-1}$ $([-T,T]\times\R)$ to be the $H^{-1}$ solution of~\eqref{eq:tkdv} with initial data $q(0)$.
\end{thm}
\begin{proof}
	In the following all spacetime norms will be taken over the slab $[-T,T]\times\R$.  \Cref{thm:hktilde alphadot 2} guarantees that there exists a constant $\kappa_0$ so that the $\hktilde$ flows $q_\kappa(t)$ are bounded in $H^{-1}$ uniformly for $|t|\leq T$ and $\kappa\geq \kappa_0$.
	
	We want to show that the difference $q_\kappa - q_\varkappa$ for $\varkappa \geq \kappa$ converges to zero as $\kappa\to\infty$.  As the $H_\kappa$ and $H_\varkappa$ flows commute (cf.~\eqref{eq:hk flow 2}), we may write the $H_\varkappa$ flow of $u$ by time $t$ as
	\begin{equation*}
	e^{tJ\nabla H_\varkappa} u = e^{tJ\nabla (H_\varkappa - H_\kappa)} e^{tJ\nabla H_\kappa} u .
	\end{equation*}
	We apply this identity to $u = q+V$ and $u=V$.  This allows us to write 
	\begin{equation*}
	q_\varkappa(t) = \wt{\Phi}_{\varkappa,\kappa,W}(t) q_\kappa(t) ,
	\end{equation*}
	where $\wt{\Phi}_{\varkappa,\kappa,W}(t)$ denotes the flow of~\eqref{eq:diff flow q} by time $t$ for the background wave obeying~\eqref{eq:diff flow V} with initial data $W(0) = V_\kappa(t)$.  We estimate
	\begin{equation}
	\norm{ q_\kappa - q_\varkappa }_{C_tH^{-1}}
	\leq \sup_{\substack{ q\in Q^*_{T}(\kappa) \\ W(0)\in V^*_T(\kappa)}}\ \sup_{\varkappa\geq \kappa}\ \snorm{ \wt{\Phi}_{\varkappa,\kappa,W}(t) q - q }_{C_tH^{-1}} ,
	\label{eq:diff flow 18}
	\end{equation}
	for the sets
	\begin{equation}
	\begin{aligned}
	Q^*_{T}(\kappa) &:= \{ \wt{\Phi}_\kappa(t) q : q\in Q,\ |t| \leq T \} , \\
	V^*_T(\kappa) &:= \{e^{tJ\nabla H_\varkappa}V(0) : |t|\leq T,\ \varkappa\geq\kappa\}
	\end{aligned}
	\label{eq:flowout}
	\end{equation}
	with $Q = \{q(0)\}$.  As $Q\subset H^{-1}$ is trivially bounded and equicontinuous, then the following more general fact will conclude the proof.  We allow $Q$ to be an arbitrary bounded and equicontinuous set so that we may reuse this fact in \cref{thm:wellposed 1}.
\end{proof}

\begin{prop}
	\label{thm:diff flow conv 2}
	Fix $V$ admissible, $T>0$, and a bounded and equicontinuous set $Q\subset H^{-1}$ of initial data.  Then the solutions $\wt{\Phi}_{\varkappa,\kappa,W}(t)q$ of the difference flow~\eqref{eq:diff flow q} with background wave $W(t)$ and initial data $q$ obey
	\begin{equation}
	\lim_{\kappa\to\infty}\, \sup_{\substack{ q\in Q^*_{T}(\kappa) \\ W(0)\in V^*_T(\kappa)}}\ \sup_{\varkappa\geq \kappa}\ \snorm{ \wt{\Phi}_{\varkappa,\kappa,W}(t) q - q }_{C_tH^{-1}([-T,T]\times\R)}  = 0 ,
	\label{eq:wellposed 2}
	\end{equation}
	where $Q^*_T(\kappa)$ and $V^*_T(\kappa)$ are defined in~\eqref{eq:flowout}.
\end{prop}
\begin{proof}
	For $q\in Q_T^*(\kappa)$ and $W(0)\in V^*_T(\kappa)$, let $q(t)$ denote the solution to the difference flow~\eqref{eq:diff flow q} with initial data $q$ and background wave $W(t)$.  As was introduced in~\cite{Killip2019}, the change variables $1/2g(k,q)$ in place of $q$ is convenient in witnessing this convergence.  Indeed, it suffices to show that under difference flow~\eqref{eq:diff flow q} we have
	\begin{equation*}
	\lim_{\kappa\to\infty}\, \sup_{\substack{ q\in Q^*_{T}(\kappa) \\ W(0)\in V^*_T(\kappa)}}\, \norm{ \frac{1}{2g(k,q(t))} - \frac{1}{2g(k,q)} }_{C_tH^{1}} = 0
	\end{equation*}
	for $k>0$ fixed sufficiently large, because then an application of the diffeomorphism property (\cref{thm:diffeo prop 2}) shows that this implies~\eqref{eq:wellposed 2}.

	Fix $\eps>0$.  We aim to show that
	\begin{equation}
	\norm{ \frac{1}{2g(k,q(t))} - \frac{1}{2g(k,q)} }_{C_tH^{1}} \leq \eps
	\label{eq:diff flow 0}
	\end{equation}
	for all $\kappa$ sufficiently large, uniformly for $\varkappa\geq \kappa$, $q\in Q^*_T(\kappa)$, and $W(0)\in V^*_T(\kappa)$.  In \cref{thm:diff flow conv 1} we already saw this convergence in $H^{-2}$, and now we will upgrade this convergence to $H^1$.  We do not rely on equicontinuity as in~\cite{Killip2019} because the presence of the background wave $V$ breaks the conservation of $\alpha$.  Instead, we will choose the parameters $\kappa$ and $\varkappa$ dependently.
	
	For $h\in\R$, we estimate the translation of $q$ by $h$ in $H^{-1}$ by truncating in Fourier variables at a large radius $K$:
	\begin{equation}
	\begin{aligned}
	&\norm{ q(t,x+h) - q(t,x) }^2_{H^{-1}_x}
	\lesssim \int_{\R} |e^{i h\xi}-1|^2 \frac{|\hat{q}(t,\xi)|^2}{\xi^2+4} \dxi \\
	&\leq  K^2h^2 \int_{|\xi|\leq K} \frac{|\hat{q}(t,\xi)|^2}{\xi^2+4} \dxi + \int_{|\xi|\geq K} \frac{|\hat{q}(t,\xi)|^2}{\xi^2+4} \dxi \\
	&\lesssim  K^2h^2 \norm{ q(t) }_{H^{-1}}^2 + \int_{\R} \frac{|\hat{q}(t,\xi)|^2}{\xi^2+4 K^2} \dxi
	\lesssim  K^2h^2 A^2 +  K\alpha( K,q(t)) ,
	\end{aligned}
	\label{eq:diff flow 1}
	\end{equation}
	where in the last inequality we used the $\alpha$ estimate~\eqref{eq:alpha 3}.  If we let $C$ denote the constant from \cref{thm:hktilde alphadot 1} then Gr\"onwall's inequality yields
	\begin{equation*}
	K\alpha( K,q(t)) \leq e^{2CT} K\alpha(K,q) \leq e^{3CT} K\alpha( K,q(0))
	\end{equation*}
	for all $|t|\leq T$ and $\kappa \geq 2K$, where $q(0) \in Q$ is the initial data of the $\hktilde$ flow $q\in Q_T^*(\kappa)$.  As the set $Q$ is equicontinuous, then $K\alpha( K,q(0)) \to 0$ as $K\to\infty$ uniformly for $q(0)\in Q$ by \cref{thm:hskappa equicty} and the $\alpha$ estimate~\eqref{eq:alpha 3}.  Therefore, given $\eta = \eta(\eps)>0$ small (to be chosen later) there exists $K$ sufficiently large so that
	\begin{equation*}
	 K \alpha( K,q(t)) \lesssim \eta
	\quad\tx{uniformly for }|t|\leq T,\ q\in Q^*_T(\kappa),\tx{ and }\kappa\geq 2 K.
	\end{equation*}
	Combining this with the estimate~\eqref{eq:diff flow 1} and optimizing in $h$, there is a value $h_0 \sim \eta^{1/2} K^{-1}$ such that
	\begin{equation}
	\norm{q(t,x+h) - q(t,x)}^2_{C_tH^{-1}} \lesssim \eta
	\quad\tx{for }|h|\leq h_0 .
	\label{eq:diff flow 2}
	\end{equation}
	
	We now will turn this control over the translates of $q$ into control of the Fourier tails of $1/g(k,q)$.  For $r\in\R$ we have
	\begin{equation*}
	\int |e^{i\xi h}-1|^2 re^{-2r|h|} \,dh
	= \frac{2\xi^2}{\xi^2+4r^2}
	\geq \begin{cases}
	\tfrac{2}{5} & |\xi|\geq r , \\ 0 & |\xi|<r .
	\end{cases}
	\end{equation*}
	Writing $\ft$ for the Fourier transform, this yields
	\begin{align*}
	 &\sup_{|t|\leq T}\, \int_{|\xi|\geq r} \left| \ft \left[ \frac{1}{2g(k,q(t))} \right](\xi) \right|^2 (\xi^2+1)\dxi \\
	 &\lesssim \int_{\R} \norm{ \frac{1}{2g(x+h;k,q(t))} - \frac{1}{2g(x;k,q(t))} }_{C_tH^1_x}^2 re^{-2r|h|} \,dh \\
	 &= \int_{\R} \norm{ \frac{1}{2g(x;k,q(t,\cdot+h))} - \frac{1}{2g(x;k,q(t,\cdot))} }_{C_tH^1_x}^2 re^{-2r|h|} \,dh \\
	 &\lesssim \int_{\R} \norm{ q(t,x+h) - q(t,x) }_{C_tH^{-1}_x}^2 re^{-2r|h|} \,dh
	 \intertext{by the diffeomorphism property (\cref{thm:diffeo prop 2}).  Splitting this integral into $|h|\leq h_0$ and $|h|\geq h_0$ and using~\eqref{eq:diff flow 2}, we obtain the upper bound}
	 &\lesssim \eta h_0 r + e^{-2rh_0} .
	 \intertext{Optimizing in $r$, we pick $r = r_0 := \tfrac{1}{2h_0} \log \tfrac{2}{\eta}$ to arrive at}
	 &\lesssim \eta\big( 1 + \log\tfrac{2}{\eta} \big) .
	 \end{align*}
	 
	 Finally, we employ this uniform control over the Fourier tails of $1/q(k,q)$ to upgrade our $H^{-2}$ convergence to $H^{1}$.  Separating $|\xi|\leq r_0$ and $|\xi|\geq r_0$, we have
	 \begin{equation}
	 \begin{aligned}
	 \norm{ \frac{1}{2g(k,q(t))} - \frac{1}{2g(k,q(0))} }_{H^1}^2
	 &\lesssim (r_0^2+1)^6 \norm{ \frac{1}{2g(k,q(t))} - \frac{1}{2g(k,q(0))} }_{H^{-2}}^2 \\
	 &\phantom{\lesssim{}}+ \sup_{|t|\leq T}\, \norm{ \frac{1}{2g(k,q(t))} }_{H^{1}(|\xi|\geq r_0)}^2 .
	 \end{aligned}
	 \label{eq:diff flow 3}
	 \end{equation}
	 We just saw that the second term of RHS\eqref{eq:diff flow 3} is $\lesssim \eta ( 1 + \log\tfrac{2}{\eta} )$, and picking $\eta = \eta(\eps)>0$ sufficiently small we can make this upper bound $\leq \tfrac{1}{2}\eps$.  With all other parameters determined, we then use \cref{thm:diff flow conv 1} to make the first term of RHS\eqref{eq:diff flow 3} less than $\tfrac{1}{2}\eps$ for all $\kappa$ sufficiently large.  This demonstrates~\eqref{eq:diff flow 0}, and hence concludes \cref{thm:diff flow conv 2} and \cref{thm:existence}.
\end{proof}

Applying the previous result to a different set $Q$, we also obtain uniform control over the limits $q(t)$ as we vary the initial data:
\begin{thm}
	\label{thm:wellposed 1}
	Fix $V$ admissible and $T>0$.  Given a convergent sequence $q_n(0) \in H^{-1}(\R)$ of initial data, the corresponding solutions $q_n(t)$ of KdV with potential~\eqref{eq:tkdv} constructed in \cref{thm:existence} are Cauchy in $C_tH^{-1}([-T,T]\times\R)$ as $n\to\infty$.
\end{thm}
\begin{proof}
	Consider the set $Q := \{ q_n(0) : n\in\N \}$ of initial data, which is bounded and equicontinuous in $H^{-1}$ since it is convergent in $H^{-1}$.  We estimate the difference $q_n(t) - q_m(t)$ using the triangle inequality, by first mediating via $\hvktilde$ flows and then estimating the difference between $\hktilde$ and $\hvktilde$ flows using~\eqref{eq:diff flow 18}.  This yields
	\begin{equation}
	\begin{aligned}
	\norm{ q_n(t) - q_m(t) }_{C_tH^{-1}}
	&\leq \snorm{ \wt{\Phi}_{\kappa}(t)q_n(0) - \wt{\Phi}_{\kappa}(t)q_m(0) }_{C_tH^{-1}} \\
	&\phantom{\leq{}} + 2 \sup_{\substack{ q\in Q^*_{T}(\kappa) \\ W(0)\in V^*_T(\kappa)}}\ \sup_{\varkappa\geq \kappa}\ \snorm{ \wt{\Phi}_{\varkappa,\kappa,W}(t) q - q }_{C_tH^{-1}} ,
	\end{aligned}
	\label{eq:wellposed 1}
	\end{equation}
	where $Q^*_{T}(\kappa)$ and $V^*_T(\kappa)$ are defined in~\eqref{eq:flowout}.  Fix $\eps>0$.  By \cref{thm:diff flow conv 2} there exists $\kappa_0$ sufficiently large so that the second term of \tx{RHS}\eqref{eq:wellposed 1} is $\leq \tfrac{1}{2}\eps$ for all $n,m \in \N$.  With $\kappa = \kappa_0$ fixed, we then know that the first term of \tx{RHS}\eqref{eq:wellposed 1} is $\leq \tfrac{1}{2} \eps$ for all $n,m$ sufficiently large due to the well-posedness of the $\hktilde$ flow (cf. \cref{thm:hktilde alphadot 2}).
\end{proof}

Finally, we use \cref{thm:wellposed 1} to conclude our main result \cref{thm:intro gwp}, following the argument of~\cite{Killip2019}.
\begin{cor}
	\label{thm:wellposed 2}
	Given $V$ admissible the KdV equation with potential $V$~\eqref{eq:tkdv} with initial data $q(0) \in H^{-1}(\R)$ is globally well-posed, in the sense that the solution map $\Phi: \R \times H^{-1}(\R) \to H^{-1}(\R)$ obtained in \cref{thm:existence} is jointly continuous.
\end{cor}
\begin{proof}
	Given $q \in H^{-1}(\R)$, we define $\Phi(t,q(0))$ to be the limit
	\begin{equation*}
	\Phi(t,q(0)) = \lim_{\kappa\to\infty} q_\kappa(t)
	\end{equation*}
	guaranteed by \cref{thm:existence}.  The limit exists in $H^{-1}(\R)$ and the convergence is uniform on bounded time intervals.  Fix $T>0$ and a sequence $q_n(0)\to q(0)$ in $H^{-1}(\R)$.  From \cref{thm:wellposed 1} we obtain
	\begin{equation*}
	\sup_{|t|\leq T}\, \norm{ \Phi(t,q_n(0)) - \Phi(t,q(0)) }_{H^{-1}} \to 0
	\quad\tx{as }n\to\infty,
	\end{equation*}
	and so we conclude that $\Phi$ is jointly continuous.
\end{proof}

\section{Uniqueness for regular initial data}
\label{sec:classical}

We will now demonstrate that for more regular initial data the solutions constructed in \cref{thm:existence} solve KdV and are unique.

First, we use a well-known $L^2$-energy argument to show that we have uniqueness for $H^2$ initial data.
\begin{lem}
	\label{thm:tkdv unique}
	Fix $V$ admissible and $T>0$.  Given an initial data $q(0) \in H^2$, there exists at most one corresponding solution to KdV with potential~\eqref{eq:tkdv} in $(C_tH^2 \cap C^1_tH^{-1})([-T,T]\times\R)$.
\end{lem}
\begin{proof}
	Suppose $q(t)$ and $\tilde{q}(t)$ are both in $(C_tH^2 \cap C^1_tH^{-1})([-T,T]\times\R)$, solve  KdV with potential, and have the same initial data $q(0) = \tilde{q}(0)$.  From the differential equation~\eqref{eq:tkdv} we see that the $L^2$-norm of the difference grows according to
	\begin{align*}
	\left| \ddt \int \tfrac{1}{2} (q - \tilde{q})^2\dx \right|
	&= \left| \int (q-\tilde{q}) \{ - (q-\tilde{q})''' + 3(q^2 - \tilde{q}^2)' + [6V(q-\tilde{q})]' \} \dx \right| .
	\intertext{The first term $(q-\tilde{q})'''$ contributes a total derivative and vanishes, while the remaining terms can be integrated by parts to obtain}
	&= \left| \int (q-\tilde{q})^2 \{ \tfrac{3}{2} (q + \tilde{q})' + 3V' \}(t,x) \dx \right| \\
	&\leq \big( \tfrac{3}{2} \norm{q'}_{L^\infty} + \tfrac{3}{2} \norm{\tilde{q}'}_{L^\infty} + 3 \norm{V'}_{L^\infty} \big) \norm{ q-\tilde{q} }_{L^2}^2 .
	\end{align*}
	Estimating $\norm{q'}_{L^\infty} \lesssim \norm{q}_{H^2}$, $\norm{\tilde{q}'}_{L^\infty} \lesssim \norm{\tilde{q}}_{H^2}$, and recalling that $V\in W^{2,\infty}$ uniformly for $|t|\leq T$ by \cref{thm:hyp}, we conclude that there exists a constant $C$ (depending on $V$ and the norms of $q$ and $\tilde{q}$ in $C_tH^2([-T,T]\times\R)$) such that 
	\begin{equation*}
	\left| \ddt \norm{ q(t)-\tilde{q}(t) }_{L^2}^2 \right| \leq C \norm{ q(t) -\tilde{q}(t) }_{L^2}^2 .
	\end{equation*}
	Gr\"onwall's inequality then yields
	\begin{equation*}
	\norm{ q(t) -\tilde{q}(t) }_{L^2}^2 \leq e^{CT} \norm{ q(0) - \tilde{q}(0) }_{L^2}^2
	\end{equation*}
	for $|t|\leq T$.  The RHS vanishes by premise, and so we conclude $\tilde{q}(t) = q(t)$ for all $|t|\leq T$.
\end{proof}

In order to employ the uniqueness of \cref{thm:tkdv unique}, we need to first show that the limits of \cref{thm:existence} are in $C_tH^2 \cap C^1_tH^{-1}$ and solve KdV with potential.  The following proposition shows that it suffices to know that the sequence $q_\kappa(t)$ of $\hktilde$ flows converges in $C_tH^2$:
\begin{prop}
	\label{thm:classical flow conv}
	Fix $V$ admissible and $T>0$.  If the sequence $q_\kappa(t)$ of solutions to the $\hktilde$ flow~\eqref{eq:hktilde flow} converges in $C_tH^2([-T,T]\times\R)$ as $\kappa\to\infty$, then the limit $q(t)$ is in $(C_tH^2 \cap C^1_tH^{-1})([-T,T]\times\R)$ and solves KdV with potential~\eqref{eq:tkdv}.
\end{prop}
\begin{proof}
	In the following all spacetime norms will be taken over the slab $[-T,T]\times\R$.  We will extract the linear and quadratic terms of the $\hktilde$ flow to witness its convergence to KdV with potential.  Differentiating the translation identity~\eqref{eq:g trans prop} for $g(\kappa,q_\kappa+V_\kappa) - g(\kappa,V_\kappa)$ at $h=0$, expanding it as a series in powers of $q$, and then expanding each resolvent $R(\kappa,V_\kappa)$ in powers of $V_\kappa$, we write
	\begin{align}
	&\ddt q_\kappa 
	\nonumber \\
	&= - 16\kappa^5 \langle \del_x , R_0 q'_\kappa R_0 \del_x \rangle + 4\kappa^2 q'_\kappa 
	\label{eq:classical conv 1}\\
	&\phantom{={}}+ 16\kappa^5 \langle \del_x , [\partial, R_0 q_\kappa R_0 q_\kappa R_0] \del_x \rangle
	\label{eq:classical conv 2}\\
	&\phantom{={}}+ 16\kappa^5 \big\{ \langle \del_x , [\partial, R_0 V_\kappa R_0 q_\kappa R_0] \del_x \rangle +  \langle \del_x , [\partial, R_0 q_\kappa R_0 V_\kappa R_0] \del_x \rangle \big\}
	\label{eq:classical conv 3}\\
	&\phantom{={}} + 16\kappa^5 \sum (\tx{terms with 3 or more }q_\kappa\tx{ or }V_\kappa) .
	\label{eq:classical conv 4}
	\end{align}
	We will show that the first three terms \eqref{eq:classical conv 1}--\eqref{eq:classical conv 3} converge to the three terms of KdV with potential~\eqref{eq:tkdv} respectively, and the tail~\eqref{eq:classical conv 4} converges to zero as $\kappa\to\infty$.
	
	We begin with the linear term~\eqref{eq:classical conv 1}.  Using the operator identity~\eqref{eq:linear op id} we write
	\begin{equation*}
	\eqref{eq:classical conv 1}
	= -q'''_\kappa  - R_0(2\kappa)\partial^2 (q_\kappa-q)''' - R_0(2\kappa) \partial^2 q''' .
	\end{equation*}
	As $q_\kappa\to q$ in $C_tH^2$ by premise, then the first term of the RHS above converges to $-q'''$ in $C_tH^{-1}$ and the second term converges to zero in $C_tH^{-1}$ because the operators $R_0(2\kappa)\partial^2$ are bounded uniformly in $\kappa$.  The last term converges to zero since the operator $R_0(2\kappa) \partial^2$ is readily seen in Fourier variables to converge strongly to zero as $\kappa\to\infty$.  Altogether we conclude
	\begin{equation*}
	\eqref{eq:classical conv 1} \to - q''' \quad\tx{in }C_tH^{-1}\tx{ as }\kappa\to\infty.
	\end{equation*}
	
	Next, we turn to the first quadratic term~\eqref{eq:classical conv 2}.  First we write
	\begin{equation*}
	\eqref{eq:classical conv 2}
	= 6q_\kappa q'_\kappa +  \big\{ 16\kappa^5 \langle \del_x , [\partial, R_0 q_\kappa R_0 q_\kappa R_0] \del_x \rangle - 6q_\kappa q'_\kappa \big\} .
	\end{equation*}
	As $q_\kappa\to q$ in $C_tH^2$ by premise, then the first term of the RHS above converges to $6qq'$ in $C_tH^{1}$ and hence in $C_tH^{-1}$ as well.  For the second term we distribute the derivative $[\partial,\cdot]$, use the operator identity~\eqref{eq:quad op id}, and estimate in $H^{-1}$ by duality.  For $\phi\in H^1$, the identity~\eqref{eq:quad op id} yields
	\begin{align*}
	&\left| \int \big\{ 16\kappa^5 \langle \del_x , R_0 q_\kappa R_0 q'_\kappa R_0 \del_x \rangle - 3q_\kappa q'_\kappa \big\} \phi \dx \right| = \\
	&\begin{aligned}\bigg| \int \big\{ {-3} [R_0(2\kappa)q''_\kappa][R_0(2\kappa)q'''_\kappa] \phi
	+ 4\kappa^2 [R_0(2\kappa)q'_\kappa][R_0(2\kappa)q''_\kappa] ( {-5} \phi + R_0(2\kappa) \phi'' ) &\\
	+ 4\kappa^2 [R_0(2\kappa)q_\kappa][R_0(2\kappa)q'_\kappa] (5 \phi'' + 2 R_0(2\kappa) \partial^2 \phi'') \big\} \dx \bigg| . & \end{aligned}
	\end{align*}
	For those terms with $\phi''$ we integrate by parts once to obtain $\phi'$, which can be put in $L^2$.  Putting the highest order $q_\kappa$ term in $L^2$, putting one term in $L^\infty \supset H^1$, and using $\snorm{R_0(2\kappa)\partial^j}_\op \lesssim \kappa^{j-2}$ for $j=0,1,2$ (the estimate for $j=0$ is also true as an operator on $L^\infty$ by the explicit kernel formula for $R_0$ and Young's inequality), we obtain
	\begin{equation*}
	\left| \int \big\{ 16\kappa^5 \langle \del_x , R_0 q_\kappa R_0 q'_\kappa R_0 \del_x \rangle - 3q_\kappa q'_\kappa \big\} \phi \dx \right|
	\lesssim \kappa^{-2} \norm{\phi}_{H^1} \norm{q_\kappa}_{H^2}^2 .
	\end{equation*}
	Taking a supremum over $\norm{\phi}_{H^1}\leq 1$, and noting that the other term from the product rule is handled analogously (indeed, the identity~\eqref{eq:quad op id} is symmetric in $f$ and $h$), we conclude
	\begin{equation*}
	\eqref{eq:classical conv 2} \to 6qq' \quad\tx{in }C_tH^{-1}\tx{ as }\kappa\to\infty.
	\end{equation*}
	
	The second quadratic term~\eqref{eq:classical conv 3} is similar, but now we must put $V_\kappa$ in $L^\infty$.  First we write
	\begin{align*}
	\eqref{eq:classical conv 3}
	= 6(V_\kappa q_\kappa)' +  \big\{ &16\kappa^5 \langle \del_x , [\partial, R_0 V_\kappa R_0 q_\kappa R_0] \del_x \rangle \\ 
	&+ 16\kappa^5 \langle \del_x , [\partial, R_0 q_\kappa R_0 V_\kappa R_0] \del_x \rangle - 6(V_\kappa q_\kappa)' \big\} .
	\end{align*}
	As $V_\kappa \to V$ in $W^{2,\infty}$, the first term of the RHS above converges to $6(Vq)'$ in $C_tH^{1}$ and hence in $C_tH^{-1}$ as well.  For the second term we distribute the two derivatives $[\partial,\cdot]$ to get four terms, use the operator identity~\eqref{eq:quad op id}, and then estimate in $H^{-1}$ by duality.  For example, for $\phi\in H^1$ we have
	\begin{align*}
	&\left| \int \big\{ 16\kappa^5 \langle \del_x , R_0 V_\kappa R_0 q'_\kappa R_0 \del_x \rangle - 3V_\kappa q'_\kappa \big\} \phi \dx \right| = \\
	&\begin{aligned}\bigg| \int \big\{ {-3} [R_0(2\kappa)V''_\kappa][R_0(2\kappa)q'''_\kappa] \phi
	+ 4\kappa^2 [R_0(2\kappa)V'_\kappa][R_0(2\kappa)q''_\kappa] ( {-5} \phi + R_0(2\kappa) \phi'' ) &\\
	+ 4\kappa^2 [R_0(2\kappa)V_\kappa][R_0(2\kappa)q'_\kappa] (5 \phi'' + 2 R_0(2\kappa) \partial^2 \phi'') \big\} \dx \bigg| . & \end{aligned}
	\end{align*}
	For those terms with $\phi''$ we integrate by parts once to obtain $\phi'$, which can be put in $L^2$.  Putting all $V_\kappa$ terms in $L^\infty$ and the remaining terms in $L^2$, we obtain
	\begin{equation*}
	\left| \int \big\{ 16\kappa^5 \langle \del_x , R_0 V_\kappa R_0 q'_\kappa R_0 \del_x \rangle - 3q_\kappa q'_\kappa \big\} \phi \dx \right|
	\lesssim \kappa^{-2} \norm{\phi}_{H^1} \norm{V_\kappa}_{W^{2,\infty}} \norm{q_\kappa}_{H^2} .
	\end{equation*}
	The other three terms obtained from the product rule are handled analogously; replacing $q'_\kappa$ by $q_\kappa$ and $V_\kappa$ by $V'_\kappa$ is harmless because we know that $V_\kappa \in W^{4,\infty}$ uniformly for $|t|\leq T$ and $\kappa$ large.  Taking a supremum over $\norm{\phi}_{H^1}\leq 1$, we conclude
	\begin{equation*}
	\eqref{eq:classical conv 3} \to 6(Vq)' \quad\tx{in }C_tH^{-1}\tx{ as }\kappa\to\infty.
	\end{equation*}

	Lastly, we show that the series tail~\eqref{eq:classical conv 4} converges to zero in $C_tH^{-1}$.  We estimate by duality; for $\phi\in H^1$ we write
	\begin{align*}
	&\left| \int \phi \cdot \eqref{eq:classical conv 4} \dx \right| \\
	&\leq 16\kappa^5 \sum_{\substack{ \ell\geq 1,\ m_0,\dots,m_\ell\geq 0 \\ \ell+m_0+\dots+m_\ell \geq 3 }} \big| \tr \big\{ \phi [\partial, R_0(V_\kappa R_0)^{m_0} q_\kappa R_0 \cdots q_\kappa R_0 (V_\kappa R_0)^{m_\ell} ] \big\} \big| .
	\intertext{Recall that we first expanded $g(\kappa,q_\kappa+V_\kappa)$ in powers of $q_\kappa$, the $\ell$th term having $\ell$-many factors of $q_\kappa R(\kappa,V_\kappa)$, and then expanded each $R(\kappa,V_\kappa)$ into a series in $V_\kappa$ indexed by $m_i$.  The condition $\ell+m_0+\dots+m_\ell \geq 3$ reflects that we have already accounted for all of the summands with one and two $q_\kappa$ or $V_\kappa$.  We distribute the derivative $[\partial,\cdot]$, use the estimate~\eqref{eq:hskappa identity 1} and the observation $\norm{ f }_{H^{-1}_\kappa} \lesssim \kappa^{-1} \norm{ f }_{L^2}$ to put $\phi$ and all copies of $q_\kappa$ in $L^2$, and then estimate $V_\kappa$ in operator norm to obtain}
	&\lesssim \kappa^5 \sum_{\substack{ \ell\geq 1,\ m_0,\dots,m_\ell\geq 0 \\ \ell+m_0+\dots+m_\ell \geq 3 }} \frac{\norm{\phi}_{L^2}}{\kappa^{3/2}} \left( \frac{\norm{q_\kappa}_{H^1}}{\kappa^{3/2}} \right)^\ell \left( \frac{\norm{V_\kappa}_{W^{1,\infty}}}{\kappa^2} \right)^{m_0 + \dots + m_\ell} .
	\intertext{We first sum over the indices $m_0,\dots,m_\ell \geq 0$ as we did in~\eqref{eq:double sum inner} using that $V_\kappa \in C_tW^{4,\infty}$ uniformly for $\kappa$ large.  Then we sum over $\ell\geq 1$ and use that $q_\kappa$ is bounded in $C_tH^2$ for $\kappa$ sufficiently large.  The condition $\ell+m_0+\dots+m_\ell \geq 3$ guarantees that summing over the two pararenthetical terms yields a gain $\lesssim (\kappa^{-3/2})^3$, from which we obtain}
	&\lesssim \kappa^{-1} \norm{\phi}_{H^1} .
	\end{align*}
	Taking a supremum over $\norm{\phi}_{H^1} \leq 1$, we conclude
	\begin{equation*}
	\eqref{eq:classical conv 4} \to 0 \quad\tx{in }C_tH^{-1}\tx{ as }\kappa\to\infty. \qedhere
	\end{equation*}
\end{proof}

It only remains to show that the sequence $q_\kappa(t)$ converges in $H^2$ as $\kappa\to\infty$.  To accomplish this task, we will use the uniform boundedness of the sequence $q_\kappa(t)$ in $H^3$ which follows from a simple \ti{a priori} estimate; this suffices by interpolating with the convergence in $H^{-1}$.
\begin{prop}
	\label{thm:classical a priori}
	Given $V$ admissible and $A,T>0$, there exist constants $C$ and $\kappa_0$ such that solutions $q_\kappa(t)$ to the $\hktilde$ flow~\eqref{eq:hktilde flow} obey
	\begin{equation*}
	\norm{q(0)}_{H^3} \leq A \quad \Longrightarrow \quad \norm{q_{\kappa}(t)}_{H^3} \leq C\tx{ for all }|t| \leq T\tx{ and }\kappa \geq \kappa_{0} .
	\end{equation*}
\end{prop}

The proof is a repetition of the energy arguments that yield the \ti{a priori} estimates in $H^s$ necessary for the Bona--Smith theorem~\cite{Bona1975} applied to the $\hktilde$ flow.  It is based on the fact that the $H_\kappa$ flow preserves the polynomial conservation laws of KdV.  For $s=0,1,2$ we control the growth of the first three conserved quantities in time (which are no longer exactly conserved for the $\hktilde$ flow), and then for $s = 3$ we directly control the growth of $q'''_\kappa$ in $L^2$.  See~\cite{Laurens2021}*{\S3} for details, where for a similar differential equation we obtain \ti{a priori} estimates in $H^s$ spaces for all integers $s\geq 0$.

It is natural to ask if for initial data in $H^3$ we have convergence in $H^3$ and not merely $H^2$.  This is also true, but the argument is more subtle.  In the companion paper~\cite{Laurens2021} we will present a more thorough argument for a similar equation, which can be adapted to this context to directly show convergence in $H^3$.

Altogether, we can now conclude our main result \cref{thm:intro classical solns}:
\begin{cor}
	\label{thm:classical solns}
	Fix $V$ admissible and $T>0$.   Given initial data $q(0) \in H^{3}$, the solution constructed in \cref{thm:existence} is the unique solution to KdV with potential~\eqref{eq:tkdv} in $(C_tH^2 \cap C^1_tH^{-1})([-T,T]\times\R)$.
\end{cor}
\begin{proof}
	We know from \cref{thm:existence} that the $\hktilde$ flows $q_\kappa(t)$ converge in $H^{-1}$ as $\kappa\to\infty$, and from \cref{thm:classical a priori} we know they are bounded in $C_tH^3([-T,T]\times\R)$ uniformly for $\kappa$ large.  From the inequality
	\begin{equation*}
	\norm{f}_{H^2} \leq \norm{f}^{1/4}_{H^{-1}} \norm{f}^{3/4}_{H^3} 
	\end{equation*}
	(which can be obtained using H\"older's inequality in Fourier variables), we deduce that $q_\kappa$ converges in $C_tH^2([-T,T]\times\R)$ as well.  \Cref{thm:classical flow conv} then tells us that the limit $q(t)$ is in $(C_tH^2\cap C^1_tH^{-1})([-T,T]\times\R)$ and solves KdV with potential.  Finally, \cref{thm:tkdv unique} guarantees that this is the unique solution in this class.
\end{proof}

\section{Example: cnoidal waves}
\label{sec:cnoidal}

Next, we will see that the periodic traveling wave solutions (cnoidal waves) of KdV are admissible background waves $V$ in the sense of \cref{thm:hyp}.  In fact, we will see that cnoidal wave profiles $V(0,x)$ are also traveling waves for the $H_\kappa$ flow~\eqref{eq:hk flow} (with a different propagation speed), which makes the analysis particularly straightforward.  The $H_\kappa$ flow possessing the same traveling wave profile as KdV is not surprising, since the $H_\kappa$ flow preserves the polynomial conserved quantities of KdV and cnoidal waves are minimizers of the KdV energy with constrained momentum (cf.~\cite{Lax1975}*{\S3}).

Rather than working with the Jacobian elliptic functions, it is much easier to perform calculus on the cnoidal waves~\eqref{eq:cnoidal 1} when expressed in terms of Weierstrass elliptic functions:
\begin{equation}
V(t,x) = 2\wp\left( x + 6\wp(\omega_1) t + \omega_3 ; \omega_1,\omega_3 \right) + \wp(\omega_1) .
\label{eq:cnoidal 2}
\end{equation}
Here, $\wp( z ; \omega_1,\omega_3 ) =: \wp(z)$ is the Weierstrass p-function with lattice generators $2\omega_1,2\omega_3 \in \C$ (see~\cite{DLMF}*{\href{http://dlmf.nist.gov/23.2}{\S23.2}} for its definition).  We must choose $\omega_1$ purely real and $\omega_3$ purely imaginary for the wave~\eqref{eq:cnoidal 2} to solve KdV, and to avoid redundancy we insist that $\omega_1$ and $\omega_3/i$ are positive.  Note that the argument $z$ of $\wp( z )$ in~\eqref{eq:cnoidal 2} is not on the real axis but is translated vertically by the imaginary half-period $\omega_3$ and thus runs halfway between two rows of poles for $\wp( z )$; this guarantees that the profile~\eqref{eq:cnoidal 2} is regular and real-valued.

\begin{prop}
	\label{thm:hk cnoidal 1}
	The cnoidal wave profile admits the traveling wave solution
	\begin{equation*}
	V_\kappa(t,x) = V(0,x+\nu t) , \qquad
	\nu = \nu(\kappa)
	\end{equation*}
	to the $H_\kappa$ flow~\eqref{eq:hk flow}.
\end{prop}
\begin{proof}
	Let $V(x) = V(0,x)$ denote the initial data.  In order to see that $V(x+\nu t)$ solves the $H_\kappa$ flow~\eqref{eq:hk flow} we need $g'(x;\kappa,V)$ to be proportional to $V(x)$.  To compute the diagonal Green's function $g(x;\kappa,V)$, we will use the representation
	\begin{equation}
	g(x) = \psi_+(x) \psi_-(x)
	\label{eq:floquet 1}
	\end{equation}
	in terms of normalized Floquet solutions $\psi_\pm$.  Recall from Floquet theory that there exist solutions $\psi_\pm(x)$ to
	\begin{equation}
	-\psi''+V\psi = - \kappa^2\psi
	\label{eq:floquet 2}
	\end{equation}
	who decay exponentially (along with their derivatives) as $x\to\pm\infty$ and grow exponentially as $x\to\mp\infty$.  Constancy of the Wronskian guarantees that these solutions are unique up to scalar multiples.  For the expression~\eqref{eq:floquet 1} to hold, we partially normalize the solutions $\psi_\pm$ by enforcing the Wronskian relation
	\begin{equation}
	\psi_+(x)\psi'_-(x) - \psi'_+(x)\psi_-(x) = 1
	\label{eq:floquet 3}
	\end{equation}
	and requiring that both $\psi_\pm$ are positive.
	
	Consider the ansatz
	\begin{equation}
	\psi_\pm(x) = a_\pm \frac{\sigma(x+\omega_3\pm b)}{\sigma(x+\omega_3)\sigma(\pm b)} e^{\mp\zeta(b) x} ,
	\label{eq:floquet 4}
	\end{equation}
	where $\sigma(z)$ and $\zeta(z)$ are the other two Weierstrass elliptic functions with the same lattice generators $\omega_1,\omega_3$ as $V$ (see~\cite{DLMF}*{\href{http://dlmf.nist.gov/23.2}{\S23.2(ii)}} for their definition and relations), and $a_\pm$ and $b$ are parameters to be chosen depending on $\kappa$.  Substituting the ansatz~\eqref{eq:floquet 4} into the eigenvalue equation~\eqref{eq:floquet 2} and using the additive identities~\cite{DLMF}*{\href{http://dlmf.nist.gov/23.10}{\S23.10(i)}}, we see that~\eqref{eq:floquet 4} solves~\eqref{eq:floquet 2} provided that $b = b(\kappa)$ satisfies
	\begin{equation}
	\kappa^2 = \wp(b) - \wp(\omega_1) .
	\label{eq:floquet 5}
	\end{equation}
	As $\wp(x)$ is real, positive, and symmetrically U-shaped for $x\in (0,2\omega_1)$, we see that in order to have $\kappa \in (0,\infty)$ we can take $b \in (0,\omega_1)$, with $b(\kappa)\downarrow 0$ as $\kappa\to\infty$.  To ensure the ansatz~\eqref{eq:floquet 4} satisfies the Wronskian relation~\eqref{eq:floquet 3} and the condition $\psi_\pm(x) > 0$, we set
	\begin{equation*}
	a_\pm = \pm \left[ - \wp'(b) \right]^{-\frac{1}{2}} .
	\end{equation*}
	As $\wp(b)$ is real, positive, and strictly decreasing for $b\in(0,\omega_1)$, then $- \wp'(b)$ is positive and we may take the positive square-root.  Although it is incidental to the proof, we note that the Floquet exponents for $\psi_\pm$ are
	\begin{equation*}
	\frac{\psi_\pm(x+2\omega_1)}{\psi_\pm(x)} = e^{\mp 2\omega_1\zeta(b)} ,
	\end{equation*}
	and they are multiplicative inverses of each other (as expected from Floquet theory).
	
	Now that we have determined the Floquet solutions~\eqref{eq:floquet 4}, the representation~\eqref{eq:floquet 1} determines the diagonal Green's function:
	\begin{equation}
	g(x;\kappa,V)
	= \frac{\wp(b(\kappa)) - \wp(x+\omega_3)}{-\wp'(b(\kappa))}
	= \frac{\wp(b(\kappa)) + \tfrac{1}{2}\wp(\omega_1)}{-\wp'(b(\kappa))} + \frac{1}{2\wp'(b(\kappa))} V(x) .
	\label{eq:cnoidal g 1}
	\end{equation}
	We notice in particular that $g'(x;\kappa,V)$ is proportional to $V'(x)$.  Recalling the translation property~\eqref{eq:g trans prop}, we conclude that the solution $V_\kappa(t,x)$ to the $H_\kappa$ flow~\eqref{eq:hk flow} with initial data $V(0,x)$ is the traveling wave $V(0,x+\nu t)$.  Moreover, the propagation speed is given by
	\begin{equation}
	\nu(\kappa) = \frac{8\kappa^5}{\wp'(b(\kappa))} + 4\kappa^2
	\label{eq:hk cnoidal 2}
	\end{equation}
	for all $\kappa$ sufficiently large.
\end{proof}

To see the convergence of $V_\kappa$ to $V$, we will first need to take a slightly closer look at the exact form of the coefficients.
\begin{lem}
	The diagonal Green's function for the traveling waves $V_\kappa$ takes the form
	\begin{equation*}
	g(x;\kappa,V_\kappa(t)) = c_1(\kappa) + c_2(\kappa) V_\kappa(t,x) ,
	\end{equation*}
	where the coefficients have the asymptotics
	\begin{equation}
	c_1(\kappa) = \tfrac{1}{2\kappa} + \bigo( \kappa^{-5} ) , \quad
	c_2(\kappa) = -\tfrac{1}{4\kappa^3} + \bigo( \kappa^{-5} )
	\quad \tx{as }\kappa\to\infty.
	\label{eq:asymp 1}
	\end{equation}
\end{lem}

The asymptotics~\eqref{eq:asymp 1} are consistent with the convergence found in \cref{thm:g conv 1}.  In fact, for cnoidal waves, \cref{thm:g conv 1} follows immediately from~\eqref{eq:asymp 1} and the fundamental theorem of calculus.

\begin{proof}
	From the expression~\eqref{eq:cnoidal g 1} for the diagonal Green's function $g(\kappa,V_\kappa)$ we have
	\begin{equation}
	c_1(\kappa) = \frac{\wp(b(\kappa)) + \tfrac{1}{2}\wp(\omega_1)}{-\wp'(b(\kappa))} , \qquad
	c_2(\kappa) = \frac{1}{2\wp'(b)} ,
	\label{eq:asymp 2}
	\end{equation}
	where $b=b(\kappa)$ is defined by the relation~\eqref{eq:floquet 5}.  As $\wp'(b)$ is nonvanishing for $b\in (0,\omega_1)$ and the p-function possesses the Laurent expansion~\cite{DLMF}*{\href{http://dlmf.nist.gov/23.9.E2}{Eq.~23.9.2}}
	\begin{equation}
	\wp(z;\omega_1,\omega_2) = \tfrac{1}{z^{2}} + \bigo( z^2 ) \quad \tx{for }0<|z|<\min\{|\omega_1|,|\omega_3|\} ,
	\label{eq:asymp 3}
	\end{equation}
	then the inverse function theorem guarantees that $b(\kappa)$ is an analytic function at $\kappa=+\infty$.  Combining the Laurent expansion~\eqref{eq:asymp 3} with the defining relation~\eqref{eq:floquet 5} for $b(\kappa)$, we can solve for the first few coefficients in the expansion for $b(\kappa)$:
	\begin{equation}
	b(\kappa) = \tfrac{1}{\kappa} +\bigo(\kappa^{-5}) .
	\label{eq:asymp 4}
	\end{equation}
	This combined with the coefficient formulas~\eqref{eq:asymp 2} yields the asymptotics~\eqref{eq:asymp 1}.
\end{proof}

Altogether, we conclude that cnoidal waves are admissible:
\begin{cor}
	If $V$ is a periodic traveling wave solution~\eqref{eq:cnoidal 1} of KdV, then the KdV equation~\eqref{eq:kdv} with initial data $u(0) \in V(0) + H^{-1}(\R)$ is globally well-posed.
\end{cor}
\begin{proof}
	In order to apply \cref{thm:wellposed 2} we must check that $V$ satisfies the criteria of \cref{thm:hyp}.  It only remains to show that $V_\kappa - V \to 0$ in $W^{2,\infty}$ as $\kappa\to\infty$ uniformly for initial data in $\{ V_\varkappa(t) : |t|\leq T,\ \varkappa\geq \kappa\}$.  By the fundamental theorem of calculus it suffices to show that the wave speed $\nu(\kappa)$ converges to that of the KdV traveling waves~\eqref{eq:cnoidal 2}.  Indeed, the expression~\eqref{eq:hk cnoidal 2} for $\nu(\kappa)$ combined with the asymptotics~\eqref{eq:asymp 3} and~\eqref{eq:asymp 4} yields $\nu(\kappa) \to 6\wp(\omega_1)$ as $\kappa\to\infty$, which is the propagation speed for the KdV traveling waves~\eqref{eq:cnoidal 2}.
\end{proof}

\section{Example: smooth periodic waves}
\label{sec:periodic}

The purpose of this section is to show that any $V(0,x) \in H^5(\T)$ (where $\T = \R/\Z$ denotes the circle) is admissible in the sense of \cref{thm:hyp}.  The proof consists of an energy argument in the spirit of Bona--Smith~\cite{Bona1975}.

Our convention for the Fourier transform of functions on the circle $\T$ is
\begin{equation*}
\hat{f}(\xi) = \int_0^1 e^{-i\xi x}f(x)\dx,
\quad\tx{so that}\quad
f(x) = \sum_{\xi\in 2\pi\Z} \hat{f}(\xi) e^{i\xi x} .
\end{equation*}
As with functions on the line, we also define the norm
\begin{equation*}
\norm{f}^2_{H^s_\kappa(\T)} = \sum_{\xi\in 2\pi\Z} (\xi^2 + 4\kappa^2)^s |\hat{f}(\xi)|^2 .
\end{equation*}

The Schr\"odinger operator $-\partial^2 + q$ from which we built the diagonal Green's function $g(x;\kappa,q)$ acts on $L^2(\R)$ and not $L^2(\T)$.  Consequently, for potentials $q$ on the circle this operator is no longer a relatively Hilbert--Schmidt (or even relatively compact) perturbation of the case $q\equiv 0$.  In place of the fundamental estimate~\eqref{eq:hskappa identity 1}, we will use the following two operator estimates from~\cite{Killip2019}*{Lem.~6.1}:
\begin{align}
\norm{ \sqrt{R_0}\, q\sqrt{R_0} }_\op
&\lesssim \kappa^{-1/2} \norm{q}_{H^{-1}_\kappa(\T)} ,
\label{eq:hskappa identity 3}\\
\norm{ \sqrt{R_0}f\psi R_0 q \sqrt{R_0} }_{\I_1}
&\lesssim \kappa^{-1} \norm{f_\kappa}_{H^{-1}(\T)} \norm{q}_{H^{-1}_\kappa(\T)} ,
\label{eq:hskappa identity 4}
\end{align}
both uniformly for $\kappa \geq 1$.  Here $\psi\in C_c^\infty(\R)$ is a fixed function so that $\sum_{k\in\Z} \psi(x-k) \equiv 1$.  This guarantees that we have the duality relation
\begin{equation}
\norm{ h }_{H^1(\T)} = \sup \left\{ \int_{\R} h(x)f(x)\psi(x)\dx : f\in C^\infty(\T),\ \norm{f}_{H^{-1}(\T)} \leq 1 \right\} .
\label{eq:hskappa identity 5}
\end{equation}
Here and throughout this section, we are viewing functions on the circle $\T$ as functions on the line $\R$ by periodic extension.

First, we obtain \ti{a priori} estimates for the $H_\kappa$ flow:
\begin{lem}
	\label{thm:periodic a priori}
	Given an integer $s \geq 0$ and $A, T>0$, there exist constants $C$ and $\kappa_{0}$ such that solutions $V_\kappa(t)$ to the $H_\kappa$ flow~\eqref{eq:hk flow} obey
	\begin{equation*}
	\norm{ V(0) }_{H^{s}(\T)} \leq A \quad \implies \quad
	\norm{V_\kappa(t)}_{H^s(\T)} \leq C\quad\tx{for all }|t| \leq T\tx{ and } \kappa \geq \kappa_{0} .
	\end{equation*}
\end{lem}
\begin{proof}
	The Hamiltonian $H_\kappa$ is constructed from $\alpha(\kappa,q)$ and the momentum functional~\eqref{eq:momentum hkdv}.  Momentum is one of the polynomial conserved quantities of KdV and $\alpha$ can be expressed as a series~\eqref{eq:alpha intro} in terms of these quantities, and so both Poisson commute with every KdV conserved quantity.  Consequently, each KdV conserved quantity is also conserved for smooth solutions $V_\kappa(t)$ of the $H_\kappa$ (which can be individually verified using the algebraic identities~\eqref{eq:g alg prop 1}--\eqref{eq:g alg prop 3}).  Therefore the classical energy arguments for KdV (cf.~\cite{Lax1975}*{Th.~3.1}) can be applied to the $H_\kappa$ flow.
\end{proof}

Next, we prove the existence for the $H_\kappa$ flow via a contraction mapping argument:
\begin{prop}
	\label{thm:periodic hk gwp}
	Given $A,T>0$, there exists a constant $\kappa_0$ so that for $\kappa \geq \kappa_0$ the $H_\kappa$ flows~\eqref{eq:hk flow} with initial data in the closed ball $B_A\subset H^5(\T)$ of radius $A$ are globally well-posed and the corresponding solutions $V_\kappa(t)$ are in $C_tH^5([-T,T]\times\T)$.
\end{prop}
\begin{proof}
	The solution $V_\kappa(t)$ to the $H_\kappa$ flow satisfies the integral equation
	\begin{equation}
	V_\kappa(t) = e^{t4\kappa^2\partial_x}V(0) + 16\kappa^5 \int_0^t e^{(t-s)4\kappa^2\partial_x} g'(\kappa,V_\kappa(s)) \ds .
	\label{eq:periodic wp 1}
	\end{equation}
	We will ultimately show that if $W,\tilde{W} \in H^4(\T)$ then
	\begin{equation}
	\snorm{ g'(\kappa,W) - g'(\kappa,\tilde{W}) }_{H^5(\T)}
	\lesssim \snorm{ g(\kappa,W) - g(\kappa,\tilde{W}) }_{H^6(\T)}
	\lesssim \snorm{ W - \tilde{W} }_{H^4(\T)}
	\label{eq:periodic wp 2}
	\end{equation}
	uniformly for $\kappa \geq 2\norm{W}^2_{H^{-1}(\T)}, 2\snorm{\tilde{W}}^2_{H^{-1}(\T)}$.  Assuming this claim, for fixed initial data $V(0) \in H^5(\T)$ we see that $W \mapsto g'(\kappa,W)$ is Lipschitz on the closed ball $B_R \subset H^5(\T)$ of radius $R:= 2A$ for all $\kappa \geq 2R^2$.  Consequently, there exists $\eps>0$ sufficiently small such that the integral operator~\eqref{eq:periodic wp 1} is a contraction on $C_tB_{R}([-\eps,\eps]\times\R)$.  Then, given an arbitrary $T>0$, we use the \ti{a priori} estimates of \cref{thm:periodic a priori} to increase $R := R( A )$ if necessary and iterate in order to conclude that the solution exists in $C_tB_R([-T,T]\times\R)$ for all $\kappa \geq 2R^2$.
	
	It remains to prove the Lipschitz estimate~\eqref{eq:periodic wp 2}, but first we must show that $g(\kappa,W)-\tfrac{1}{2\kappa}$ is in $H^6(\T)$ for $W\in H^4(\T)$.  To accomplish this, we will show that $[g(\kappa,W)-\tfrac{1}{2\kappa}]^{(s)}$ is in $H^1$ for $s=0,1,\dots,5$ using the duality relation~\eqref{eq:hskappa identity 5}.  For $f\in C^\infty(\T)$ we can obtain a series for $g^{(s)}(\kappa,W)$ by differentiating the translation relation~\eqref{eq:g trans prop} at $h=0$:
	\begin{equation*}
	\left| \int [g(x;\kappa,W)-\tfrac{1}{2\kappa}]^{(s)} f(x)\psi(x) \dx \right|
	\leq \sum_{\ell=1}^\infty \left| \tr\{ f\psi [\partial^s, R_0 (WR_0)^\ell ] \} \right| .
	\end{equation*}
	Using the operator estimates~\eqref{eq:hskappa identity 3} and~\eqref{eq:hskappa identity 4}, we put all copies of $W$ in $H^{-1}(\T)$:
	\begin{equation*}
	\left| \int [g(\kappa,W)-\tfrac{1}{2\kappa}]^{(s)} f\psi \dx \right|
	\lesssim \kappa^{-1/2}\norm{f}_{H^{-1}} \sum_{\ell=1}^\infty \sum_{\substack{ \sigma\in\N^\ell \\ |\sigma| = s }} \binom{s}{\sigma} \prod_{j=1}^\ell \kappa^{-1/2} \snorm{ W^{(\sigma_j)} }_{H^{-1}} .
	\end{equation*}
	Applying H\"older's inequality in Fourier variables we see that
	\begin{equation*}
	\prod_{j=1}^\ell \snorm{ W^{(\sigma_j)} }_{H^{-1}(\T)} 
	\leq \snorm{ W^{(s)} }_{H^{-1}(\T)} \snorm{ W }_{H^{-1}(\T)}^{\ell-1} ,
	\end{equation*}
	and so
	\begin{align*}
	\left| \int [g(\kappa,W)-\tfrac{1}{2\kappa}]^{(s)} f\psi \dx \right|
	&\lesssim \kappa^{-1}\norm{f}_{H^{-1}} \snorm{W^{(s)}}_{H^{-1}} \sum_{\ell=1}^\infty \ell^s \big( \kappa^{-1/2} \snorm{ W }_{H^{-1}} \big)^{\ell-1} \\
	&\lesssim \kappa^{-1}\norm{f}_{H^{-1}} \norm{W}_{H^{s-1}}
	\end{align*}
	provided that we have $\kappa \geq 2\snorm{ W }^2_{H^{-1}}$.  Taking a supremum over $\norm{f}_{H^{-1}(\T)} \leq 1$ yields the claim.
	
	Lastly, we turn to the Lipschitz inequality~\eqref{eq:periodic wp 2}.  It suffices to show that the linear functional $h\mapsto dg|_W(h) - dg|_0(h)$ is bounded $H^4(\T)\to H^6(\T)$ for $\kappa \geq 2\norm{W}^2_{H^{-1}(\T)}$ by the fundamental theorem of calculus.  To demonstrate this, we estimate its $s$th derivative in $H^1(\T)$ for $s=0,\dots,5$ using the duality relation~\eqref{eq:hskappa identity 5} and the previous argument.  Expanding the resolvents within the functional derivative expression~\eqref{eq:cleanup lem 10} into series, we have
	\begin{align*}
	&\left| \int [dg|_W(h) - dg|_0(h)]^{(s)}(x) f(x)\psi(x) \dx \right| 
	\lesssim \kappa^{-3/2} \norm{f}_{H^{-1}} \snorm{W}_{H^{s-1}} \snorm{h}_{H^{s-1}} 
	\end{align*}
	for $\kappa \geq 2\snorm{ W }^2_{H^{-1}}$.  Taking a supremum over $\norm{f}_{H^{-1}(\T)} \leq 1$ yields the claim.
\end{proof}

We now know that the $H_\kappa$ flows $V_\kappa(t)$ satisfy the second condition in the definition of admissibility, provided that we take initial data $V(0) \in H^5(\T)$.  The first condition---that the corresponding solution $V(t)$ of KdV is sufficiently regular---then follows from the well-posedness of KdV in $H^3(\T)$.  Alternatively, we could reprove this classical well-posedness result by constructing $V(t)$ as the limit of the $H_\kappa$ flows $V_\kappa$, but we will not pursue this here.

Our next objective is to verify the third condition in the definition of admissibility, which says that $V_\kappa$ converges to $V$ in $W^{2,\infty}(\R)$ as $\kappa\to\infty$.  To begin, we control the growth of the difference $V_\kappa-V$ in $L^2(\T)$:
\begin{prop}
	\label{thm:periodic L2 conv}
	Given $A,T>0$, there exists a constant $C$ so that the quantity
	\begin{equation*}
	P(t) := \tfrac{1}{2} \int_{\T} [V_\kappa(t,x) - V(t,x)]^2 \dx \quad\tx{ with }V_\kappa(0,x) = V(0,x) \in H^5(\T)
	\end{equation*}
	obeys
	\begin{equation*}
	\left| \ddt P(t) \right|
	\leq C \big( P + o(1) \sqrt{P}\big) \quad\tx{as }\kappa\to\infty
	\end{equation*}
	uniformly for $|t|\leq T$ and $\norm{ V(0) }_{H^5(\T)} \leq A$.
\end{prop}
\begin{proof}
	Let $u := V_\kappa - V$ so that $P(t) = \tfrac{1}{2} \norm{u}^2_{L^2(\T)}$.  Then $u$ obeys the differential equation
	\begin{align*}
	\ddt u
	&= 16\kappa^5 g'(\kappa,V_\kappa) + 4\kappa^2V'_\kappa + V''' - 6VV' \\
	&= 16\kappa^5 g'(\kappa,V_\kappa) + 4\kappa^2V'_\kappa + V''' - 6V_\kappa V_\kappa' + 6(V_\kappa u)' - 6uu' .
	\end{align*}
	Multiplying by $u \in C^\infty(\T)$ and integrating over $\T$, we obtain an equality for the time derivative of $P(t)$.  The contribution from $6uu'$ is a total derivative and hence vanishes.  Expanding $g'(\kappa,V_\kappa)$ in a series and extracting the linear and quadratic terms, we write
	\begin{align}
	&\ddt P(t)
	\nonumber \\
	&= 6 \int u(x) (V_\kappa u)'(x) \dx 
	\label{eq:periodic L2 conv 1} \\
	&\phantom{={}}+ \int u(x) \{ - 16\kappa^5 \langle\del_x, R_0V'_\kappa R_0\del_x \rangle + 4\kappa^2 V'_\kappa(x) + V'''(x) \} \dx
	\label{eq:periodic L2 conv 2} \\
	&\phantom{={}}+ \int u(x) \{ 16\kappa^5 \langle\del_x, [\partial, R_0V_\kappa R_0V_\kappa R_0 ] \del_x \rangle - 3(V_\kappa^2)'(x) \} \dx
	\label{eq:periodic L2 conv 3} \\
	&\phantom{={}}+ \int u(x)\, 16\kappa^5 \sum_{\ell=3}^\infty (-1)^\ell \langle\del_x, [ \partial, R_0(V_\kappa R_0)^\ell ] \del_x \rangle \dx .
	\label{eq:periodic L2 conv 4}
	\end{align}
	We will estimate the four terms~\eqref{eq:periodic L2 conv 1}--\eqref{eq:periodic L2 conv 4} individually.
	
	For the first term~\eqref{eq:periodic L2 conv 1}, we integrate by parts to move the derivative onto $V_\kappa$:
	\begin{equation*}
	|\eqref{eq:periodic L2 conv 1}|
	= \left| \int 3V'_\kappa(x) u^2(x) \dx \right|
	\leq 3 \norm{V'_\kappa}_{L^\infty} P .
	\end{equation*}
	We know $\norm{V'_\kappa}_{L^\infty}$ is bounded uniformly for $|t|\leq T$ and $\kappa$ large by the embedding $H^1(\T) \hookrightarrow L^\infty(\T)$ and the \ti{a priori} estimates of \cref{thm:periodic a priori}.
	
	Next we estimate the linear term~\eqref{eq:periodic L2 conv 2}.  Using the first operator identity of~\eqref{eq:linear op id}, we have
	\begin{align*}
	\eqref{eq:periodic L2 conv 2}
	= &\int u \big\{ \big[ {-16}\kappa^4 R_0(2\kappa) + 4\kappa^2 + \partial^2 \big] V' \big\} \dx \\
	&+ \int u \big\{ \big[ {- 16}\kappa^4 R_0(2\kappa) + 4\kappa^2 \big] u' \big\} \dx .
	\end{align*}
	As differentiation commutes with the resolvent $R_0(2\kappa)$, the last integrand is a total derivative and the integral vanishes.  For the remaining term, we use the rest of the identity~\eqref{eq:linear op id} and Cauchy--Schwarz to estimate
	\begin{equation*}
	|\eqref{eq:periodic L2 conv 2}|
	\leq \snorm{ R_0(2\kappa) V^{(5)} }_{L^2} P^{1/2}
	\leq \kappa^{-2} \snorm{ V^{(5)} }_{L^2} P^{1/2} .
	\end{equation*}
	The factor $\snorm{ V^{(5)} }_{L^2}$ is bounded uniformly for $|t|\leq T$ and $\kappa$ large by the \ti{a priori} estimates of \cref{thm:periodic a priori}.
	
	Now we examine to the quadratic contribution~\eqref{eq:periodic L2 conv 3}.  Consider the term when the derivative $[\partial,\cdot]$ hits the second factor of $V_\kappa$, and expand in Fourier variables:
	\begin{align*}
	\int_{\T} u(x) \langle \del_x , R_0 V_\kappa R_0 V'_\kappa R_0 \del_x \rangle \dx
	= \sum_{\xi_1,\xi_2,\xi_3\in 2\pi\Z} \frac{ \hat{u}(\xi_1-\xi_3) \wh{V_\kappa} (\xi_3-\xi_2) \wh{V'_\kappa}(\xi_2-\xi_1) }{ (\xi_3^2+\kappa^2) (\xi_2^2+\kappa^2) (\xi_1^2+\kappa^2) } .
	\end{align*}
	Re-indexing $\eta_1 = \xi_2-\xi_1$, $\eta_2 = \xi_3-\xi_2$, $\eta_3 = \xi_3$, the RHS becomes
	\begin{equation*}
	\sum_{\eta_1,\eta_2,\eta_3\in 2\pi\Z} \frac{ \hat{u}(-\eta_1-\eta_2) \wh{V_\kappa} (\eta_2) \wh{V'_\kappa}(\eta_1)}{ (\eta_3^2+\kappa^2) ((\eta_3-\eta_2)^2+\kappa^2) ((\eta_3-\eta_1-\eta_2)^2+\kappa^2) } .
	\end{equation*}
	The numerator is now independent of $\eta_3$, and so if we approximate the sum over $\eta_3\in 2\pi\Z$ by an integral over $\eta_3\in\R$ then we can evaluate the integral using residue calculus and eliminate $\eta_3$: 
	\begin{align*}
	&\sum_{\eta_1,\eta_2\in 2\pi\Z} \frac{1}{2\pi} \int_{\R} \frac{ \hat{u}(-\eta_1-\eta_2) \wh{V_\kappa} (\eta_2) \wh{V'_\kappa}(\eta_1) }{ (\eta_3^2+\kappa^2) ((\eta_3-\eta_2)^2+\kappa^2) ((\eta_3-\eta_1-\eta_2)^2+\kappa^2) } \deta_3 \\
	&= \kappa^{-1} \sum_{\eta_1,\eta_2\in 2\pi\Z} \frac{ \hat{u}(-\eta_1-\eta_2) \wh{V_\kappa} (\eta_2) \wh{V'_\kappa}(\eta_1) (12\kappa^2 + \eta_1^2 + \eta_1\eta_2 + \eta_2^2)}{ (\eta_1^2+4\kappa^2) (\eta_2^2+4\kappa^2) ((\eta_1+\eta_2)^2+4\kappa^2) } .
	\end{align*}
	Note that this last summand is symmetric in $\eta_1$ and $\eta_2$, and so both terms of $[\partial , R_0V_\kappa R_0 V_\kappa R_0]$ produce the same contribution.
	
	We are now prepared to estimate the term~\eqref{eq:periodic L2 conv 3}.  Changing to Fourier variables and replacing the sum over $\eta_3\in 2\pi\Z$ with an integral over $\eta_3\in\R$, we write
	\begin{align}
	&\eqref{eq:periodic L2 conv 3} 
	\nonumber \\
	&= \sum_{\eta_1,\eta_2} \hat{u}(-\eta_1-\eta_2) \wh{V_\kappa}(\eta_1) \wh{V'_\kappa}(\eta_2) \left[ \tfrac{32\kappa^4( 12\kappa^2+\eta_1^2 + \eta_1\eta_2 + \eta_2^2 )}{ (\eta_1^2+4\kappa^2) (\eta_2^2+4\kappa^2) ((\eta_1+\eta_2)^2+4\kappa^2) } - 6 \right] 
	\label{eq:periodic L2 conv 5}\\
	&\phantom{={}}+ \sum_{\eta_1,\eta_2} \left[ \sum_{\eta_3} F(\eta_1,\eta_2,\eta_3) - \frac{1}{2\pi} \int_{\R} F(\eta_1,\eta_2,\eta_3) \deta_3 \right] .
	\label{eq:periodic L2 conv 6}
	\end{align}
	Here, all summations are over $2\pi\Z$ and the integrand is given by
	\begin{equation*}
	F(\eta_1,\eta_2,\eta_3)
	:= \frac{ 16\kappa^5 \hat{u}(-\eta_1-\eta_2) \big[ \wh{V'_\kappa} (\eta_2) \wh{V_\kappa}(\eta_1) + \wh{V_\kappa} (\eta_2) \wh{V'_\kappa}(\eta_1) \big] }{ (\eta_3^2+\kappa^2) ((\eta_3-\eta_2)^2+\kappa^2) ((\eta_3-\eta_1-\eta_2)^2+\kappa^2) } .
	\end{equation*}
	
	The upshot of our manipulation is that in~\eqref{eq:periodic L2 conv 5} the $O(1)$ term as $\kappa\to\infty$ cancels out, and we are left with
	\begin{equation*}
	\left| \frac{32\kappa^4( 12\kappa^2+\eta_1^2 + \eta_1\eta_2 + \eta_2^2 )}{ (\eta_1^2+4\kappa^2) (\eta_2^2+4\kappa^2) ((\eta_1+\eta_2)^2+4\kappa^2) } - 6 \right|
	\lesssim \frac{\eta_1^2 + \eta_2^2}{\kappa^2} .
	\end{equation*}
	Absorbing $\eta_1^2$ and $\eta_2^2$ as derivatives on $V_\kappa$, we put $u$ and the copy of $V_\kappa$ with the most derivatives in $\ell^2$ and estimate
	\begin{align*}
	| \eqref{eq:periodic L2 conv 5} |
	&\lesssim \kappa^{-2} \sup_{i,j\in\{0,2\}} \sum_{\eta_1,\eta_2} \left| \hat{u}(-\eta_1-\eta_2) \wh{V_\kappa^{(i)}}(\eta_1) \wh{V^{(j+1)}_\kappa}(\eta_2) \right| \\
	&\leq \kappa^{-2} \sup_{i,j\in\{0,2\}} \norm{ u }_{L^2(\T)} \snorm{ V_\kappa^{(j+1)} }_{L^2(\T)} \sum_{\eta_2} |\wh{V_\kappa^{(i)}}(\eta_2)| \\
	&\lesssim \kappa^{-2} \norm{V_\kappa}_{H^3(\T)}^2 P^{1/2} .
	\end{align*}
	In the last inequality, we used Cauchy--Schwarz to estimate
	\begin{equation*}
	\sum_{\eta_2} |\wh{V_\kappa}(\eta_2)|
	\leq \bigg( \sum_{\eta_2} ( 1+\eta_2^2 )^{-1} \bigg)^{\frac{1}{2}} \bigg( \sum_{\eta_2}  ( 1+\eta_2^2 ) |\wh{V_\kappa}(\eta_2)|^2 \bigg)^{\frac{1}{2}}
	\lesssim \norm{V_\kappa}_{H^1(\T)} .
	\end{equation*}
	
	Next, we must check that the remainder~\eqref{eq:periodic L2 conv 6} yields an acceptable contribution, which is due to the smoothness of the integrand $F$.  First, we will bound the trapezoid rule error term
	\begin{equation*}
	E_n(h) := \tfrac{1}{2} h [ F(\eta_1,\eta_2,a_n) + F(\eta_1,\eta_2,a_n+h) ] - \int_{a_n}^{a_n+h} F(\eta_1,\eta_2,\eta_3)\deta_3
	\end{equation*}
	for arbitrary $a_n\in\R$, $n\in\Z$.  It is easily checked that $E_n(0) = E'_n(0) = 0$, and so
	\begin{align*}
	|E_n(h)|
	= \left| \int_0^h \int_0^t E''_n(s) \ds\dt \right|
	&= \left| \int_0^h \int_0^t \tfrac{1}{2}s\,\partial^2_{\eta_3} F(\eta_1,\eta_2,a_n+s)\ds\dt \right|  \\
	&\leq \tfrac{1}{12} h^3 \norm{ \partial^2_{\eta_3} F }_{L^\infty_{\eta_3}([a_n,a_n+h])} .
	\end{align*}
	Therefore, setting $a_n = 2\pi n$ and $h=2\pi$ we have
	\begin{align*}
	&\left| \pi [ F(\eta_1,\eta_2,2\pi n) + F(\eta_1,\eta_2,2\pi(n+1)) ] - \int_{2\pi n}^{2\pi(n+1)} F(\eta_1,\eta_2,\eta_3)\deta_3 \right| \\
	&\leq \tfrac{1}{12} (2\pi)^3 \norm{ \partial^2_{\eta_3} F }_{L^\infty_{\eta_3}([2\pi n,2\pi(n+1)])}
	\end{align*}
	for all $n\in\Z$.  Each $\eta_3$ derivative applied to $F$ introduces one order of decay in $|\eta_3|$.  The term $(\eta_3^2+\kappa^2)^{-1}$ is bounded by the summable sequence $(n^2+1)^{-1}$, and every other order of decay in $|\eta_3|$ yields a factor of $\kappa^{-1}$.  Altogether we estimate
	\begin{equation*}
	\norm{ \partial^2_{\eta_3} F }_{L^\infty_{\eta_3}([2\pi n,2\pi(n+1)])}
	\lesssim \frac{\big| \hat{u}(-\eta_1-\eta_2) \big[ \wh{V'_\kappa} (\eta_2) \wh{V_\kappa}(\eta_1) + \wh{V_\kappa} (\eta_2) \wh{V'_\kappa}(\eta_1) \big] \big| }{ \kappa (n^2+1) } .
	\end{equation*}
	Summing over $n\in\Z$, the trapezoid rule error estimate yields
	\begin{align*}
	|\eqref{eq:periodic L2 conv 6}|
	&\lesssim \kappa^{-1} \sum_{\eta_1,\eta_2 \in 2\pi\Z} \left| \hat{u}(-\eta_1-\eta_2) \big[ \wh{V'_\kappa} (\eta_2) \wh{V_\kappa}(\eta_1) + \wh{V_\kappa} (\eta_2) \wh{V'_\kappa}(\eta_1) \big] \right| \\
	&\lesssim \kappa^{-1} \norm{V_\kappa}^2_{H^1} P^{1/2} .
	\end{align*}
	This is as acceptable contribution, and thus concludes the estimate of the quadratic term~\eqref{eq:periodic L2 conv 3}.
	
	Lastly we estimate the last term~\eqref{eq:periodic L2 conv 4}.  Using the operator estimates~\eqref{eq:hskappa identity 3} and~\eqref{eq:hskappa identity 4}, we have
	\begin{align*}
	|\eqref{eq:periodic L2 conv 4}|
	&\leq \kappa^5 \sum_{\ell=3}^\infty \left| \tr\{ u [\partial, R_0(V_\kappa R_0)^\ell ] \} \right| \\
	&\lesssim \kappa^4 \norm{ u }_{H^{-1}_\kappa(\T)} \norm{ V'_\kappa }_{H^{-1}_\kappa(\T)} \sum_{\ell=3}^\infty \big( \kappa^{-1/2}\norm{ V_\kappa }_{H^{-1}_\kappa(\T)} \big)^{\ell-1} \\
	&\lesssim \kappa^{-1} \norm{V_\kappa}_{L^2(\T)}^2 \norm{V'_\kappa}_{L^2(\T)} P^{1/2} .
	\end{align*}
	This concludes the estimate of $\ddt P(t)$ and hence the proof of \cref{thm:periodic L2 conv}.
\end{proof}

We are now prepared to prove \cref{thm:intro periodic}:
\begin{cor}
	\label{thm:periodic gwp}
	Given a background wave $V(0) \in H^5(\T)$, the KdV equation~\eqref{eq:kdv} with initial data $u(0) \in V(0) + H^{-1}(\R)$ is globally well-posed.
\end{cor}
\begin{proof}
	In view of \cref{thm:wellposed 2} it suffices to check that $V$ satisfies the three conditions of \cref{thm:hyp}.  Conditions (i) and (ii) are satisfied by the embedding $H^1 \hookrightarrow L^\infty$ and the \ti{a priori} estimates of \cref{thm:periodic a priori}, and so it only remains to verify condition (iii).
	
	Fix $T>0$.  By the embedding $H^1 \hookrightarrow L^\infty$, it suffices to show that $V_\kappa - V$ converges to zero in $C_tH^3([-T,T]\times \T)$ uniformly for initial data in the fixed set $\{ V_\varkappa(t) : |t|\leq T,\ \varkappa\geq \kappa_0 \}$.  By \cref{thm:periodic hk gwp}, we may pick the constant $\kappa_0$ so that the $H_\varkappa$ flows $\{ V_\varkappa(t) : |t|\leq T,\ \varkappa\geq \kappa_0 \}$ are contained in a ball $B_A \subset H^5(\T)$ of radius $A>0$, and so that the $H_\kappa$ flows are well-posed on $B_A$ for $\kappa\geq\kappa_0$.  From \cref{thm:periodic L2 conv} and the observation $o(1)\sqrt{P} \leq P + o(1)$, we have
	\begin{equation*}
	\left| \ddt P(t) \right| \leq C ( P + o(1) )
	\quad\tx{as }\kappa\to\infty
	\end{equation*}
	uniformly for $|t|\leq T$ and initial data in $B_A$.  Gr\"onwall's inequality then yields
	\begin{equation*}
	\tfrac{1}{2} \norm{ V_\kappa - V }_{L^2(\T)}^2 = P(t) \leq e^{CT} P(0) + o(1) ( e^{CT} - 1 )
	\end{equation*}
	uniformly for $|t|\leq T$ and initial data in $B_A$.  As $P(0) = 0$ by definition, we conclude
	\begin{equation*}
	\norm{ V_\kappa - V }_{C_tL^2([-T,T]\times\T)} \to 0
	\quad\tx{as }\kappa\to\infty
	\end{equation*}
	uniformly for initial data in $B_A$.
	
	Using H\"older's inequality in Fourier variables, we have
	\begin{equation*}
	\norm{ f }_{H^3(\T)} \leq \norm{f}^{2/5}_{L^2(\T)} \norm{f}^{3/5}_{H^5(\T)} .
	\end{equation*}
	By the \ti{a priori} estimates of \cref{thm:periodic a priori}, $V_\kappa$ is bounded in $C_tH^5([-T,T]\times\T)$ uniformly for $\kappa$ large and initial data in $B_A$.  Therefore, applying the above inequality to $V_\kappa - V$, we conclude that $V_\kappa-V\to 0$ in $C_tH^3([-T,T]\times \T)$ uniformly for initial data in the smaller set $\{ V_\varkappa(t) : |t|\leq T,\ \varkappa\geq \kappa_0 \} \subset B_A$.
\end{proof}

\bibliography{kdv_gwp_nonzero_asymp}

\end{document}